\documentclass{article}
\usepackage[utf8]{inputenc}

\title{On topological groupoids that represent theories}
\author{J. L. Wrigley\footnote{School of Mathematical Sciences, Queen Mary University London, London, E1 4NS UK, email: \texttt{j.wrigley@qmul.ac.uk}}}

\usepackage[only,llbracket,rrbracket,shortleftarrow,shortrightarrow,mapsfrom]{stmaryrd}

\usepackage[dvipsnames]{xcolor}
\definecolor{Forest}{rgb}{0.13, 0.55, 0.13}
\definecolor{Darkcyan}{rgb}{0.0, 0.55, 0.55}

\usepackage[pdfencoding=auto,pdfusetitle]{hyperref}
\hypersetup{colorlinks=true, linkcolor={black}, citecolor={black}, urlcolor={black}}

\usepackage{amssymb}
\usepackage{amsmath}
\usepackage{amsthm}
\usepackage{tikz-cd}
\usepackage{fancyhdr}
\usepackage{mathtools}
\usepackage{enumitem}
\usepackage{extpfeil}
\usepackage{bussproofs}
\usepackage{mathtools}
\usepackage{scalerel}
\usepackage{booktabs}

\usepackage[noabbrev,capitalise]{cleveref}
\crefformat{equation}{(#2#1#3)}
\crefformat{enumi}{#2#1#3}
\crefformat{enumii}{#2#1#3}

\RequirePackage{tikz-cd}
\RequirePackage{amssymb}
\usetikzlibrary{calc}
\usetikzlibrary{decorations.pathmorphing}

\tikzset{curve/.style={settings={#1},to path={(\tikztostart)
			.. controls ($(\tikztostart)!\pv{pos}!(\tikztotarget)!\pv{height}!270:(\tikztotarget)$)
			and ($(\tikztostart)!1-\pv{pos}!(\tikztotarget)!\pv{height}!270:(\tikztotarget)$)
			.. (\tikztotarget)\tikztonodes}},
	settings/.code={\tikzset{quiver/.cd,#1}
		\def\pv##1{\pgfkeysvalueof{/tikz/quiver/##1}}},
	quiver/.cd,pos/.initial=0.35,height/.initial=0}

\tikzset{tail reversed/.code={\pgfsetarrowsstart{tikzcd to}}}
\tikzset{2tail/.code={\pgfsetarrowsstart{Implies[reversed]}}}
\tikzset{2tail reversed/.code={\pgfsetarrowsstart{Implies}}}
\tikzset{no body/.style={/tikz/dash pattern=on 0 off 1mm}}

\tikzset{
	labl/.style={anchor=south, rotate=90, inner sep=.5mm}
}
\setlist{listparindent = \parindent, parsep=0pt,}
\setenumerate[1]{label = (\roman*), ref = (\roman*)}

\theoremstyle{plain}
\newtheorem{thm}{Theorem}[section]
\newtheorem{lem}[thm]{Lemma}
\newtheorem{coro}[thm]{Corollary}
\newtheorem{prop}[thm]{Proposition}

\theoremstyle{definition}
\newtheorem{df}[thm]{Definition}
\newtheorem{dfs}[thm]{Definitions}
\newtheorem{rem}[thm]{Remark}
\newtheorem{rems}[thm]{Remarks}
\newtheorem{ex}[thm]{Example}
\newtheorem{exs}[thm]{Examples}

\renewcommand{\phi}{\varphi}

\newcommand{\op}{^{\rm op}}
\newcommand{\id}{{\rm{id}}}
\newcommand{\1}{{\bf 1}}
\newcommand{\Hom}{{\rm{Hom}}}

\newcommand{\lrset}[2]{\left\{\,{#1}\,\middle\vert\,{#2}\,\right\}}
\newcommand{\lrangle}[1]{\left \langle {#1} \right \rangle}

\newcommand{\2}{{\bf 2}}
\newcommand{\N}{\mathbb{N}}
\newcommand{\Q}{\mathbb{Q}}
\newcommand{\Z}{\mathbb{Z}}
\newcommand{\R}{\mathbb{R}}

\newcommand{\theory}{\mathbb{T}}

\newcommand{\topos}{\mathcal{E}}
\newcommand{\ftopos}{\mathcal{F}}

\newcommand{\sets}{{\bf Sets}}

\newcommand{\Top}{{\bf Top}}
\newcommand{\Topos}{{\bf Topos}}

\newcommand{\Sh}{{\bf Sh}}
\newcommand{\opens}{\mathcal{O}}

\newcommand{\Sub}{{\rm{Sub}}}
\newcommand{\Geom}{{\bf Geom}}

\newcommand{\rmcan}{\text{\rm can}}

\newcommand{\B}{{\bf B}}
\newcommand{\BG}{{\bf B}G}
\newcommand{\G}{\mathbb{G}}
\newcommand{\X}{\mathbb{X}}

\newcommand{\Xdd}{\mathbb{X}_\delta^\delta}
\newcommand{\Xtd}{\mathbb{X}_{\tau_0}^\delta}
\newcommand{\Xtt}{\mathbb{X}_{\tau_0}^{\tau_1}}
\newcommand{\Xlog}{\mathbb{X}_{\taulogo}^{\tauloga}}

\newcommand{\FG}{\mathcal{FG}}
\newcommand{\BM}{\mathcal{BM}}

\newcommand{\Aut}{{\rm Aut}}
\newcommand{\Iso}{{\bf Iso}}

\newcommand{\form}[2]{\{\,\vec{{#2}} : {#1}\,\}}

\newcommand{\class}[1]{\llbracket\, {#1} \,\rrbracket}
\newcommand{\classv}[2]{\llbracket\, \vec{{#2}} : {{#1}} \,\rrbracket}
\newcommand{\lrclass}[1]{\left\llbracket\, {#1} \,\right\rrbracket}

\newcommand{\Tmodels}[1]{\mathbb{T}\text{-}{\bf Mod}({{#1}})}
\newcommand{\Dtheory}{\mathbb{D}}
\newcommand{\AItheory}{\mathbb{AI}}
\newcommand{\AIform}{\mathcal{AI}}

\newcommand{\Etheory}{\mathbb{E}}
\newcommand{\Th}{{\rm Th}}

\newcommand{\DLO}{{\mathbb{L}_\infty}}

\newcommand{\p}{{\bf p}}

\newcommand{\taulogo}{{\tau\text{-}{\rm log}_0}}
\newcommand{\tauloga}{{\tau\text{-}{\rm log}_1}}

\renewcommand{\log}{{\rm log}}

\newcommand{\Index}{\mathfrak{K}}
\newcommand{\Indexset}{\mathfrak{K}}

\newcommand{\tp}{{\rm tp}}

\newcommand{\paronto}{%
	\rightharpoondown\mathrel{\mspace{-15mu}}\rightharpoondown
}

\newcommand{\Frac}{{\rm Frac}}

\usepackage{geometry}
\begin{document}
	
	\maketitle
	
		\begin{abstract}
		Grothendieck toposes, and by extension, logical theories, can be represented by topological structures.  Butz and Moerdijk showed that every topos with enough points can be represented as the topos of sheaves on an open topological groupoid.  This paper tackles a follow-up question: we characterise, in model-theoretic terms, which open topological groupoids can represent the classifying topos of a theory.  Intuitively, this characterises which groupoids of models contain enough information to reconstruct the theory.  Our treatment subsumes many of the previous approaches found in the literature, such as that of Awodey, Forssell, Butz and Moerdijk.
	\end{abstract}
	
%%%%%%%%%%%%%%%%%%%%%%%%%%%%%%%%%%%%%%%%%%%%%%%%%%%%%%%%%%%%

%\setcounter{tocdepth}{1}
\tableofcontents

%%%%%%%%%%%%%%%%%%%%%%%%%%%%%%%%%%%%%%%%%%%
%%%%%%%%%%%%%%%%%%%%%%%%%%%%%%%%%%%%%%%

\section{Introduction}

\paragraph{Representing theories with groupoids.}

For some theories, there exist models whose automorphism structure is sufficiently rich that the data of the theory ought to be recoverable from the group of automorphisms.  For example, the rationals with the usual ordering is a conservative, ultrahomogeneous model for the theory of dense linear orders without endpoints.  Only complete and atomic theories can be represented by the automorphism group of a single model.  To represent an arbitrary theory, we must make the leap to groupoids of models.

Grothendieck topos theory (hereafter, just topos theory) is a natural language in which to formalise the problem of representing theories using topological categories since the discipline sits at the intersection of categorical logic and topological algebra.

On the topological side, a topos can be associated to topological/algebraic data, e.g.\ each topological group yields a topos of continuous actions on discrete sets.  In \cite{BM}, Butz and Moerdijk show that a topos has enough points if and only if it is the topos of sheaves on an open topological groupoid (in practice, most toposes encountered in the wild have enough points).  In this way, a topos can be viewed as a `generalised space' in which points possess non-trivial isomorphisms.

Meanwhile, on the logical side, every theory of geometric first-order logic possesses a \emph{classifying topos}, a topos that embodies the ``essence'' of the theory.  If a topos can be likened to a `generalised space' in the above sense, then the classifying topos of a theory is the `generalised space' whose points are models of the theory.

We will say that a theory is \emph{represented} by a (open) topological groupoid if the topos of sheaves on the groupoid classifies the theory.  The result of Butz and Moerdijk expresses that a geometric theory is represented by an open topological groupoid if and only if the set-based models are a conservative class of models (see \cite[\S 1.4.5]{TST}).  

%%%%%%%%%%%%%%%%%%%%%%%%%%%%%%%%%%
\paragraph{Our Contribution.}

This paper concerns itself with the next obvious question: which open topological groupoids represent a given theory?  Informally, this question is equivalent to asking: which groupoids of models have enough information to recover the theory?  The main result of this paper (\cref{maintheorem}) is a model-theoretic characterisation of which groupoids of models can be endowed with topologies to yield representing open topological groupoids.

We will observe that, just as with Caramello's \emph{topological Galois theory} of \cite{caramellogalois}, it does not suffice to have only a \emph{conservative} set of models for the theory, but instead model-theoretic conditions must be imposed on the groupoid as well.  The condition that must be imposed is a natural notion of the groupoid \emph{eliminating parameters}.  In addition to yielding novel applications, our characterisation subsumes the previous examples of representing groupoids found in the literature.

%%%%%%%%%%%%%%%%%%%%%%%%%%%%%%%%%%%%%%%%%%%

\paragraph{Review of the previous literature.}

The previous literature on using groupoids to represent toposes can be divided as to whether \emph{localic} (i.e.\ pointfree) or topological groupoids are used.  Both approaches have markedly different flavours.

By Caramello's {topological Galois theory} \cite{caramellogalois}, a complete and atomic theory is represented by the topological group of automorphisms of a model if and only if that model is ultrahomogeneous, i.e.\ any finite partial isomorphism of the model can be extended to a total automorphism.  This is in contrast to the \emph{localic Galois theory} developed by Joyal, Tierney and Dubuc (\cite[Theorem VIII.3.1]{JT} and \cite{dubuc}),  wherein it is shown that the classifying topos of a complete and atomic theory is equivalent to the topos of continuous actions by the {localic} automorphism group of \emph{any} model.  We take this as evidence that while the disciplines of categorical logic and classical model theory, in which ultrahomogeneous models play an important role, are normally viewed as entirely distinct, this is not the case when we prioritise a topological rather than localic viewpoint.

In a similar fashion, Blass and \v{S}\v{c}edrov characterise Boolean coherent toposes in \cite{blassscedrov} as those toposes that can be expressed as the coproduct of toposes of continuous actions by \emph{coherent} topological groups.  Moreover, these groups can be taken as the automorphism groups of ultrahomogeneous models for the theory classified by the topos.

Joyal and Tierney famously showed in \cite{JT} that every topos is the topos of sheaves on some open localic groupoid.  The parallel topological result was given in \cite{BM}, where Butz and Moerdijk show that every topos with enough points is represented by an open topological groupoid.  When a topos with enough points is known to classify a theory $\theory$, Forssell's thesis and subsequent papers with Awodey \cite{awodeyforssell,forssell,forssellphd} give an explicitly logical description of a representing open topological groupoid.  Namely, their results express that $\theory$ is recoverable from the groupoid of all $\Indexset$-\emph{indexed models}, for a sufficiently large cardinal $\Indexset$ (we shall call such groupoids \emph{Forssell groupoids}).

In summary, the relevant literature on localic and topological representing groupoids for toposes can be divided as follows.
{
	\renewcommand{\arraystretch}{1.2}
	\begin{center}
		\begin{tabular}{@{}l@{\hskip7pt}l@{\hskip7pt}l @{}}
			\toprule
			&Localic representation &Topological representation \\
			\midrule
			\emph{Connected atomic toposes} & localic groups \cite{JT,dubuc}, & topological groups \cite{caramellogalois}, \\
			\emph{Boolean coherent toposes} & & coproduct of coherent topological groups \cite{blassscedrov}, \\
			\emph{All toposes (with enough points)} & open localic groupoids \cite{JT}, & open topological groupoids \cite{BM,awodeyforssell,forssell,forssellphd}. \\
			\bottomrule			
		\end{tabular}
	\end{center}
}
Our characterisation of the representing open topological groupoids recovers the previous results for the right-hand column of the above table.

\paragraph{Overview.}  The paper proceeds as follows.

\begin{enumerate}[label = (\arabic*)]
	
	\item Some preliminaries are included in \cref{sec:prelims}.  In \cref{subsec:prelim:classtopos}, we recall the notion of a classifying topos for a geometric theory.  Throughout this paper, we will use a different site, rather than the standard \emph{syntactic site}, to present the classifying topos of a theory as this will ease our computations.  A description of this `alternative' site is given in \cref{subsec:prelim:classtopos}.  The definition of the topos of sheaves on a topological groupoid is recalled in in \cref{subsec:eqsheaves}.
	
	\item \cref{sec:thetheorem} is divided into three parts.  In the former two, \cref{subsec:index} and \cref{subsec:definables}, we define indexings of sets of models and an extension of the notion of definable subset of a model to groupoids of models.  This will allow us to express the statement of our classification theorem for representing open topological groupoids, completed in \cref{sec:thmstatement}.  \cref{sec:thmstatement} also includes a brief discussion of the relation between our result and the descent theory of Joyal and Tierney \cite{JT}.
	
	\item The proof of our classification result is contained in \crefrange{sec:bigpicture}{sec:mainproof}.  In \cref{sec:bigpicture}, we prove some properties about how toposes of sheaves on topological groupoids behave when we vary the topologies with which the space of arrows and the space of objects are endowed.  The method we will follow in \crefrange{sec:facttopobjs}{sec:mainproof} is laid out in \cref{subsec:method}.  Given a groupoid $\X = (X_1 \rightrightarrows X_0)$, the possible topologies with which $X_0$, respectively $X_1$, can be endowed that lead to a representing open topological groupoid are characterised in \cref{sec:facttopobjs}, resp., \cref{sec:facttopars}.  The final steps of the proof of the classification result are completed in \cref{sec:mainproof}.

	\item In \cref{sec:applications}, we present some applications of our characterisation, including a demonstration that the other logical treatments of representing open topological groupoids considered in the literature can be recovered via our classification result. 
	
	\begin{enumerate}[label = (\alph*)]
		\item In \cref{sec:atomic}, we recover the principle result of \cite{caramellogalois}, that is an atomic theory is represented by the automorphism group of a single model if and only if that model is conservative and ultrahomogeneous, as well as a characterisation of Boolean toposes with enough points reminiscent of \cite{blassscedrov}. 
		
		\item \cref{subsec:decidable} concerns the representing groupoids of decidable theories.
		
		\item In \cref{sec:etalecomplete} we show that every open representing model groupoid is Morita-equivalent to its \emph{étale completion}. 
		
		\item The representation results of Awodey, Butz, Forssell and Moerdijk \cite{awodeyforssell,BM,forssell}, including the case of Forssell groupoids, are recovered in \cref{sec:forssell}. 
		
		\item Having described representing groupoids for a given theory, in \cref{sec:theory_of_groupoid} we answer the converse problem by adapting the methods of \cite[Theorem 4.14]{hodges} to describe a theory represented by a given groupoid of indexed structures.
		
	\end{enumerate}
	
	\item Finally, in \cref{sec:algint} we give a worked example in further detail of a representing groupoid for the theory of algebraic integers.
	
\end{enumerate}

\section{Preliminaries}\label{sec:prelims}

We will recall some preliminary material on classifying toposes and toposes of equivariant sheaves necessary for understanding the results of this paper.  First, however, we mention a few conventions:
\begin{enumerate}
	\item As aforementioned, by topos we mean a Grothendieck topos;
	\item We will follow some standard abuses of notation in model theory and write $\vec{n} \in N$ to mean that $\vec{n}$ is a finite tuple of elements from a set (or model) $N$.
\end{enumerate}

\subsection{Classifying toposes}\label{subsec:prelim:classtopos}

A \emph{classifying topos} of a first order theory $\theory$ is a topos $\topos_\theory$ for which there is an equivalence
\begin{equation}\label{eq:classtopos}
	\Tmodels{\ftopos} \simeq \Geom(\ftopos,\topos_\theory)
\end{equation}
natural in $\ftopos$, where $\Tmodels{\ftopos}$ is the category of models of $\theory$ in the topos $\ftopos$ (see \cite[\S D1]{elephant} for how to construct models in an arbitrary topos) and $\Geom(\ftopos,\topos_\theory)$ is the category of \emph{geometric morphisms} from $\ftopos$ to $\topos_\theory$.  By the universal property \cref{eq:classtopos}, the classifying topos of a theory is unique up to equivalence.  Not only does every \emph{geometric theory} have a classifying topos (see \cite[Theorem 2.1.10]{TST}), but every topos is the classifying topos for some geometric theory (see \cite[Theorem 2.1.11]{TST}).  Intuitively, a classifying topos contains the essential information about a theory.  In particular, by the equivalence \cref{eq:classtopos}, a classifying topos contains the semantic data of the models of $\theory$, not only in $\sets$, but in any topos.

\paragraph{Geometric logic.}  The reader is directed to \cite[\S D1]{elephant} for a full account of geometric logic.  In short, it is the fragment of infinitary first order logic whose permissible symbols are finite conjunction $\land$ (the empty conjunct yielding `truth' $\top$), infinitary disjunction $\bigvee$ (the empty disjunct yielding `falsity' $\bot$), an equality predicate $=$ and existential quantification $\exists$.  A \emph{geometric theory} consists of a set of axioms of the form $\phi \vdash_{\vec{x}} \psi$, which has the interpretation ``for all $\vec{x}$, if $\phi(\vec{x})$ then $\psi(\vec{x})$''.

A theory of geometric logic is allowed to have function symbols in its underlying signature.  In contrast, many of the constructions we will see in this paper (e.g., in \cref{df:theory-of-indexed-groupoid}) use an exclusively relational signature.  This can be explained since, by replacing a function symbol with a relation symbol expressing its graph (\cite[Lemma D1.4.9]{elephant}), every geometric theory is equivalent to a theory without function symbols.  Similarly, while we will allow geometric theories to be multi-sorted, by combining the disparate sorts of a multi-sorted theory into one and keeping track of the original sorts by new relations (\cite[Lemma D1.4.13]{elephant}), every theory is equivalent to a single-sorted one, and so without loss of generality one is free to assume that theories are single-sorted.

\begin{rem}\label{rem:morleyization}
	Geometric logic subsumes \emph{coherent logic} -- the fragment of first order logic involving finite conjunction, finite disjunction, equality and existential quantification.  Moreover, via a process dubbed \emph{Morleyization}, for every theory $\theory$ of classical first order logic, i.e.\ coherent logic with negation $\neg$ and universal quantification $\forall$ satisfying the law of excluded middle, {etc.}, there is a coherent theory $\theory'$ such that the models of $\theory$ and $\theory'$ (at least in $\sets$) coincide.  The idea is to introduce, for each coherent formula $\phi$ in context $\vec{x}$, a new relation symbol $(\neg \phi)(\vec{x})$ and axioms expressing that $(\neg \phi)$ is the classical negation of $\phi$.  See \cite[Lemma D1.5.13]{elephant} for the full construction.  Thus, theories of classical first order logic have classifying toposes.
\end{rem}

\paragraph{The universal model.}

The classifying topos $\topos_\theory$ has an internal model of $\theory$ known as the \emph{universal model} $U_\theory$ (which corresponds to the morphism $\id_{\topos_\theory}$ under the equivalence \cref{eq:classtopos}; see \cite[\S 2.1.2]{TST}).  The universal model is an entirely syntactic construction -- indeed, it is witnessed by the inclusion of the \emph{syntactic category} of $\theory$ as a full subcategory of $\topos_\theory$.  For a geometric formula $\phi$ in context $\vec{x}$, we will use $\form{\phi}{x}$ to denote the object of $\topos_\theory$ that interprets $\phi$ in the universal model.  Given a geometric morphism $f \colon \ftopos \to \topos_\theory$, the interpretation of $\phi$ in the model $M$ corresponding to $f$ under the equivalence \cref{eq:classtopos} is obtained by applying the inverse image part $f^\ast$ to the object $\form{\phi}{x}$.

Two properties of the objects $\form{\phi}{x}$ will prove important in our narrative:
\begin{enumerate}
	\item First, the objects $\form{\phi}{x}$ form a \emph{generating set} for $\topos_\theory$ (this follows from the standard construction of the classifying topos, e.g.\ \cite[\S D3.1]{elephant}).
	\item The second, which follows from \cite[Lemma D1.4.4(iv)]{elephant}, will become important enough later to warrant a separate lemma.
\end{enumerate}

\begin{lem}\label{lem:subobjofform}
	The subobjects of $\form{\phi}{x}$ in $\topos_\theory$ are precisely the objects $\form{\psi}{x}$, where $\psi$ is a geometric formula for which $\theory$ proves $\psi \vdash_{\vec{x}} \phi$.
\end{lem}

\subsection{Equivariant sheaves}\label{subsec:eqsheaves}

A \emph{topological groupoid} is a groupoid internal to the category $\Top$ of topological spaces and continuous maps.  That is, a topological groupoid $\X$ consists of a diagram
\begin{equation}\label{eq:groupoid}\begin{tikzcd}
		X_1 \times_{X_0} X_1 \ar[shift left = 4]{r}{\pi_{2}} \ar{r}{m} \ar[shift right = 4]{r}{\pi_{1}} & X_1  \ar[loop, distance=2em, in=305, out=235, "i"'] \ar[shift left = 4]{r}{t} \ar[shift right = 4]{r}{s} & \ar{l}[']{e} X_0 ,
\end{tikzcd}\end{equation}
in $\Top$ such that the equations
\[s \circ e = t \circ e = \id_{X_0}, \ \ s \circ m = s \circ \pi_1, \ \ t \circ m = t \circ \pi_2,\]
\[m \circ (\id_{X_1} \times_{X_0} m) = m \circ (m \times_{X_0} \id_{X_1}), \ \  m \circ (\id_{X_1} \times_{X_0} e \circ t)= \id_{X_1} = m \circ (e \circ s \times_{X_0} \id_{X_1}) ,\]
expressing that \cref{eq:groupoid} is an internal category (where $s$ and $t$ send an arrow to, respectively, its source and target, $e$ sends an object to its identity morphism and $m$ sends a pair of composable arrows to their composite), and
\[s \circ i = t, \ \ t \circ i = s, \ \ i \circ i = \id_{X_1} ,\]
\[ m \circ (\id_{X_1} \times_{X_0} i) = e \circ s, \ \  m \circ (i \times_{X_0} \id_{X_1}) = e \circ t, \]
expressing that $i$ sends an arrow to its inverse, are all satisfied.  Equivalently, a topological groupoid is a (small) groupoid in the usual sense and a choice of topologies for the set of objects and the set of arrows such that all the groupoid structure morphisms (i.e.\ the displayed arrows in \cref{eq:groupoid}) are continuous with respect to these topologies.  

Since we will mostly be concerned with the `source' and `target' maps $s$ and $t$, we will often write $\X = (X_1 \rightrightarrows X_0)$ to denote the topological groupoid.  As $s \circ i = t$, and $i$ is a homeomorphism, $s$ is open if and only if $t$ is open.  This allows us to simplify the definition of an open topological groupoid.

\begin{df}
	A topological groupoid is said to be \emph{open} if either $s$ or $t$ are open maps (and hence both are).  
\end{df}

Open topological groupoids will play an important role in our approach.  Of particular importance for us is the fact that in an open topological groupoid, the \emph{orbit} $t(s^{-1}(U))$ of an open $U \subseteq X_0$, i.e.\ the closure of $U$ under the action of $X_1$, is still open.  The restriction to open topological groupoids is not prohibitive since every topological groupoid is \emph{Morita-equivalent} to an open one, in the sense that for every topological groupoid $\X$, its \emph{topos of equivariant sheaves} (defined below) is equivalent to the sheaves on an open topological groupoid.  This follows from \cite{BM}.

\begin{df}\label{df:toposofsheaves}
	Given a topological groupoid $\X$, we can construct the topos of \emph{equivariant sheaves} $\Sh(\X)$.  This is a construction that generalises simultaneously both toposes of sheaves on spaces and toposes of continuous group actions.
	\begin{enumerate}
		\item\label{df:enum:sheaves} Objects of $\Sh(\X)$, called $\X$-\emph{sheaves}, consist of triples $(Y,q,\beta)$ where $q \colon Y \to X_0$ is a local homeomorphism and $\beta$ is a continuous $X_1$-\emph{action} on $Y$, by which mean a continuous map
		\[
		\beta \colon Y \times_{X_0} X_1 \to Y,
		\]
		where $Y \times_{X_0} X_1$ is the pullback
		\[
		\begin{tikzcd}
			Y \times_{X_0} X_1 \ar{r} \ar{d} & Y \ar{d}{q} \\
			X_1 \ar{r}{s} & X_0,
		\end{tikzcd}
		\]
		satisfying the equations
		\[\beta(\beta(y,g),h) = \beta(y,m(g,h)), \ q(\beta(y,g)) = t(g), \ \beta(y,e(q(y))) = y.\]
		\item\label{df:enum:morphofsheaves} An arrow $(Y,q,\beta) \xrightarrow{f} (Y',q',\beta')$ of $\Sh(\X)$ consists of a continuous map $f \colon Y \to Y'$ such that the diagram
		\begin{equation}\label{eq:morphofsheaves}
			\begin{tikzcd}[column sep=tiny]
				Y \times_{X_0} X_1 \ar{rr}{f \times_{X_0} \id_{X_1}} \ar{d}[']{\beta} && Y' \times_{X_0} X_1  \ar{d}{\beta'} \\
				Y \ar{rr}{f} \ar{rd}[']{q} && Y' \ar{ld}{q'} \\
				& X_0  &
			\end{tikzcd}
		\end{equation}
		commutes.
	\end{enumerate}
\end{df}

\begin{exs}
	That the topos of equivariant sheaves on a topological groupoid simultaneously generalises the topos of sheaves on a space and the topos of continuous actions by a topological group is clear by the following examples.
	\begin{enumerate}
		\item Let $X$ be a topological space.  The diagram
		\[
		\begin{tikzcd}
			X \ar{r}{\id_X} \ar[shift left = 4]{r}{\id_X} \ar[shift right = 4]{r}{\id_X} & X \arrow[loop,  "\id_X "', distance=2em, in=305, out=235] \ar[shift left = 4]{r}{\id_X} \ar[shift right = 4]{r}{\id_X} & X \ar{l}[']{\id_X}
		\end{tikzcd}
		\]
		is a topological groupoid whose topos of sheaves is the familiar topos of sheaves on a space $\Sh(X)$.
		\item Let $(G,e,m,i)$ be a topological group.  The diagram
		\[\begin{tikzcd}
			G \times G \ar[shift left = 4]{r}{\pi_1} \ar{r}{m} \ar[shift right = 4]{r}{\pi_2} & G  \arrow[loop,  "i "', distance=2em, in=305, out=235] \ar[shift left = 3]{r}{!} \ar[shift right = 3]{r}{!} & 1 \ar{l}[']{e}
		\end{tikzcd}\]
		is a topological groupoid whose topos of sheaves is the topos $\BG$ of continuous group actions by $G$ on discrete sets (see \S III.9 \cite{SGL}).
	\end{enumerate}
\end{exs}

\paragraph{Subsheaves.}  The subobjects of an $\X$-sheaf $(Y,q,\beta)$ are easy to describe.  A morphism $(Y,q,\beta) \xrightarrow{f} (Y',q',\beta')$ of $\X$-sheaves is a monomorphism in $\Sh(\X)$ if and only if $ Y \xrightarrow{f} Y'$ is a monomorphism in $\Sh(X_0)$, i.e.\ $f$ is the inclusion of an open subspace.  The requirement that $f$ makes the diagram \cref{eq:morphofsheaves} forces the subspace $Y \subseteq Y'$ to also be \emph{stable}\footnote{Note that we are following the terminology of \cite{awodeyforssell,forssellphd,forssell}, where the term `stable' was used to reduce confusion with closed subspaces.} under the $X_1$-action $\beta$ on $Y'$, by which we mean that if $y \in Y \subseteq Y'$ then $\beta(y,\alpha) \in Y \subseteq Y'$ too, for any suitable $\alpha \in X_1$.  Thus, we arrive at the following.

\begin{lem}\label{lem:subobjofsheaves}
	The subobjects of a $\X$-sheaf $(Y,q,\beta)$ can be identified with the $\X$-sheaves $(U,q \circ i , \beta|_{U})$, where $i \colon U \hookrightarrow Y$ is the inclusion of an open and stable subspace. 
\end{lem}
%%%%%%%%%%%%%%%%%%%%%%%%%%%%%%%%%%%%%%%%

%%%%%%%%%%%%%%%%%%%%%%%%%%%%%%%%%%%%%%%%%%%%

\section{The classification theorem}\label{sec:thetheorem}

In order to state the classification theorem for representing open topological groupoids, we must first develop our terminology for \emph{indexed structures} and \emph{definables}.  The former notion is jointly inspired the signature of the \emph{diagram} of a model (see \cite[Definition 2.3.2]{marker}) and the \emph{enumerated models} and \emph{indexed models} studied in, respectively, \cite{BM} and \cite{awodeyforssell,forssell,forssellphd} (the connection with these works will be fully illustrated in \cref{sec:forssell}).  Indexed structures capture the intuition of constructing models from a list of parameter names.  Meanwhile, the latter notion of definables extends the standard notion of definable subset found in model theory (see \cite[\S 3]{hodges}).  The two properties that characterise representing open topological groupoids, `conservativity' and `elimination of parameters', are introduced in \cref{sec:thmstatement}, in which we also state the classification theorem.  `Conservativity' will be a familiar notion to the logician, but we believe `elimination of parameters' to be a novel addition.

\subsection{Indexed structures}\label{subsec:index}

%%%%%%%%%%%%%%%%%%%%%%%%%%%%%%%%%%%%%%%%%%%

Let $\Sigma$ be a signature.  Given a $\Sigma$-structure $M$, a standard model-theoretic construction is to consider $M$ as a structure over the expanded signature $\Sigma \cup \{\,c_n \mid n \in M \,\}$, the signature of the \emph{diagram} of $M$, where we have added a constant symbol for each element of $M$.  This allows us to express via formulae over the expanded signature those subsets of $M$ that are defined in relation finite tuples of elements of $M$.  We present a modification of this construction below.

\begin{df}
	Let $\Sigma$ be a signature with $N$ sorts, and let $\Index = (\Index_k)_{k \in N}$ be an $N$-tuple of sets.  We denote by $\Sigma \cup \Index$ the expanded signature obtained by adding, for each $m \in \Index_k$, a constant symbol $c_m$ (of the $k$th sort) to $\Sigma$.  We will call these added constant symbols \emph{parameters}.  
	
	A $\Index$-\emph{indexing} of a $\Sigma$-structure $M$ consists of:
	\begin{enumerate}
		\item a \emph{sub-expansion} of $\Sigma \cup \Index$, that is the signature $\Sigma \cup \lrset{c_m}{m \in \Index'_k, \, k \in N}$ for a tuple $\Index' = (\Index'_k)_{k \in N}$ of subsets $\Index'_k \subseteq \Index_k$,
		\item and an interpretation of $M$ as a $\Sigma \cup \lrset{c_m}{m \in \Index'_k, \, k \in N}$-structure such that, for each $k \in N$, $M$ satisfies the sequent
		\[
		\top \vdash_x \bigvee_{m \in \Index'_k} x = c_m.
		\]
	\end{enumerate} 
	In other words, we have interpreted in $M$ some of the parameters introduced by $\Index$ in such a way that every element $n \in M$ is the interpretation of a parameter.
	
\end{df}

Our definition of $\Index$-indexed structure is equivalent to the homonymous notion found in \cite{awodeyforssell,forssell,forssellphd} that a $\Sigma$-structure is $\Index$-indexed if the interpretation of the $k$th sort $M^{A_k}$ is presented as a subquotient of $\Index_k$, i.e.\:
\begin{enumerate}
	\item there is a partial surjection $\Index_k \paronto M^{A_k}$,
	\item or equivalently, there is a subset $S \subseteq \Index_k$ and an equivalence relation $\sim$ on $S$ such that $M^{A_k} = S/\sim$.
\end{enumerate}
We will abuse notation and write $m$ for both the parameter as an element of $\Index$ and its interpretation in an $\Index$-indexed structure $M$.  We denote a choice of $\Index$-indexing of $M$ by $\Index \paronto M$.  

Now let $\X = (X_1 \rightrightarrows X_0)$ be a (small) groupoid with a functor $p \colon \X \to \Tmodels{\sets}$.  Because our applications exclusively concern the case where the groupoid $\X$ is already a groupoid of set-based models for $\theory$ and $p$ is simply the inclusion functor, we will notationally conflate an object of $\X$ with its image under $p$, i.e.\ a set-based model of $\theory$, and similarly we will conflate an arrow of $\X$ with its image, i.e.\ an isomorphism of $\theory$-models.  Additionally, when $p$ is the inclusion of a subgroupoid, we will omit it from our notation.

\begin{df}
	Let $\theory$ be a geometric theory over a signature $\Sigma$, and let $\X = (X_1 \rightrightarrows X_0)$ be a groupoid with a functor $p \colon \X \to \Tmodels{\sets}$.  An $\Index$-\emph{indexing} of the pair $(\X,p)$ is a $\Index$-indexing $\Index \paronto M$ for each model $M $ in the image of $p$.  We will write $\Index \paronto \X$ for such a choice of indexing.
	
	Note that an element $n \in M$ can be the interpretation of multiple parameters, and also that models are allowed to share parameters, i.e.\ for $M, N \in X_0$, the same parameter $m \in \Index$ can be interpreted in both $M$ and $N$.

\end{df}

\begin{exs}\label{ex:indexedstructures}
	Every such functor $p \colon \X \to \Tmodels{\sets}$ admits multiple indexings by various sets of parameters.
	\begin{enumerate}
		\item\label{ex:indexedstructures:enum:trivial} Every model is trivially indexed by its own elements, and so $\X$ can be indexed by the set $\bigcup_{M \in X_0} M$.
		\item Since $\X$ is a small groupoid, there is a sufficiently large cardinal $\Indexset$ such that every $M \in X_0$ is of cardinality at most $\Indexset$, and thus there is a choice of partial surjection $\Indexset \paronto M$ for each $M \in X_0$.
	\end{enumerate}
\end{exs}

%%%%%%%%%%%%%%%%%%%%%%%%%%%%%%%%%%%%%%%%

\subsection{Definables}\label{subsec:definables}

In this subsection, we generalise definable subsets to groupoids of models, and show that this generalisation naturally carries the structure of an equivariant sheaf over the groupoid of models in question (when endowed with the discrete topologies).  Let $M$ be a model of a theory $\theory$.  We use the notation $\classv{\varphi}{x}_M$ to denote the subset \emph{defined} by the formula in context $\form{\varphi}{x}$, i.e.\ $\classv{\varphi}{x}_M = \{\,\vec{n} \in M \mid M \vDash \varphi(\vec{n})\,\}$.

\begin{dfs}Let $p \colon \X \to \Tmodels{\sets}$ be a functor whose domain is a (small) groupoid.
	\begin{enumerate}
		\item The \emph{definable} of a formula in context $\form{\varphi}{x}$, which we denote by $\classv{\varphi}{x}_\X$, is the coproduct $\coprod_{M \in X_0} \classv{\varphi}{x}_{M}$.  Elements of $\classv{\varphi}{x}_\X$ we denote as pairs $\langle \vec{n} , M\rangle$, where $\vec{n} \in M$ and $M \in X_0$;
		
		\item Suppose each model in the image of $p$ is indexed by a set of parameters $\Index$.  For a tuple $\vec{m}$ of parameters of $\Index$, we denote by $\class{\vec{x},\vec{m}:\psi}_\X$ the \emph{definable with parameters}
		\[\class{\vec{x},\vec{m}:\psi}_\X =  \{\,\langle\vec{n},  M \rangle \mid \vec{m} \in M \in X_0 \text{ and } M \vDash \psi(\vec{n},\vec{m})\,\}.\]
		If $\vec{m} = \emptyset$, we say that the definable $\class{\vec{x},\vec{m}:\psi}_\X$ is \emph{definable without parameters}.  Note that our models may share parameters.
	\end{enumerate}
	
\end{dfs}

A definable $\classv{\varphi}{x}_\X$ possesses an evident projection $\pi_{\classv{\varphi}{x}}$ to $X_0$ which sends the pair $\langle \vec{n},M\rangle$ to the model $M \in X_0$, as visually represented in the bundle diagram:
\[
\begin{tikzpicture}
	%X_0
	\draw (-.5,0) node {$M$};
	\draw (1,0) node {$M'$};
	\draw (2.25,0) node {$\dots$};
	\draw (3.5,0) node {$N$};
	\draw[color = gray, fill = none, rounded corners, very thick, dashed] (-1,-0.5) rectangle (4,0.5);
	\draw[color = gray] (4.5,-0.5) node {$X_0$.};

	\draw (2.25,3.5) node {$\dots$};
	
	%M
	\draw (-0.5,5) node {$\vec{a}$};
	\draw (-0.5,4) node {$\vec{a}'$};
	\draw (-0.5,3) node {$\vdots$};
	\draw (-0.5,2) node {$\vec{a}''$};
	\draw[color = Maroon, fill = none, rounded corners, very thick, dashed] (-1,1.5) rectangle (0,5.5);
	\draw[color = Maroon] (-1.5,5.8) node {$\classv{\varphi}{x}_M$};
	
	%M'
	\draw (1,5) node {$\vec{b}$};
	%\draw (1,4) node {$\vec{b}'$};
	\draw (1,3.5) node {$\vdots$};
	\draw (1,2) node {$\vec{b}'$};
	\draw[color = Plum, fill = none, rounded corners, very thick, dashed] (0.5,1.5) rectangle (1.5,5.5);
	\draw[color = Plum] (2,5.8) node {$\classv{\varphi}{x}_{M'}$};
	
	%N
	\draw (3.5,4.5) node {$\vec{c}$};
	\draw (3.5,3.5) node {$\vdots$};
	\draw (3.5,2.5) node {$\vec{c}\,'$};
	\draw[color = Orchid, fill = none, rounded corners, very thick, dashed] (3,2) rectangle (4,5);
	\draw[color = Orchid] (4.5,5.3) node {$\classv{\varphi}{x}_N$};
	
	%projections
	\draw[->] (-.5,1.25) -- (-.5,.75);
	\draw[->] (1,1.25) -- (1,.75);
	\draw[->] (3.5,1.75) -- (3.5,.75);
	
\end{tikzpicture}
\]

\paragraph{Actions on definables.}
The bundle $\pi_{\classv{\varphi}{x}} \colon \classv{\varphi}{x}_\X \to X_0$ admits a canonical \emph{lifting} of the $X_1$-action on $X_0$.  By this we mean there is a map $\theta_{\classv{\varphi}{x}} \colon \classv{\varphi}{x}_\X \times_{X_0} X_1 \to \classv{\varphi}{x}_\X$, where $\classv{\varphi}{x}_\X \times_{X_0} X_1$ is the pullback
\[
\begin{tikzcd}
	\classv{\varphi}{x}_\X \times_{X_0} X_1 \ar{r} \ar{d} & \classv{\varphi}{x}_\X \ar{d}{\pi_{\classv{\varphi}{x}}} \\
	X_1 \ar{r}{s} & X_0,
\end{tikzcd}
\]
satisfying the equations
\[
\theta_{\classv{\varphi}{x}}(\theta_{\classv{\varphi}{x}}(\langle \vec{m}, M \rangle,\alpha),\gamma) = \theta_{\classv{\varphi}{x}}(\vec{m},\gamma \circ \alpha), \ \ \theta_{\classv{\varphi}{x}}(\langle \vec{m}, M \rangle, \id_M) = \langle \vec{m} , M \rangle.
\]
This defines $(\classv{\varphi}{x}_\X, \pi_{\classv{\varphi}{x}},\theta_{\classv{\varphi}{x}})$ as a sheaf on the groupoid $\X = (X_1 \rightrightarrows X_0)$, in the sense of \cref{df:toposofsheaves}\cref{df:enum:sheaves}, when $X_1$ and $X_0$ are both endowed with the discrete topology.

For each $\theory$-model isomorphism $M \xrightarrow{\alpha} M' \in X_1$, the map $\theta_{\classv{\varphi}{x}}$ acts by sending the pair $(\langle \vec{a} , M \rangle,\alpha)$, where $\langle \vec{a} , M \rangle \in \classv{\varphi}{x}_\X$, to $\langle \alpha(\vec{a}), M' \rangle$.  Since $\alpha$ is a morphism of $\theory$-models, $M' \vDash \varphi(\alpha(\vec{a}))$ and so $\theta_{\classv{\varphi}{x}_\X}$ is well-defined.  The action $\theta_{\classv{\varphi}{x}}$ can be visualised as acting on the bundle $\pi_{\classv{\varphi}{x}} \colon \classv{\varphi}{x}_\X \to X_0$ as in the diagram
\[
\begin{tikzpicture}
	%X_0
	\node (M) at (-.5,0) {$M$};
	\node (M') at (1,0) {$M'$};
	\draw (2.25,0) node {$\dots$};
	\draw (3.5,0) node {$N$};
	\draw[color = gray, fill = none, rounded corners, very thick, dashed] (-1,-0.5) rectangle (4,0.5);
	\draw[color = gray] (4.5,-0.5) node {$X_0$.};

	\draw (2.25,3.5) node {$\dots$};
	
	%M
	\node (a) at (-0.5,5) {$\vec{a}$};
	\draw (-0.5,4) node {$\vec{a}'$};
	\draw (-0.5,3) node {$\vdots$};
	\draw (-0.5,2) node {$\vec{a}''$};
	\draw[color = Maroon, fill = none, rounded corners, very thick, dashed] (-1,1.5) rectangle (0,5.5);
	\draw[color = Maroon] (-1.5,5.8) node {$\classv{\varphi}{x}_M$};
	
	%M'
	\draw (1,5) node {$\vec{b}$};
	\node[color = Emerald] (aimg) at (1,4) {$\alpha(\vec{a})$};
	\draw (1,3) node {$\vdots$};
	\draw (1,2) node {$\vec{b}'$};
	\draw[color = Plum, fill = none, rounded corners, very thick, dashed] (0.5,1.5) rectangle (1.5,5.5);
	\draw[color = Plum] (2,5.8) node {$\classv{\varphi}{x}_{M'}$};
	
	%N
	\draw (3.5,4.5) node {$\vec{c}$};
	\draw (3.5,3.5) node {$\vdots$};
	\draw (3.5,2.5) node {$\vec{c}\,'$};
	\draw[color = Orchid, fill = none, rounded corners, very thick, dashed] (3,2) rectangle (4,5);
	\draw[color = Orchid] (4.5,5.3) node {$\classv{\varphi}{x}_N$};
	
	%projections
	\draw[->] (-.5,1.25) -- (-.5,.75);
	\draw[->] (1,1.25) -- (1,.75);
	\draw[->] (3.5,1.75) -- (3.5,.75);

	%arrow alpha
	\draw[color = Emerald, ->] (M) to[bend right] node[midway, above] {$\alpha$} (M');
	\draw[color = White, line width = 1mm] (a) to[bend left] (aimg);
	\draw[color = Emerald, ->] (a) to[bend left]  (aimg);
\end{tikzpicture}
\]

%%%%%%%%%%%%%%%%%%%%%%
\paragraph{Orbits of definables.}  Note that every definable with parameters forms a subset of a definable without parameters, e.g.\ $\class{\vec{x},\vec{m}:\varphi}_\X \subseteq \classv{\exists \vec{y} \, \varphi}{x}_\X, \classv{\top}{x}_\X$.  The $X_1$-action $\theta_{\classv{\top}{x}}$ does not restrict to an $X_1$-action on $\class{\vec{x},\vec{m}:\varphi}_\X$ since the subset may not be stable under the action, i.e.\ if $\langle \vec{a}, M\rangle \in \class{\vec{x},\vec{m}:\varphi}$ and there is an isomorphism of $\theory$-models $M \xrightarrow{\alpha} M' \in X_1$, it does not follow that $M' \vDash \varphi(\alpha(\vec{a}),\vec{m})$.  This is because $\alpha$ is an isomorphism only of the $\Sigma$-structure and need not preserve the interpretation of any of the parameters we have added.  Indeed, $\vec{m}$ might not even be interpreted in $M'$.  Of course, if $\vec{m} = \emptyset$, then $\classv{\varphi}{x}_\X$ is stable under the $X_1$-action $\theta_{\classv{\top}{x}}$, the restricted action being precisely $\theta_{\classv{\varphi}{x}}$.

\begin{df}\label{df:orbit}
	The \emph{orbit} $\overline{\class{\vec{x},\vec{m}:\psi}}_\X$ of a definable with parameters $\class{\vec{x},\vec{m}:\psi}_\X$ is the closure under the isomorphisms contained in $X_1$, explicitly:
	\[\overline{\class{\vec{x},\vec{m}:\psi}}_\X = \lrset{\langle \vec{n}, N \rangle}{ \exists \, M \xrightarrow{\alpha} N \in X_1 \text{ such that } M \vDash \psi(\alpha^{-1}(\vec{n}),\vec{m}) }.\]
	Equivalently, $\overline{\class{\vec{x},\vec{m}:\psi}}_\X$ is the smallest stable subset of $\classv{\top}{x}_\X$ containing $\class{\vec{x},\vec{m}:\psi}_\X$.
\end{df}

%%%%%%%%%%%%%%%%%%%%%%%%

\subsection{Statement of the classification theorem}\label{sec:thmstatement}

Having developed our terminology for the definables of a groupoid of $\theory$-models indexed by $\Index$, we can now state the classification theorem.

\begin{dfs}\label{df:conservandelimpara}
	Let $\theory$ be a geometric theory over a signature $\Sigma$ and let $\X = (X_1 \rightrightarrows X_0)$ be a (small) groupoid with a functor $p \colon \X \to \Tmodels{\sets}$ where each model in the image of $p$ is indexed by a set of parameters $\Index$.
	\begin{enumerate}
		\item We say that the pair $(\X,p)$ is \emph{conservative} if the set of models in the image of $p$ is a conservative, i.e.\ for each pair of geometric formulae over $\Sigma$ in context $\vec{x}$, if $\classv{\varphi}{x}_\X \subseteq \classv{\psi}{x}_\X$, then $\theory$ proves the sequent $\varphi \vdash_{\vec{x}} \psi$.
		\item We say that $(\X,p)$ \emph{eliminates parameters} if the orbit of each definable with parameters $\class{\vec{x},\vec{m}:\psi}_\X$ is definable without parameters, i.e.\ there exists a geometric formula in context $\form{\varphi}{x}$ such that
		\[\overline{\class{\vec{x},\vec{m}:\psi}}_\X = \classv{\varphi}{x}_\X.\]
	\end{enumerate}
\end{dfs}

\begin{rems}\label{rem:df:conservandelimpara}
	Let $\theory$ be a geometric theory over a signature $\Sigma$, let $\X$ be a groupoid and let $p \colon \X \to \Tmodels{\sets}$ be a functor with a choice of indexing $\Index \paronto \X$ by a set of parameters $\Index$.
	\begin{enumerate}
		\item Recall that a topos $\topos$ \emph{has enough points} if the points $ \sets \to \topos$	are jointly conservative -- that is the inverse image functors are jointly faithful. If $\topos$ classifies a theory $\theory$, then $\topos$ has enough points if and only if the set-based models of $\theory$ are conservative.  By \cite[Corollary 7.17]{topos}, if $\topos$ has enough points, a \emph{small} set of conservative models can always be chosen.  We will therefore mix our terminology and say that a theory has enough points to mean there exists a conservative set of set-based models.
		
		\item Our terminology `elimination of parameters' is justified since, in the special case of field theory, it is demonstrably the groupoid removing the parameters from the defining polynomial of a solution set.  As a simple example from traditional Galois theory, the orbit of the definable with parameters $\class{x = i}_{\Q(i)} = \{\,i\,\} $ under the automorphisms of $\Q(i)$ that fix $\Q$ is definable without parameters, namely
		\[
		\overline{\class{x = i}}_{\Aut(\Q(i))} = \{\,i,-i\,\} = \class{x^2 = -1}_{\Aut(\Q(i))}.
		\]
		Our terminology `elimination of parameters' is also inspired by the parallel model-theoretic study of Galois theory and the theory of elimination of imaginaries of Bruno Poizat \cite[\S 2]{poizat} that arises therein.

		\item\label{enum:rem:df:conservandelimpara:sufficetochecktuple} To check that the pair $(\X,p)$ eliminates parameters, it suffices to show that, for each tuple of parameters $\vec{m}$, there exists a formula in context $\form{\chi}{y}$ without parameters such that
		\[
		\overline{\class{\vec{y}= \vec{m}}}_\X = \classv{\chi}{y}_\X
		\]
		since, for an arbitrary definable with parameters $\class{\vec{x},\vec{m}:\psi}_\X$, we have that
		\[
		\overline{\class{\vec{x},\vec{m}:\psi}}_\X = \overline{\class{\vec{x}:\exists \vec{y} \, \psi[\vec{y}/\vec{m}] \land \vec{y} = \vec{m}}}_\X = \classv{\exists \vec{y} \, \psi[\vec{y}/\vec{m}] \land \chi}{x}_\X.
		\]

		\item\label{enum:rem:df:elimpara:elimpara_and_sigma} We note that the condition that $(\X,p)$ eliminates parameters does not depend on the theory $\theory$ but only on the signature $\Sigma$, in the sense that if $(\X,p)$ eliminates parameters, then $(\X,q)$ also eliminates parameters, where $q$ is the composite $\X \xrightarrow{p} \Tmodels{\sets} \subseteq \mathbb{E}_\Sigma\text{-}{\bf Mod}(\sets)$ and $\mathbb{E}_\Sigma$ denotes the empty theory over the signature $\Sigma$.  We will revisit this observation in \cref{rem:subtoposiselimpara}.  Conservativity, by contrast, evidently depends on the theory $\theory$.
		
		\item\label{enum:rem:df:elimpara:maxformula} Suppose that the pair $(\X,p)$ is conservative.  Given a definable formula with parameters $\class{\vec{x},\vec{m}:\psi}_\X$, we may wonder what restrictions can be made on the formula $\varphi$ for which
		\[\overline{\class{\vec{x},\vec{m}:\psi}}_\X = \classv{\varphi}{x}_\X.\]
		We note that, since $\class{\vec{x},\vec{m}:\psi}_\X \subseteq \overline{\class{\vec{x},\vec{m}:\psi}}_\X \subseteq \classv{\exists \vec{y} \ \psi}{x}_\X$ and $(\X,p)$ is conservative, $\theory$ must prove the sequent $\varphi \vdash_{\vec{x}} \exists \vec{y} \ \psi$.  Similarly, if every instance of the parameters $\vec{m}$ in a model $M \in X_0$ satisfies a formula $\chi$, then $\theory$ also proves the sequent $\varphi \vdash_{\vec{x}} \exists \vec{y} \ \psi \land \chi$.  In particular, we have that
		\[
		\varphi \vdash_{\vec{x}} \exists \vec{y} \ \psi \land \bigwedge_{ m_i = m_j} y_i = y_j,
		\]
		where the conjunction $\bigwedge_{ m_i = m_j} y_i = y_j$ ranges over the elements of the tuple $\vec{m}$ that are equal.  Conversely, there is in general no `lower bound' on the formula $\varphi$.
		
		\item Conservative sets of models are commonly considered outside of geometric logic.  Some readers may wonder how our notion of elimination of parameters applies to {classical} first order logic, since theories of classical logic also have classifying toposes (see \cref{rem:morleyization}).  Considering the process of Morleyization by which an equivalent geometric theory $\theory''$ (over a signature $\Sigma''$) is produced from a classical one $\theory'$ (over a signature $\Sigma'$), we conclude that a geometric formula over $\Sigma''$ is $\theory''$-provably equivalent to a join of coherent formulae over $\Sigma'$ and their formal negations.  Thus, an indexed groupoid of models for a classical theory eliminates parameters if, for every definable with parameters $\class{\vec{x},\vec{m}:\psi}_\X$ -- where $\psi$ is a classical formula, there is a set $\lrset{\varphi_i}{i \in I}$ of classical formulae in context $\vec{x}$ such that
		\[
		\overline{\class{\vec{x},\vec{m}:\psi}}_\X = \bigcup_{i \in I} \classv{\varphi_i}{x}_\X.
		\]
	\end{enumerate}
	
\end{rems}

%%%

\begin{thm}[Classification of representing open topological groupoids for a geometric theory]\label{maintheorem}
	Let $\theory$ be a geometric theory, and let $\X = (X_1 \rightrightarrows X_0)$ be a (small) groupoid.  The following are equivalent:
	\begin{enumerate}
		\item There exist topologies on $X_0$ and $X_1$ making $\X$ an open topological groupoid such that there is an equivalence of toposes $\Sh(\X) \simeq \topos_\theory$; 
		
		\item there is a functor $p \colon \X \to \Tmodels{\sets}$ such that the pair $(\X,p)$ is conservative and each model in the image of $p$ admits an indexing by a set of parameters $\Index$ for which $(\X,p)$ eliminates parameters.
	\end{enumerate}
\end{thm}

The proof of \cref{maintheorem} is completed in \crefrange{sec:bigpicture}{sec:mainproof}.  We note that, if we require that the topologies involved are $T_0$ (also known as \emph{Kolmogorov} spaces), then it is possible to simplify \cref{maintheorem} as follows.

\begin{coro}\label{coro:when-topologies-are-T0}
	Let $\theory$ be a geometric theory and let $\X = (X_1 \rightrightarrows X_0)$ be a (small) groupoid.  The following are equivalent:
	\begin{enumerate}
		\item\label{enum:coro:topologies-T0:eqv} there exist topologies on $X_0$ and $X_1$, satisfying the $T_0$ separation axiom, making $\X$ an open topological groupoid such that there is an equivalence $\Sh(\X) \simeq \topos_\theory$;
		\item\label{enum:coro:topologies-T0:elimpara} the groupoid $\X$ is a groupoid of set-based models of $\theory$, i.e.\ a subgroupoid of $\Tmodels{\sets}$, such that $X_0$ is a conservative set of models and each model $M\in X_0$ admits an indexing by a set of parameters $\Index$ for which $\X$ eliminates parameters.
	\end{enumerate}
\end{coro}

\begin{rem}
	Let $\theory$ be a geometric theory and let $p \colon \X \to \Tmodels{\sets}$ be a functor whose domain is a groupoid satisfying the hypotheses of \cref{maintheorem}.  
	\begin{enumerate}
		\item In general, we can not {\it a priori} infer an indexing of $\X$ for which $(\X,p)$ eliminates parameters without knowledge of the topologies on $\X$ for which $\Sh(\X) \simeq \topos_\theory$.  In other words, there does not exist a canonical indexing $\Index_\rmcan \paronto \X$ with the property that there exist topologies making $\X$ an open topological groupoid with $\Sh(\X) \simeq \topos_\theory$ if and only if $\X$ eliminates parameters for the canonical indexing $\Index_\rmcan \paronto \X$.  This is because, as will become apparent in \cref{sec:facttopobjs}, a choice of indexing for $\X$ yields a choice of topology on the set of objects $X_0$ and, vice versa, a choice of topology on $X_0$ yields a  choice of indexing for $\X$.
		
		\item If $\X'$ is another groupoid for which there is an equivalence of categories $\X \simeq \X'$, then it is not necessarily true that $\X'$ can be endowed with topologies for which $\Sh(\X') \simeq \topos_\theory$.  This apparent defect arises because we are considering topological categories.  Indeed, if $\X'$ and $\X$ were equivalent \emph{as topological categories}, by which we mean that the functors $F \colon \X \to \X'$ and $G \colon \X' \to \X$ that witness the equivalence are continuous on both objects and arrows, then there would be an equivalence $\Sh(\X') \simeq \Sh(\X) \simeq \topos_\theory$.
	\end{enumerate}
\end{rem}

  %Given a geometric theory $\theory$, we will call a (small) groupoid $\X = (X_1 \rightrightarrows X_0)$ of set-based $\theory$-models, i.e.\ a small subcategory $\X \subseteq \Tmodels{\sets}$ in which every arrow has an inverse, a \emph{model groupoid} for $\theory$.  The case where $\X$ is allowed to be an arbitrary (small) groupoid, and the equivalence $\Sh(\X) \simeq \topos_\theory$ is also arbitrary, is obtained as a corollary of \cref{maintheorem} in ****.

%%%%%%%%%%%%%%%%%%%%
%%%%%%%%%%%%%%%%%%%%%%%

\section{The big picture}\label{sec:bigpicture}

In this section we lay out the necessary facts regarding the sheaves on a topological groupoid we will use in the proof of Theorem \ref{maintheorem}.  Let $\X = (X_1 \rightrightarrows X_0)$ be a (small) groupoid, which can also be considered as a topological groupoid where both $X_0$ and $X_1$ have both been endowed with the discrete topology.  We will write $\mathbb{X}^\delta_\delta = (X_1^\delta \rightrightarrows X_0^\delta)$ to emphasise this fact.  Let $\tau_0$ and $\tau_1$ be topologies on $X_0$ and $X_1$ respectively such that all the structure morphisms of $\X$ are continuous with respect to these topologies, i.e.\ $\mathbb{X}^{\tau_1}_{\tau_0} = (X_1^{\tau_1} \rightrightarrows X_0^{\tau_0})$ is a topological groupoid.

\begin{dfs}

	\begin{enumerate}
		\item As above, let $\Sh(\X_{\tau_0}^{\tau_1})$ (respectively, $\Sh(\X_{\delta}^\delta)$) denote the topos of sheaves on the topological groupoid $\X_{\tau_0}^{\tau_1} = (X_1^{\tau_1} \rightrightarrows X_0^{\tau_0})$ (resp., $\mathbb{X}^\delta_\delta = (X_1^\delta \rightrightarrows X_0^\delta)$).
		
		\item By $\Sh(\X_{\tau_0}^\delta)$ we denote the category whose objects are local homeomorphisms $q \colon Y \to X_0^{\tau_0}$ equipped with a (not necessarily continuous) action $\beta \colon Y \times_{X_0} X_1 \to Y$, satisfying the same equations as in \cref{df:toposofsheaves}\cref{df:enum:sheaves}.  Arrows $(Y,q,\beta) \to (Y',q',\beta')$ are continuous maps $f \colon Y \to Y'$ such that the diagram
		\[
		\begin{tikzcd}[column sep = tiny]
			Y \times_{X_0} X_1 \ar{d}[']{\beta} \ar{rr}{f \times_{X_0} \id_{X_1}} & & Y' \times_{X_0} X_1 \ar{d}{\beta'} \\
			Y \ar{rr}{f} \ar{rd}[']{q} && Y' \ar{ld}{q'} \\
			& X_0 &
		\end{tikzcd}
		\]
		commutes as in \cref{df:toposofsheaves}\cref{df:enum:morphofsheaves}.
	\end{enumerate}

\end{dfs}

\begin{rem}
	
	Note that $\mathbb{X}_{\tau_0}^\delta = (X_1^\delta \rightrightarrows X_0^{\tau_0})$ is \emph{not} a topological groupoid (unless $\tau_0$ is also the discrete topology).  If $\mathbb{X}_{\tau_0}^\delta = (X_1^\delta \rightrightarrows X_0^{\tau_0})$ were a topological groupoid, then, for each $x \in X_0$, the singleton 
	\[\{\,x\,\} = e^{-1}(\{\,\id_x\,\})\]
	would be an open subset of $X_0^{\tau_0}$. % Despite this, the category $\Sh(\Xtd)$ is still a topos, a consequence of \cite[Theorem 2.5]{cont1} and Lemma \ref{pushout} below.
	
\end{rem}

We note that the toposes $\Sh(X_0^\delta)$ and $\Sh(\Xdd)$ can be written in a more familiar manner.  There are, of course, evident equivalences
\[
\Sh(X_0^\delta) \simeq \sets/X_0 \simeq \sets^{X_0}, \ \ \Sh(\Xdd) \simeq \sets^\X.
\]
We will construct a commutative diagram of toposes and geometric morphisms
\begin{equation}\label{diag:bigpicture}
	\begin{tikzcd}[row sep = tiny]
		\Sh(X_0^\delta) \ar{r}{j} \ar{dd}[']{v} & \Sh(X_0^{\tau_0}) \ar{rd}{u} \ar{dd}{u^\delta} & \\
		&& \Sh(\Xtt) \\
		\Sh(\Xdd) \ar{r}{j'} & \Sh(\Xtd) \ar{ru}[']{w} & 
	\end{tikzcd}
\end{equation}
such that the following are satisfied:
\begin{enumerate}
	\item $j$ and $u^\delta$ are both localic surjections,
	\item $u$ is a localic surjection and, additionally, open if $\Xtt$ is an open topological groupoid,
	\item $v$ is an open localic surjection,
	\item $j'$ is a surjection,
	\item $w$ is a hyperconnected geometric morphism,
\end{enumerate}
and the left-hand square is a pushout of toposes.

The overarching method we will follow to prove Theorem \ref{maintheorem} is laid out in \S \ref{subsec:method}.  Our first obligation is to construct the diagram (\ref{diag:bigpicture}) and demonstrate the listed properties.  To do so, we will make repeated use of \cite[Theorem B2.4.6]{elephant} to deduce that whenever a functor between toposes preserves finite limits and arbitrary colimits, it is the inverse image part of a geometric morphism.

\paragraph{Forgetting the action.}

We first note that the forgetful functors $U \colon \Sh(\Xtt) \to \Sh(X_0^{\tau_0})$ and $U^\delta \colon \Sh(\Xtd) \to \Sh(X_0^{\tau_0})$,
which forget the $X_1^{\tau_1}$-action (respectively, $X_1^\delta$-action), create all colimits and finite limits.  This is deduced since a colimit or finite limit in $\Sh(X_0^{\tau_0})$ of spaces with an $X_1^{\tau_1}$-action (resp., $X_1^\delta$-action) can be given an obvious $X_1^{\tau_1}$-action (resp., $X_1^\delta$-action) making it an object of $\Sh(\Xtt)$ (resp., $\Sh(\Xtd)$).

\begin{ex}\label{ex:product_preserved}
	
	We prove in more detail, as an example, that $U \colon \Sh(\Xtt) \to \Sh(X_0^{\tau_0})$ preserves binary products, and remark that the other finite limits and arbitrary colimits follow just as easily.
	
	Let $(Y,q,\beta)$, $(Y',q',\beta')$ be objects of $\Sh(\Xtt)$.  The product of $(Y,q)$ and $(Y,q')$ in the topos $\Sh(X_0^{\tau_0})$ is given by the pullback of spaces:
	\[\begin{tikzcd}
		Y \times_{X_0} Y' \ar{r}{\pi_2} \ar{d}[']{\pi_1} & Y' \ar{d}{q'} \\
		Y \ar{r}{q} & X^{\tau_0}.
	\end{tikzcd}\]
	Let $(y,y',\alpha)$ be an element of $ Y \times_{X_0} Y' \times_{X_0} X_1^{\tau_1} $, where $Y \times_{X_0} Y' \times_{X_0} X_1^{\tau_1} $ is the pullback
	\[\begin{tikzcd}
		Y \times_{X_0} Y' \times_{X_0} X_1^{\tau_1} \ar{r} \ar{d} &  Y \times_{X_0} Y'\ar{d}  \\
		X_1^{\tau_1} \ar{r}{s} & X^{\tau_0}.
	\end{tikzcd}\]
	The definition $B(y,y',\alpha) = (\beta(y,\alpha), \beta'(y',\alpha))$ yields a $X_1^{\tau_1}$-action $B \colon Y \times_{X_0} Y' \times_{X_0} X_1^{\tau_1} \to  Y \times_{X_0} Y'$.  The action $B$ is continuous since $\beta$ and $\beta'$ are both continuous and the necessary equations on $B$ are direct  consequences of the equivalent equivalent equations for $(Y,q,\beta)$ and $(Y',q',\beta')$.  Hence $(Y \times_{X_0} Y',q \circ \pi_1, B)$ is an object of $\Sh(\Xtt)$.
	
	The projections $\pi_1 \colon Y \times_{X_0} Y' \to Y$ and $\pi_2 \colon Y \times_{X_0} Y' \to Y'$ define morphisms in $\Sh(\Xtt)$ since, by a simple diagram chase, the following diagrams both commute
	\begin{center}
		\begin{tikzcd}
			Y \times_{X_0} Y' \times_{X_0} X_1^{\tau_1}  \ar{d}[']{B} \ar{rr}{ \pi_1 \times_{X_0} \id_{X_1}} && Y \times_{X_0} X_1^{\tau_1}  \ar{d}{\beta}  \\
			Y \times_{X_0} Y' \ar{rr}{\pi_1} \ar{dr}[']{q \circ \pi_1 } && Y \ar{dl}{q} \\
			& X^{\tau_0}_0, &
		\end{tikzcd}
		\begin{tikzcd}
			Y \times_{X_0} Y' \times_{X_0} X_1^{\tau_1}  \ar{d}[']{B} \ar{rr}{ \pi_2 \times_{X_0} \id_{X_1}} &&  Y' \times_{X_0} X_1^{\tau_1} \ar{d}{\beta'}  \\
			Y \times_{X_0} Y' \ar{rr}{\pi_2} \ar{dr}[']{q \circ \pi_1 = q' \circ \pi_2 } && Y' \ar{dl}{q'} \\
			& X^{\tau_0}_0 .&
		\end{tikzcd}
	\end{center}
	
	Finally, we can demonstrate the universal property of $(Y\times_{X_0} Y', q \circ \pi_1, B)$.  Let $(Z,p,\gamma)$ be a $\X_{\tau_0}^{\tau_1}$-sheaf with morphisms $f$ and $g$ to $(Y,q,\beta)$ and $(Y',q',\beta')$ respectively.  By the universal property of the pullback, there is a unique commuting continuous map
	\[\begin{tikzcd}
		Z \ar[dashed]{rd}{h}\ar[bend left]{rrd}{g} \ar[bend right]{rdd}{f} && \\
		& Y \times_{X_0} Y' \ar{r} \ar{d} & Y' \ar{d}{q'} \\
		& Y \ar{r}{q} & X_0^{\tau_0}
	\end{tikzcd}\]
	that sends $z \in Z$ to $(f(z),g(z)) \in Y \times_{X_0} Y'$.  Thus, $(Y \times_{X_0} Y',q \circ \pi_1,B)$ is the product of $(Y,q,\beta)$ and $(Y',q',\beta')$ in the category $\Sh(\Xtt)$ if $h$ makes the diagram
	\[\begin{tikzcd}
		Z \times_{X_0} X_1^{\tau_1} \ar{d}[']{\gamma} \ar{rr}{h \times_{X_0} \id_{X_1}} &&   Y \times_{X_0} Y' \times_{X_0} X_1^{\tau_1} \ar{d}{B}  \\
		Z \ar{rr}{h} \ar{dr}[']{p} &&  Y \times_{X_0} Y'\ar{dl}{q \circ \pi_1} \\
		& X^{\tau_0}_0 &
	\end{tikzcd}\]
	commute.  This is easily checked by another diagram chase.  And so, $U \colon \Sh(\Xtt) \to \Sh(X_0^{\tau_0})$ preserves binary products.
	
\end{ex}

Since $U$ (respectively, $U^\delta$) preserves finite limits and all colimits, by \cite[Theorem B2.4.6]{elephant}, $U$ (resp., $U^\delta$) is the inverse image of a geometric morphism $u \colon \Sh(X_0^{\tau_0}) \to \Sh(\Xtt)$ (resp., $u^\delta \colon \Sh(X_0^{\tau_0}) \to \Sh(\Xtd)$) between toposes.

\begin{lem}\label{lem:forget_action_surj}
	The geometric morphisms $u \colon \Sh(X_0^{\tau_0}) \to \Sh(\Xtt)$ and $ u^\delta \colon \Sh(X_0^{\tau_0}) \to \Sh(\Xtd)$ are both localic surjections.
\end{lem}
\begin{proof}
	As $U$ (respectively, $U^\delta$) is clearly a faithful functor whose codomain $\Sh(X_0^{\tau_0})$ is a localic topos, $u$ (resp., $u^\delta$) is a surjective localic geometric morphism.
\end{proof}

\begin{lem}[Proposition 4.4 \cite{cont1}]\label{lem:forget_action_open}
	If $\Xtt$ is an open topological groupoid, $u \colon \Sh(X_0^{\tau_0}) \to \Sh(\Xtt)$ is additionally open.
\end{lem}

\begin{rem}\label{rem:epiandmonoofsheaves}
	
	The functors $U$ and $U^\delta$ above reflect jointly epimorphic families and monomorphisms.  As shown in \cite[Proposition II.6.6]{SGL}, a family of morphisms in $\Sh(X_0^{\tau_0})$ is jointly epimorphic if and only if they are jointly surjective, and hence so too in $\Sh(\Xtt)$ and $\Sh(\Xtd)$.  Also, a morphism in $\Sh(X_0^{\tau_0})$ is a monomorphism if and only if it is injective, and hence so too in $\Sh(\Xtt)$ and $\Sh(\Xtd)$.
	
\end{rem}

Let $V \colon \Sh(\Xdd) \to \Sh(X_0^\delta)$ denote the analogous functor that forgets the $X_1^\delta$-action.  By an identical analysis to the above, we conclude the following.

\begin{lem}
	The functor $V$ is the inverse image functor of a geometric morphism $v \colon \Sh(X_0^\delta) \to \Sh(\Xdd) $ that is open, localic and surjective.
\end{lem}

\paragraph{Forgetting the topology on arrows.}

The topos $\Sh(\Xtt)$ is evidently a full subcategory of $\Sh(\Xtd)$.  Denote the inclusion functor by $W \colon \Sh(\Xtt) \to \Sh(\Xtd)$.  Clearly, there is a commuting triangle of functors
\[\begin{tikzcd}
	\Sh(\Xtt) \ar{r}{W} \ar{rd}[']{U}& \Sh(\Xtd) \ar{d}{U^\delta}\\
	& \Sh(X_0^{\tau_0}).
\end{tikzcd}\]
Since finite limits and arbitrary colimits in $\Sh(\Xtt)$ and $\Sh(\Xtd)$ are computed in the same fashion, $W$ preserves them, and therefore $W$ is the inverse image of a geometric morphism $w \colon \Sh(\Xtd) \to \Sh(\Xtt)$.

\begin{prop}\label{prop:coarsehyp}
	The geometric morphism $w \colon \Sh(\Xtd) \to \Sh(\Xtt)$ is hyperconnected.
\end{prop}
\begin{proof}
	A hyperconnected geometric morphism is one whose inverse image functor is full and faithful and its image is closed under subobjects (see \cite[Proposition A4.6.6]{elephant}).  The functor $W \colon \Sh(\Xtt) \to \Sh(\Xtd)$ is already full and faithful by construction.
	
	Let $(Y,q,\beta)$ be a $\Xtd$-space whose $X_1^\delta$-action $\beta$ becomes a continuous map $\beta \colon Y \times_{X_0} X_1^{\tau_1} \to Y$ when $X_1$ is endowed with the topology ${\tau_1}$.  If $Z$ is a subobject of $Y$ in the topos $\Sh(\Xtd)$, then $Z$ is an open subspace of $Y$ whose $X_1^\delta$-action is the restriction of $\beta$ to the subset $Z \times_{X_0} X_1 \subseteq Y \times_{X_0} X_1$.  Since, for each $U \in \opens(Z)$, $\beta|_{Z \times_{X_0} X_1}^{-1}(U) = \beta^{-1}(U) \cap ( Z \times_{X_0} X_1)$, and as $\beta$ is continuous for the topology on $Y \times_{X_0} X_1^{\tau_1}$, so too is $\beta|_{Z \times_{X_0} X_1} \colon Z \times_{X_0} X_1^{\tau_1}  \to Z$.  Thus, the image of $W$ is closed under subobjects.
\end{proof}

\begin{rem}\label{rem:prop:coarsehyp}
	
	Let $\X = (X_1 \rightrightarrows X_0)$ be a groupoid that becomes a topological groupoid when $X_0$ is endowed with the topology $\tau_0$ and $X_1$ with $\tau_1$.  The construction of the hyperconnected morphism $w \colon \Sh(\Xtd) \to \Sh(\Xtt)$ relied only on the fact that the discrete topology $\delta$ contains the topology $\tau_1$.  Indeed, if $\sigma$ is another topology on $X_1$, containing $\tau_1$, such that $\X_{\tau_0}^{\sigma} = \left(X_1^{\sigma} \rightrightarrows X_0^{\tau_0}\right)$ is also a topological groupoid, then there is a hyperconnected geometric morphism $\Sh\left(\X_{\tau_0}^{\sigma}\right) \to \Sh(\Xtt)$ whose inverse image is the inclusion of $\Sh(\Xtt)$ into $\Sh\left(\X_{\tau_0}^{\sigma}\right)$.
	
\end{rem}

\paragraph{Forgetting the topology on objects.}

The identity $\id_{X_0} \colon X_0^\delta \to X_0^{\tau_0}$ is a surjective continuous map of topological spaces and so induces (see \cite[\S II.9 \& \S IX.4]{SGL}) a surjective localic geometric morphism 
\[
\begin{tikzcd}
	j \colon \Sh(X_0^\delta) \ar{r} & \Sh(X_0^{\tau_0}).
\end{tikzcd}\]
The inverse image of $j$ is the functor $J \colon \Sh(X_0^{\tau_0}) \to \Sh(X_0^\delta)$ that sends a local homeomorphism $q \colon Y \to X_0^{\tau_0}$ the pullback of $q$ along $\id_{X_0} \colon X_0^\delta \to X_0^{\tau_0}$.  In other words, $J$ is the functor that forgets the topology.  We denote $J(Y)$ by $Y^\delta$.

There is a similar forgetful functor $J' \colon \Sh(\Xtd) \to \Sh(\Xdd)$ that sends a $\Xtd$-space $(Y,q,\beta)$ to $(Y^\delta,q,\beta)$.  Clearly, there is a commutative square
\[\begin{tikzcd}
	\Sh(\Xtd) \ar{r}{U^\delta} \ar{d}[']{J'} & \Sh(X_0^{\tau_0})\ar{d}{J}   \\
	\Sh(\Xdd) \ar{r}{V}  & \Sh(X_0^\delta) .
\end{tikzcd}\]
Once again, it is easily shown that $J'$ preserves finite limits and arbitrary colimits too.  Therefore, $J'$ is the inverse image of a geometric morphism $j' \colon \Sh(\Xdd) \to \Sh(\Xtd)$ which makes the square
\[\begin{tikzcd}
	\Sh(X_0^\delta) \ar{r}{j} \ar{d}[']{v} & \Sh(X_0^{\tau_0}) \ar{d}{u^\delta} \\
	\Sh(\Xdd) \ar{r}{j'} & \Sh(\Xtd)
\end{tikzcd}\]
commute.  Moreover, since $j$, $u^\delta$ and $v$ are surjective geometric morphisms, so too is $j'$ (since surjective geometric morphisms are the left class in an orthogonal factorisation system, see Theorem A4.2.10 \cite{elephant}).

\begin{lem}\label{pushout}
	The square
	\[\begin{tikzcd}
		\Sh(X_0^\delta) \ar{r}{j} \ar{d}[']{v} & \Sh(X_0^{\tau_0}) \ar{d}{u^\delta} \\
		\Sh(\Xdd) \ar{r}{j'} & \Sh(\Xtd)
	\end{tikzcd}\]
	is a (bi)pushout in the (bi)category $\Topos$ of Grothendieck toposes and geometric morphisms.
\end{lem}
\begin{proof}
	By \cite[Theorem 2.5]{cont1}, the (bi)pushout of the diagram
	\[\begin{tikzcd}
		\Sh(X_0^\delta) \ar{r}{j} \ar{d}[']{v} & \Sh(X_0^{\tau_0})  \\
		\Sh(\Xdd)  & 
	\end{tikzcd}\]
	in $\Topos$ is computed as the bipullback (see \cite[Example 15]{lack}) of the inverse image functors
	\begin{equation}\label{lem:pushout:eq:cospan}\begin{tikzcd}
			& \Sh(X_0^{\tau_0})   \ar{d}{J} \\
			\Sh(\Xdd)\ar{r}{V} & \Sh(X_0^\delta)  
	\end{tikzcd}\end{equation}
	in $\mathfrak{CAT}$, the category of (large) categories.

	It is then easy to see that the commuting square
	\begin{equation}\label{pushout:eq:bicone}\begin{tikzcd}
			\Sh(\Xtd) \ar{r}{U^\delta} \ar{d}[']{J'} & \Sh(X_0^{\tau_0})\ar{d}{J}   \\
			\Sh(\Xdd) \ar{r}{V}  & \Sh(X_0^\delta) 
	\end{tikzcd}\end{equation}
	is said bipullback.  Given a bicone $F \colon A \to \Sh(X_0^{\tau_0})$, $G \colon A \to \Sh(\Xdd)$ of the cospan \cref{lem:pushout:eq:cospan}, i.e.\ $J \circ F \cong V \circ G$, there is a unique (up to isomorphism) functor such that the diagram
	\[\begin{tikzcd}
		A \ar[dashed]{rd} \ar[bend left]{rrd}{F} \ar[bend right]{rdd}[']{G}&& \\
		&\Sh(\Xtd) \ar{r}{U^\delta} \ar{d}[']{J'} & \Sh(X_0^{\tau_0})\ar{d}{J}   \\
		&\Sh(\Xdd) \ar{r}{V}  & \Sh(X_0^\delta) .
	\end{tikzcd}\]
	commutes up to isomorphism.  The functor $\begin{tikzcd}[column sep = small]
		&[-15pt] A \ar[dashed]{r} & \Sh(\Xtd) &[-15pt]
	\end{tikzcd}$ is constructed as follows.
	\begin{enumerate}
		\item For an object $a \in A$, $F(a)$ is a local homeomorphism $Y \to X_0^{\tau_0}$.  Since $Y^\delta = J \circ F(a) \cong V \circ G(a)$, the set $Y^\delta$ can be endowed with the (non-continuous) $X_1^\delta$-action given by $G(a)$, thus defining $Y \to X_0^{\tau_0}$ as an object of $\Sh(\Xtd)$.
		\item Each arrow $g$ of $A$ is sent by $F$ to a continuous map $f \colon Y \to Y'$ for which the triangle
		\[
		\begin{tikzcd}[column sep=tiny]
			Y \ar{rd} \ar{rr}{f} && Y' \ar{ld}\\
			& X_0^{\tau_0}  &
		\end{tikzcd}
		\]
		commutes.  Again using that $f = V \circ G(g) \cong J \circ F(g)$, we deduce that $f$ is also equivariant with respect to the imposed $X_1^\delta$-actions on $Y^\delta$ and ${Y'}^\delta$.  Thus, $f$ also defines an arrow of $\Sh(\Xtd)$.
	\end{enumerate} 
	It is clear by definition that the diagram \cref{pushout:eq:bicone} commutes up to natural isomorphism.
	
	It remains to show that these natural isomorphisms also satisfy the universal property required by the bipullback.  However, we can elide these details since, as the perceptive reader will notice, by applying the same reasoning as above, $\Sh(\Xtd)$ is also the 1-pullback of the cospan \cref{lem:pushout:eq:cospan}.  By \cite{joyalstreet}, we know that the bipullback and the 1-pullback are equivalent if the functor $V \colon \Sh(\Xdd) \to \Sh(X_0^\delta)$ satisfies the \emph{invertible-path lifting property}.  This is indeed the case: if $f$ is an $X_1^\delta$-equivariant map of sets over $X_0^\delta$ such that $V(f)$ has an inverse in $\Sh(X_0^\delta)$, then this inverse must also be $X_1^\delta$-equivariant, and so defines an inverse for $f$ in $\Sh(\Xdd)$.
\end{proof}

This completes the construction of the diagram \cref{diag:bigpicture} and the demonstration of its required properties.

\subsection{Our method}\label{subsec:method}

We now lay out the method that we will follow in \crefrange{sec:facttopobjs}{sec:mainproof} in order to prove \cref{maintheorem}.  Let $\theory$ be a geometric theory and $\X = (X_1 \rightrightarrows X_0)$ a groupoid.  First, recall from \cite[Corollary D1.2.14]{elephant} that a model of $\theory$ internal to the topos $\sets^\X$ is simply a functor $\X \to \Tmodels{\sets}$.  Thus, the datum of a functor $p \colon \X \to \Tmodels{\sets}$ is equivalent to the datum of a geometric morphism
\[\begin{tikzcd}
\p \colon \sets^\X \simeq \Sh(\Xdd) \ar{r} & \topos_\theory.
\end{tikzcd}\]
By construction, the inverse image functor $\p^\ast$ sends the object $\form{\phi}{x} \in \topos_\theory$ to $\classv{\phi}{x}_\X \in \Sh(\X)$.  We will denote the composite
\[
\begin{tikzcd}
	\Sh(X_0^\delta) \ar{r}{v} & \Sh(\Xdd) \ar{r}{\p} & \topos_\theory
\end{tikzcd}
\]
by $\p_0$.

\begin{df}
	
	A \emph{factoring topology for objects} is a topology $\tau_0$ on $X_0$ such that the geometric morphism $\p_0$ factors as
	\[\begin{tikzcd}
		\Sh(X_0^\delta) \ar{r}{j} & \Sh(X_0^{\tau_0}) \ar[dashed]{r} & \topos_\theory.
	\end{tikzcd}\]
	
\end{df}

Given a factoring topology for objects $\tau_0$, recall that, by \cref{pushout}, there exists a unique (up to isomorphism) geometric morphism $\p'$ such that
\[\begin{tikzcd}
	\Sh(X_0^\delta) \ar{r}{j} \ar{d}{v} & \Sh(X_0^{\tau_0}) \ar{d}{u^\delta} \ar[bend left]{ddr} & \\
	\Sh(\Xdd) \ar{r}{j'} \ar[bend right = 2em]{drr}[']{\p} & \Sh(\Xtd) \ar[dashed]{dr}{\p'}& \\
	&& \topos_\theory
\end{tikzcd}\]
commutes.

\begin{df}
	Given a factoring topology for objects $\tau_0$ on $X_0$, a \emph{factoring topology for arrows} is a topology $\tau_1$ on $X_1$ such that $\Xtt = (X_1^{\tau_1} \rightrightarrows X_0^{\tau_0})$ is a topological groupoid and $\p'$ factors as
	\[\begin{tikzcd}
		\Sh(\Xtd) \ar{r}{w} & \Sh(\Xtt) \ar[dashed]{r} & \topos_\theory.
	\end{tikzcd}\]
	We denote the factoring geometric morphism by $\p'' \colon \Sh(\Xtt) \to \topos_\theory$.
	%A factoring topology $\tau_1$ is said to be an \emph{open factoring topology} if the groupoid $\Xtt = (X_1^{\tau_1} \rightrightarrows X_0^{\tau_0})$ is moreover an \emph{open} topological groupoid.
\end{df}

Our method for proving \cref{maintheorem} is broken down into four intermediate steps as follows.

\begin{enumerate}[label = {(\Alph*)}]
	\item In \cref{sec:facttopobjs} we classify the possible factoring topologies for $X_0$.  We will show that, when the models in the image of $p$ are indexed by some set of parameters $\Index$, the logical topology on objects, introduced in \cite{forssell,forssellphd}, is a factoring topology for $X_0$ and that, up to a choice of indexing for each $M \in X_0$, every factoring topology for $X_0$ contains a logical topology on objects.
	
	\item Given a factoring topology on objects for $X_0$, we classify in \cref{sec:facttopars} the factoring topologies on arrows for $X_1$.  We will show that a topology $\tau_1$ on $X_1$ is a factoring topology for arrows if and only if $\tau_1$ contains the logical topology for arrows, another topology utilised in \cite{forssell,forssellphd} (where we have assumed that $\tau_0$ contains a logical topology for objects for some indexing of the models).
	
	\item We demonstrate in \cref{subsec:logtopislocalic} that, if $X_0$ and $X_1$ are endowed with the logical topologies $\taulogo$ and $\tauloga$, and that the resulting topological groupoid is \emph{open}, then the factoring geometric morphism 
	\begin{equation}\label{eq:localicmorphinmethod}
		\begin{tikzcd}
			\p'' \colon \Sh\left(\Xlog\right) \ar{r} & \topos_\theory
		\end{tikzcd}
	\end{equation}
	is localic.
	
	\item Finally in \cref{sec:mainproof}, we deduce \cref{maintheorem} by studying the homomorphisms of subobject lattices induced by $\p''$.
\end{enumerate}

%%%%%%%%%%%%%%%%%%%%%%%%%%%%%%%%%%%%%%%%%

%%%%%%%%%%%%%%%%%%%%%%%%%%%%%%%%%%%%%%%%%

\paragraph{Comparison with Joyal-Tierney descent.}  

Before embarking on the main proof, we elaborate further the connection between the representation of toposes by topological data and localic data.  

Let $\topos$ be a topos.  Let $X_0$ be a set of points of $\topos$, i.e.\ $X_0 \subseteq \Geom(\sets,\topos)$, and let $\tau_0$ be a factoring topology for objects on $X_0$, i.e.\ there is a canonically induced geometric morphism
\[
\begin{tikzcd}
	\Sh\left(X_0^{\tau_0}\right) \ar{r} & \topos.
\end{tikzcd}
\]
There are two groupoids that can naturally be associated with the pair $(X_0,\tau_0)$.
\begin{enumerate}
	\item Firstly, there is the \emph{concrete} groupoid $\X = (X_1 \rightrightarrows X_0)$ (i.e.\ a groupoid internal to $\sets$) obtained by taking $X_1$ as the set of all isomorphisms of the points $X_0$ (which is always small).
	
	\item Secondly, there is the \emph{localic} groupoid $\X^{\rm loc} = (X_1^{\rm loc} \rightrightarrows X_0^{\rm loc})$, whose locale of objects is the frame of opens $\opens(X_0^{\tau_0})$ and whose locale of arrows is the locale in the (bi)pullback of toposes
	\[
	\begin{tikzcd}
		\Sh(X_1^{\rm loc}) \ar{r} \ar{d} & \Sh(X_0^{\tau_0}) \ar{d} \\
		\Sh(X_0^{\tau_0}) \ar{r} & \topos.
	\end{tikzcd}
	\] 
	Thus, a point of the locale $X_1^{\rm loc}$ corresponds to an isomorphism between the points of $X_0$, i.e.\ a point of $X_1$.  However, the locale $X_1^{\rm loc}$ may be non-spatial.
\end{enumerate}
It is natural to ask: if $\X$ is the underlying groupoid of an (open) representing topological groupoid for $\topos$, is $\X^{\rm loc}$ a representing \emph{localic} groupoid, and vice versa?

One implication is true.  Suppose that there exists a topology $\tau_1$ on $X_1$ making $\Xtt = (X_1^{\tau_1} \rightrightarrows X_0^{\tau_0})$ an open topological groupoid for which there is an equivalence of toposes $\topos \simeq \Sh(\Xtt)$.  Then, by \cref{lem:forget_action_surj} and \cref{lem:forget_action_open}, the canonical geometric morphism
\[
\begin{tikzcd}
	u_\X \colon \Sh\left(X^{\tau_0}\right) \ar{r} & \Sh(\Xtt) \simeq \topos
\end{tikzcd}
\]
is an open surjection, and so, by an application of the descent theory of Joyal and Tierney \cite{JT}, the topos $\topos$ is also the topos of sheaves on the localic groupoid $\X^{\rm loc}$.

The converse, however, is not true.  In \cref{ex:atomic_but_not_ultra} we give a counterexample consisting of a topos $\topos$ and an open surjective point $p \colon \sets \to \topos$ for which there is an equivalence $\topos \simeq \B \Aut(p)^{\rm loc}$ (here $\Aut(p)^{\rm loc}$ denotes the localic automorphism group of $p$) but where there is not an equivalence $\topos \not \simeq \B \Aut(p)^{\tau_1}$ for any topology $\tau_1$ on the concrete automorphism group $\Aut(p)$.

%%%%%%%%%%%%%%%%%%%%%%%%%%%%%%%%%%%%%%%%%

%%%%%%%%%%%%%%%%%%%%%%%%%%%%%%%%%%%%%%%%%

%%%%%%%%%%%%%%%%%%%%%%%%%%%%%%%%%%%%%%%%%%

\section{Factoring topologies for objects}\label{sec:facttopobjs}

Let $\theory$ be a geometric theory and let $p_0 \colon X_0 \to \Tmodels{\sets}$ be a function picking out set-based models of $\theory$.  We wish to classify the possible factoring topologies for $X_0$.  We will begin by recalling from \cite{forssellphd} a choice of factoring topology is present when the constituent models of $X_0$ are endowed with some indexing by a set of parameters.  Next, we demonstrate that any factoring topology yields such an indexing of the models in $X_0$, and hence we arrive at our characterisation of factoring topologies for objects.

%%%%%

\paragraph{The logical topology on objects.}

Given an indexing of each model in our set of objects by some parameters, we can define a certain factoring topology for objects: the \emph{logical topology for objects}.  This is an adaptation of the topology used in the papers by Awodey and Forssell (see \cite[Definition 3.1.1.2]{forssellphd}, \cite[Definition 3.1]{forssell}), itself an adaptation of the topology used by Butz and Moerdijk in \cite[\S 2]{BM}.  As we will see in \cref{thm:logictopo}, any factoring topology for objects $\tau_0$ on $X_0$ can, essentially, be chosen to contain a logical topology.

\begin{df}[Definition 3.1 \cite{forssell}, Definition 1.2.1 \cite{awodeyforssell}]
	
	Let $\theory$ be a geometric theory and let $p_0 \colon X_0 \to \Tmodels{\sets}$ be a function where each model in the image of $p_0$ is indexed by a set of parameters $\Index$.  The \emph{logical topology for objects} $\taulogo$ on $X_0$ is the topology generated by the basis whose opens are \emph{sentences with parameters} 
	\[\class{\vec{m}:\varphi}_\X = \lrset{M \in X_0}{M \vDash \phi(\vec{m})},\]
	where $\phi$ is a geometric formula and $\vec{m} \in \Index$ is a tuple of parameters.
	% i.e.\ those definables with parameters $\class{\vec{x},\vec{m}:\varphi}_\X$ where $\vec{x}=\emptyset$.
\end{df}

\begin{rem}\label{rem:taulogo:atomic_sents_only}
%	Let $X_0$ be a set of indexed $\theory$-models as above.  
	We note that, when generating the logical topology for objects $\taulogo$ on $X_0$, we can focus our attention on only those opens $\classv{\varphi}{m}_\X$ where $\varphi$ is an \emph{atomic formula} (see \cite[Definition D1.1.3]{elephant}).  This is because any of the logical symbols $\left\{\,\land, \bigvee, \exists\,\right\}$ used to construct a composite geometric formula from atomic ones can be replaced by topological constructions:
	\begin{enumerate}
		\item $\classv{\varphi \land \psi}{m}_\X = \classv{\varphi}{m}_\X \cap \classv{\psi}{m}_\X$,
		\item $\lrclass{\vec{m}:\bigvee_{i \in I} \varphi_i}_\X = \bigcup_{i \in I} \classv{\varphi_i}{m}_\X$,
		\item $\classv{\exists y \, \varphi}{m}_\X = \bigcup_{{m}' \in \Index} \class{\vec{m},{m}':\varphi}_\X$.
	\end{enumerate}
\end{rem}

\begin{lem}[Lemmas 3.1.1.6 \& 3.1.1.10 \cite{forssellphd}]\label{logtopisfact}
	Let $\theory$ be a geometric theory and let $p_0 \colon X_0 \to \Tmodels{\sets}$ be a function.  For any indexing of the models in the image of $p_0$ by parameters, the resultant `logical topology for objects' is a factoring topology for objects.
\end{lem}

The proof of \cref{logtopisfact} is easily adapted from the exposition found in \cite[\S 3.1.1]{forssellphd}, which deal with the specific case of Forssell groupoids (which we will study in more detail in \cref{sec:forssell}).  We sketch the salient points here.  To demonstrate that $\p_0 \colon \Sh(X_0^\delta) \to \topos_\theory$ factors through $j \colon \Sh(X_0^\delta) \to \Sh(X_0^\taulogo)$, it suffices to show that the restriction of the inverse image $\p_0^\ast$ to (the full subcategory spanned by) a set of generating objects for $\topos_\theory$ factors through $J \colon \Sh(X_0^\taulogo) \to \Sh(X_0^\delta)$.  

In particular, for each formula $\phi$ in context $\vec{x}$, we must find a topology $T_{\vec{x}}$ on $\classv{\phi }{x}_\X$ for which $\pi_{\classv{\phi}{x}} \colon \classv{\phi}{x}_\X \to X_0^\taulogo$ is a local homeomorphism.  The topology generated by the basis consisting of definables with parameters
\[
\class{\vec{x},\vec{m} : \phi \land \psi}_\X = \lrset{\lrangle{\vec{n},M}}{M \vDash \phi(\vec{n}) \land \psi(\vec{n},\vec{m})} \subseteq \classv{\phi}{x}_\X
\]
suffices.  This is because each $\lrangle{\vec{n},M} \in \classv{\phi}{x}_\X$ is contained in the image of the evident local section $s \colon \classv{\phi}{m}_\X \to \classv{\phi \land \vec{x} = \vec{m}}{x}_\X \subseteq  \classv{\phi}{x}_\X$, which acts by $M \mapsto \lrangle{\vec{m},M}$, where $\vec{m}$ is a tuple of parameters indexing $\vec{n} \in M$.  Further details can be found in \cite[\S 3.1.1]{forssellphd}.

%%%%%%%%%%%%%%%%%%%%%%%%%%%%%%%%%

\paragraph{Every factoring topology for objects contains the logical topology.}

There is a sense in which the `logical topology for objects' is essentially the only factoring topology that need be considered.  Specifically, we prove the following, constituting our characterisation of factoring topologies for objects.

\begin{prop}\label{thm:logictopo}
	Let $\theory$ be a geometric theory and let $p_0 \colon X_0 \to \Tmodels{\sets}$ be a function.  A topology $\tau_0$ on $X_0$ is a factoring topology for objects if and only if there exists an indexing of each model in the image of $p_0$ by a set of parameters $\Index$ such that $\tau_0$ contains a `logical topology for objects' $\taulogo$.
\end{prop}

\begin{proof}
	One direction is clear: if $\tau_0$ contains a `logical topology for objects' $\taulogo$ then, by \cref{logtopisfact}, the geometric morphism $\p_0$ factors as
	\[\begin{tikzcd}
		\Sh(X_0^\delta) \ar{r} & \Sh(X_0^{\tau_0}) \ar{r} & \Sh(X_0^{\taulogo}) \ar{r} & \topos_\theory.
	\end{tikzcd}\]
	
	When we assume instead that $\tau_0$ is a factoring topology for objects, then, in particular, for each singleton context, the projection $\pi_{\class{x:\top}} \colon \class{x:\top}_\X \to X_0^{\tau_0}$ is a local homeomorphism for some choice of topology $T_x$ on $\class{x:\top}_\X$.  In particular, every element $\lrangle{n,M} $ of $ \class{x:\top}_\X$ lies in the image of some local section $s \colon U \to \class{x:\top}_\X$ of $\pi_{\class{x:\top}}$ whose image is open.  Let $\Index$ be the set of parameters whose elements are such local sections of $\pi_{\class{x:\top}}$.  We can index each $M \in X_0$ by interpreting the parameter $s \colon U  \to \class{x:\top}_\X$ by $n \in M$ if $\lrangle{n,M}$ lies in the image $s(U)$.  Thus, the open $s(U) \subseteq \class{x:\top}_\X$ is the interpretation of the definable with parameters $\class{x = s}_\X$, as in the diagram
	\begin{center}
		\begin{tikzpicture}
			%X_0
			\node (M) at (-.5,0) {$M$};
			\node (M') at (1,0) {$M'$};
			\draw (2.25,0) node {$\dots$};
			\draw (3.5,0) node {$N$};
			\draw[color = gray, fill = none, rounded corners, very thick, dashed] (-1,-0.5) rectangle (4,0.5);
			\draw[color = gray] (4.5,-0.3) node {$X_0$,};

			\draw (2.25,3.5) node {$\dots$};
			
			%M
			\node (a) at (-0.5,5) {$a$};
			%\draw (-0.5,4) node {$\vec{a}'$};
			{\node (bimg) at (-0.5,4) {$a'$};}
			\draw (-0.5,3) node {$\vdots$};
			\draw (-0.5,2) node {$a''$};
			\draw[color = Maroon, fill = none, rounded corners, very thick, dashed] (-1,1.5) rectangle (0,5.5);
			\draw[color = Maroon] (-1.5,5.8) node {$\class{x:\top}_M$};
			
			%M'
			\node (b) at (1,5) {$b$};
			{\node (aimg) at (1,4) {$b'$};}
			\draw (1,3) node {$\vdots$};
			\draw (1,2) node {$b''$};
			\draw[color = Plum, fill = none, rounded corners, very thick, dashed] (0.5,1.5) rectangle (1.5,5.5);
			\draw[color = Plum] (2,5.8) node {$\class{x:\top}_{M'}$};
			
			%N
			\draw (3.5,4.5) node {$c$};
			\draw (3.5,3.5) node {$\vdots$};
			\draw (3.5,2.5) node {$c'$};
			\draw[color = Orchid, fill = none, rounded corners, very thick, dashed] (3,2) rectangle (4,5);
			\draw[color = Orchid] (4.5,5.3) node {$\class{x:\top}_N$};
			
			%projections
			\draw[->] (-.5,1.25) -- (-.5,.75);
			\draw[->] (1,1.25) -- (1,.75);
			\draw[->] (3.5,1.75) -- (3.5,.75);
			
			{
				%Definable
				\draw[color = Darkcyan, fill = none, rounded corners, very thick, dashed] (-0.9,4.5) rectangle (1.4,5.4);
				\draw[color = Darkcyan] (-2,4.5) node {$\class{x = s}_\X$};
			}
			
			{
				%section
				\draw[color = Darkcyan, fill = none, rounded corners, very thick, dashed] (-0.9,-0.4) rectangle (1.4,0.4);
				\draw[color = Darkcyan] (-1.3,0) node {$U$};
				\draw[->, color=Darkcyan] (-1.1,0.6) to[bend left] node[midway, left] {$s$} (-1.3,4.1);
			}

		\end{tikzpicture}
	\end{center}
	As the local sections with open image are jointly surjective, this does indeed define an indexing of each $M \in X_0$ by the parameters $\Index$.
	
	It remains to show that $\tau_0$ contains the logical topology for this indexing.  We first note that, by construction, $s(U) = \class{x = s}_\X$ is open in $\class{x:\top}_\X$.  Let $k \colon \Sh(X_0^{\tau_0}) \to \topos_\theory$ denote the factoring geometric morphism.  As $k^\ast$ preserves finite limits and $\form{\top}{x} = \prod_{{x_i \in \vec{x}}} \{x_i : \top\}$ for any context $\vec{x}$, we deduce that $\classv{\top}{x}_\X = k^\ast(\form{\top}{x})$ is the product in $\Sh(X_0^{\tau_0})$ of $ \class{x_i : \top}_\X = k^\ast(\{x_i : \top\})$ over each $x_i \in \vec{x}$.  That is, $\classv{\top}{x}_\X$ is the wide pullback of topological spaces $\class{x_1 : \top}_\X \times_{X_0} \dots \times_{X_0} \class{x_n: \top}_\X$.  In particular, $\classv{\top}{x}_\X$ is endowed with the product topology.  Therefore, for any tuple of parameters $\vec{s} \in \Index$ and geometric formula $\phi$, we have that
	\begin{align*}
		\classv{\varphi \land \vec{x} = \vec{s}}{x}_\X & = \classv{\varphi}{x}_\X \cap \prod_{x_i \in \vec{x}} \class{x_i = s_i}_\X , \\
		&= \lrset{\lrangle{\vec{n},M}}{ M \vDash \varphi(\vec{s}), \,  \vec{n} = \vec{s}}
	\end{align*}
	is an open subset of $\classv{\top}{x}_\X$.  Since the local homeomorphism  $\pi_{\classv{\top}{x}} \colon \classv{\top}{x}_\X \to X_0^{\tau_0}$ is, in particular, an open map, the image $\pi_{\classv{\top}{x}}(\classv{\varphi \land \vec{x} = \vec{s}}{x}_\X) = \classv{\varphi[\vec{s}/\vec{x}]}{s}_\X$ is open in $\tau_0$, i.e.\ $\tau_0$ contains the `logical topology for objects', as desired.
	
\end{proof}

\begin{rem}\label{rem:alwaysdefwithparam}
	Let $\theory$ be geometric theory, $p_0 \colon X_0 \to \Tmodels{\sets}$ a function, and let $\tau_0$ be a factoring topology for objects on $X_0$.  
	\begin{enumerate}
		\item When we discover an indexing of $\X$ by the set of parameters $\Index$ such that $\tau_0$ contains the `logical topology for objects' $\taulogo$ via the method of \cref{thm:logictopo}, we note that we have also forced the chosen topology on $\classv{\top}{x}_\X$ to contain as opens the definables with parameters $\class{\vec{x},\vec{m}:\varphi}_\X$.
		
		\item Suppose that $\theory$ is an \emph{inhabited} (single-sorted) theory, meaning that $\theory$ proves the sequent $ \vdash \exists x \, \top$.  Then  every model $M \in X_0$ is non-empty and so $\pi_{\class{x : \top}} \colon \class{x: \top} \to X_0^{\tau_0}$ is a surjective local homeomorphism.  This entails that $\tau_0$ is generated by the opens of the form $U = \pi_{\class{x:\top}} s(U)$, where $s \colon U \to \class{x:\top}_\X$ is a local section of $\pi_{\class{x : \top}}$ whose image is open, i.e.\ $\tau_0$ is generated by the subsets $\class{s : \top}_\X \subseteq X_0$ where $s$ is a parameter in the indexing set $\Index$ constructed in \cref{thm:logictopo}.  Thus, for this indexing of $\X$, the logical topology for objects coincides with $\tau_0$.  That is, when $\theory$ is an inhabited theory, every factoring topology for objects is the logical topology for objects for some indexing $\Index \paronto \X$.
	\end{enumerate}
\end{rem}

%%%%%%%%%%%%%%%%%%%%%%

%%%%%%%%%%%%%%%%%%%%%
%%%%%%%%%%%%%%%%%%%%%%%%%%%%

\section{Factoring topologies for arrows}\label{sec:facttopars}

Let $\theory$ be a geometric theory, $\mathbb{X} = (X_1 \rightrightarrows X_0)$ a groupoid with a functor $p \colon \X \to \Tmodels{\sets}$, and let $\tau_0$ be a factoring topology for objects on $X_0$.  We seek to classify the possible factoring topologies for arrows on $X_1$.  As in our sketch proof of \cref{logtopisfact}, we note that there is a factorisation
\[
\begin{tikzcd}
	\Sh(\Xtd) \ar{rd}{\p'} \ar{d}[']{w} & \\
	\Sh(\Xtt) \ar[dashed]{r} & \topos_\theory
\end{tikzcd}
\]
if and only if the inverse image functor ${\p'}^\ast$ restricted to a generating set of objects for $\topos_\theory$ factors through $W \colon \Sh(\Xtt) \to \Sh(\Xtd)$.  That is to say, $\tau_1$ is a factoring topology for arrows if and only if  the $X_1^\delta$-action on $\class{\vec{x}:\varphi}_\X$ is continuous, for each geometric formula $\phi$. In fact, it suffices to only check that, for each context $\vec{x}$, the action $\theta_{\classv{\top}{x}} \colon \classv{\top}{x}_\X \times_{X_0} X_1^{\tau_1} \to  \classv{\top}{x}_\X $ is continuous since then the restriction of $\theta_{\classv{\top}{x}}$ to the subspace $\classv{\varphi}{x}_\X \subseteq \classv{\top}{x}_\X$, i.e.\ the action $\theta_{\classv{\varphi}{x}}$, is continuous as well.

%%%%%%%%%%%%%%%%%%%%%%%%%%%%%%

%%%%%%%%%%%%%%%%%%%%%%%%%%%%%%%%%

\paragraph{Logical topology for arrows.}  

We first recall the logical topology for arrows, another variation on a topology utilised in \cite{awodeyforssell,forssellphd,forssell,BM}, and that this is a factoring topology for arrows.  Much like the logical topology for objects, we will observe in \cref{prop:allfactaislog} that the logical topology for arrows plays a special role among all factoring topologies for arrows.

\begin{df}[Definition 3.1 \cite{forssell}, Definition 1.2.1 \cite{awodeyforssell}]\label{df:logtopforar}
	Let $\theory$ a geometric theory, $\X = (X_1 \rightrightarrows X_0)$ a groupoid and $p \colon \X \to \Tmodels{\sets}$ a functor such that each model in the image of $p$ is indexed by the set of parameters $\Index$.  The \emph{logical topology for arrows} is the topology on $X_1$ generated by basic opens of the form
	\[
	\lrclass{
		\begin{matrix}
			\vec{a}:\varphi \\
			\vec{b} \mapsto \vec{c} \\
			\vec{d} : \psi
		\end{matrix}
	}_\X =
	\lrset{M \xrightarrow{\alpha} N \in X_1}{ 
		\begin{matrix}
			\vec{a} \in M, \, M \vDash \varphi(\vec{a}), \\
			\vec{b} \in M, \, \vec{c} \in N, \, \alpha(\vec{b}) = \vec{c}, \\
			\vec{d} \in N , \, N \vDash \psi(\vec{d})
	\end{matrix}},
	\]
	where $\form{\varphi}{x}, \form{\psi}{y}$ are formulae in context, and $\vec{a},\vec{b},\vec{c},\vec{d}$ are tuples of parameters in $\Index$.  We will write $\class{\vec{b} \mapsto \vec{c}}_\X$ for the subset $\{\,\alpha \in X_1\mid \alpha(\vec{b}) = \vec{c}\,\} \subseteq X_1$ (i.e., when $\vec{a} = \vec{d} = \emptyset$ and $\phi = \psi = \top$).
\end{df}

\newcommand\hugecup{%
	\scaleobj{2}{%
		\bigcup\limits_{\smash{\raisebox{0\baselineskip}{\(\scaleobj{0.425}{\vec{e} \in \Index}\)}}}
	}
}

\begin{lem}[Lemma 3.2.1.1 \cite{forssellphd}]\label{lem:topologicalgrpdwhenlogtopa}
	The groupoid $\Xlog$,
	\[\begin{tikzcd}
		X^{\tauloga}_1 \times_{X_0} X^{\tauloga}_1 \ar[shift left = 4]{r}{\pi_{2}} \ar{r}{m} \ar[shift right = 4]{r}{\pi_{1}} & X^{\tauloga}_1 \ar[loop, distance=2em, in=305, out=235, "i"'] \ar[shift left = 4]{r}{t} \ar[shift right = 4]{r}{s} & \ar{l}[']{e} X^{\taulogo}_0 ,
	\end{tikzcd}\]
	is a topological groupoid, i.e.\ the maps $s$, $t$, etc., are continuous.
\end{lem}

\begin{lem}[Lemma 3.2.2.1 \cite{forssellphd}]\label{lem:logtopaisfact}
	Suppose that $X_0$ is endowed with the logical topology on objects $\taulogo$.  When $X_1$ is given the logical topology on arrows $\tauloga$, the action $\theta_{\classv{\top}{x}} \colon \classv{\top}{x}_\X \times_{X_0} X_1^{\tau_1} \to  \classv{\top}{x}_\X $ is continuous, and hence $\tauloga$ is a factoring topology for arrows.
\end{lem}

%%%%%%%%%%%%%%%%%%%%%%%%%%%%%%%%%%%%%%%%%

Next, we express the special role the logical topology for arrows plays amongst all factoring topologies.  This constitutes our characterisation of the latter.

\begin{prop}\label{prop:allfactaislog}
	Let $\theory$ be a geometric theory, $\mathbb{X} = (X_1 \rightrightarrows X_0)$ a groupoid, $p \colon \X \to \Tmodels{\sets}$ a functor, and let $\tau_0$ be a factoring topology for objects and $\tau_1$ a factoring topology for arrows.  If $\tau_0$ contains the logical topology for objects when each model in the image of $p$ is indexed by a set of parameters $\Index$, then $\tau_1$ contains the logical topology on arrows for this indexing.
\end{prop}
\begin{proof}
	We note that, for each basic open in $\tauloga$,
	\[
	\lrclass{
		\begin{matrix}
			\vec{a}:\varphi \\
			\vec{b} \mapsto \vec{c} \\
			\vec{d} : \psi
		\end{matrix}
	}_\X
	= s^{-1}(\classv{\varphi}{a}_\X) \cap \class{\vec{b} \mapsto \vec{c}}_\X \cap \, t^{-1}(\classv{\psi}{d}_\X).
	\]
	As $s,t$ are continuous maps and $\classv{\varphi}{a}_\X, \classv{\psi}{d}_\X$ are both open in $X_0^{\tau_0}$, since $\tau_0 \supseteq \taulogo$, we deduce that $s^{-1}(\classv{\varphi}{a}_\X), t^{-1}(\classv{\psi}{d}_\X) \subseteq X_1^{\tau_1}$ are open.  Thus, it suffices to show that $\tau_1$ contains the subset
	\[
	\class{\vec{b} \mapsto\vec{c}}_\X = \lrset{M \xrightarrow{\alpha} N \in X_1 }{ M \vDash \alpha(\vec{b}) = \vec{c}}.
	\]
	By \cref{rem:alwaysdefwithparam}, we may assume that the topology $T_{\vec{x}}$ on $\classv{\top}{x}_\X$, for which $\pi_{\classv{\top}{x}} \colon \classv{\top}{x}_\X^{T_{\vec{x}}} \to X_0^{\tau_0}$ is a local homeomorphism, contains as opens the definables with parameters $\class{\vec{x},\vec{m}:\varphi}_\X$.
	
	Since $\theta_{\classv{\top}{x}}$ is continuous, the subset
	\[ \theta_{\classv{\top}{x}}^{-1}(\class{x= \vec{c}}_\X) \cap \left( \class{\vec{x}=\vec{b}}_\X \times_{X_0} X_1 \right) = \{\,(\langle \vec{n}, M \rangle, \alpha) \mid \vec{n} = \vec{b}, \, \alpha(\vec{n}) = \vec{c} \,\}\]
	is open in $\classv{\top}{x}_\X \times_{X_0} X_1^{\tau_1}$.  The projection $\pi_2 $ in the pullback square
	\[
	\begin{tikzcd}
		\classv{\top}{x}_\X \times_{X_0} X_1^{\tau_1}  \ar{r} \ar{d}[']{\pi_2} & \classv{\top}{x}_\X \ar{d}{\pi_{\classv{\top}{x}}} \\
		X_1^{\tau_1} \ar{r}{s} & X_0^{\tau_0}
	\end{tikzcd}
	\]
	is an open map since the local homeomorphism $\pi_{\classv{\top}{x}} \colon \classv{\top}{x}_X \to X_0^{\tau_0}$ is an open map too, and open maps are stable under pullback.  Therefore,
	\[\pi_2 \left( \theta_{\classv{\top}{x}}^{-1}(\class{x= \vec{c}}_\X) \cap \left( \class{\vec{x}=\vec{b}}_\X \times_{X_0} X_1 \right) \right)
	= \class{\vec{b} \mapsto\vec{c}}_\X
	\]
	is an open subset of $X_1^{\tau_1}$, and thus $\tau_1$ contains the logical topology on arrows.
\end{proof}

\subsection{Characterising the logical topology for arrows}\label{subsec:logtopislocalic}

Let $\theory$ be a geometric theory and let $p \colon \X \to \Tmodels{\sets}$ be a functor whose domain is a groupoid and where each model in the image of $p$ is indexed by a set of parameters $\Index$.  By \cref{lem:logtopaisfact} and \cref{prop:allfactaislog}, for any factoring topology on arrows $\tau_1$, when $X_0$ is endowed with the logical topology for objects $\taulogo$, there is a factorisation of the geometric morphism $
\p' \colon \Sh(\X_\taulogo^\delta) \to \topos_\theory$ as
\[
\begin{tikzcd}
	\Sh(\X_\taulogo^\delta) \ar{r} & \Sh(\X_\taulogo^{\tau_1}) \ar{r} &
	\Sh(\X_\taulogo^\tauloga) \ar{r} & \topos_\theory.
\end{tikzcd}
\]
Moreover, by \cref{prop:coarsehyp} and \cref{rem:prop:coarsehyp}, the factoring geometric morphisms $
\Sh(\X_\taulogo^\delta) \to \Sh(\X_\taulogo^{\tau_1})$ and $
\Sh(\X_\taulogo^{\tau_1})\to
\Sh(\X_\taulogo^\tauloga)$ are both hyperconnected morphisms.  We may therefore wonder whether the the factorisation
\begin{equation}\label{eq:hyploc}
	\begin{tikzcd}
		\Sh(\X_\taulogo^\delta) \ar{r} &
		\Sh(\X_\taulogo^\tauloga) \ar{r} & \topos_\theory.
	\end{tikzcd}
\end{equation}
is the hyperconnected-localic factorisation of the geometric morphism $\p'$.

We answer affirmatively under the condition that 
\[\Xlog = \left(X_1^\tauloga \rightrightarrows X_0^\taulogo\right)\]
is an \emph{open} topological groupoid.  In general, there is no reason for $\Xlog$ to be an open topological groupoid, though the groupoid eliminating imaginaries is a sufficient condition, as observed in \cref{lem:logicaltopsgiveopengrpd}.

The proof that \cref{eq:hyploc} is the hyperconnected-localic factorisation is essentially contained in Lemmas 2.3.4.10-13 of \cite{forssellphd} (see also Lemmas 3.2.2.8 \& 3.2.2.9 \cite{forssellphd}).  We sketch some of the details of the proof to assure the reader that the only required assumption is that $\Xlog$ is an open topological groupoid.

\begin{lem}[Lemma 2.3.4.10 \cite{forssellphd}]\label{forsselllemma}
	Suppose that $\Xlog$ is an open topological groupoid.  Let $(Y,\beta,q)$ be an $\Xlog$-sheaf.  For each $y \in Y$, there exists a basic open $\classv{\xi}{m}_\X$ of $X_0^\taulogo$ with a local section $f \colon \classv{\xi}{m}_\X \to Y$ of $q$ such that:
	\begin{enumerate}
		\item the point $y$ is in the image of $f$,
		\item for any $M \in \classv{\xi}{m}_\X$ and an isomorphism $M \xrightarrow{\alpha} N \in X_1$ such that $\alpha$ also preserves the interpretation of the parameters $\vec{m}$ (and so $N \in \classv{\xi}{m}_\X$), $\beta\left(f(M), \alpha \right) = f(N)$.
	\end{enumerate}
\end{lem}

Recall that the subobjects of a $\Xlog$-sheaf $(Y,q,\beta)$ are given precisely by the open subsets of $Y$ that are stable under the action of $\beta$.  If $\Xlog$ is an open topological groupoid, then, since $\theta_{\classv{\top}{x}}$ is open and
\[\overline{\class{\vec{x},\vec{m}:\varphi}}_\X  = \theta_{\classv{\top}{x}}(\class{\vec{x},\vec{m}:\varphi}_\X \times_{X_0} X_1),\]
the orbit of a definable with parameters $\class{\vec{x},\vec{m}:\varphi}_\X$, is an open, stable subset and hence defines a subobject of $\classv{\top}{x}_\X$.

\begin{prop}[Lemmas 2.3.4.11-13 \cite{forssellphd}]\label{prop:pislocalic}
	If $\Xlog = (X_1^\tauloga \rightrightarrows X_0^\taulogo)$ is an open topological groupoid, the factoring geometric morphism $\p^\log \colon \Sh(\Xlog) \to \topos_\theory$ is localic.
\end{prop}

We sketch the proof to \cref{prop:pislocalic}.  To show that $\p$ is localic, it suffices to show that the subobjects of $\classv{\top}{x}_\X$ form a generating set of $\Sh(\Xlog)$.  Given an object $(Y,q,\beta)$ of $\Sh(\Xlog)$ and a point $y$ of $Y$, by Lemma \ref{forsselllemma}, there exists a basic open $\classv{\xi}{m}_\X$ of $X_0^\taulogo$ with a local section $f \colon \classv{\xi}{m}_\X \to Y$ of $q$ such that:
\begin{enumerate}
	\item the point $y$ is in the image of $f$,
	\item for any $M \in \classv{\xi}{m}_\X$ and an isomorphism $M \xrightarrow{\alpha} N \in X_1$ such that $\alpha(\vec{m}) = \vec{m}$, we have that $\beta\left(f(M), \alpha \right) = f(N)$.
\end{enumerate}

Let $\vec{x}$ be a context with the same type as $\vec{m}$.  Evidently, there is a local section $g \colon \classv{\xi}{m}_\X \to \classv{\top}{x}_\X$ of $\pi_{\classv{\top}{x}} \colon \classv{\top}{x}_\X \to X_0$ that sends $M \in \classv{\xi}{m}_\X$ to $\langle \vec{m}, M \rangle$.  The image of $g$ is thus the open subset $\class{\vec{x}:\vec{x}=\vec{m} \land \xi}_\X \subseteq \classv{\top}{x}_\X$.  Hence, there is a commuting diagram of continuous maps
\[\begin{tikzcd}
	\classv{\top}{x}_\X & \ar[tail]{l} \overline{\classv{\vec{x}=\vec{m} \land \xi}{x}}_\X & Y \\
	& \ar[tail, bend left = 2em]{lu} \ar[tail]{u} \classv{\vec{x}=\vec{m} \land \xi}{x}_\X & \\
	& \classv{\xi}{m}_\X . \ar[bend left = 4em]{luu}{g} \ar[two heads]{u} \ar[bend right = 4em]{ruu}[']{f} &
\end{tikzcd}\]
The remainder of the proof consists of constructing a continuous map $h \colon \overline{\class{\vec{x}=\vec{m} \land \xi}}_\X \to Y$ which completes the above diagram and moreover constitutes a morphism of $\Xlog$-sheaves.

As each element of $\overline{\class{\vec{x}=\vec{m} \land \xi}}_\X$ is of the form $\langle \alpha(\vec{m}), N \rangle$ where $M \xrightarrow{\alpha} N$ is a $\theory$-model isomorphism in $X_1$ and $M \in \classv{\xi}{m}$, we take the obvious definition and set $h(\langle \alpha(\vec{m}), N \rangle)$ as $\beta(f(M),\alpha)$.  It must first be checked that this is well-defined, and here we use the fact that $\X$ is a groupoid.  Given a second isomorphism $M \xrightarrow{\gamma} N$ such that $\langle \alpha(\vec{m}) , N \rangle = \langle \gamma(\vec{m}) , N \rangle$, then $ \gamma^{-1} \circ \alpha$ is an automorphism of $M$, contained in $X_1$, such that $\gamma^{-1} \circ \alpha(\vec{m}) = \vec{m}$.  Hence, by hypothesis, $ \beta(f(M),\gamma^{-1} \circ \alpha) = f(M)$, and so
\[\beta(f(M),\gamma) = \beta(\beta(f(M),\gamma^{-1} \circ \alpha),\gamma)  = \beta(f(M),\alpha).\]
It remains to show that $h$ is continuous and that $h$ is a morphism of $\Xlog$-sheaves.  For these details, the reader is directed to \cite{forssellphd}.

Thus, for each object $(Y,q,\beta)$ of $\Sh(\Xlog)$, the arrows to $(Y,q,\beta)$ in $\Sh(\Xlog)$ whose domains are subobjects of $\classv{\top}{x}_\X$ are jointly surjective and therefore the geometric morphism $\p^\log \colon \Sh(\Xlog) \to \topos_\theory$ is localic.

\begin{coro}\label{coro:hyplocfactgiveslogicaltopa}
	If $\Xlog$ is an open topological groupoid, the geometric morphism $\p^\log \colon \Sh(\Xlog) \to \topos_\theory$ is the localic part of the hyperconnected-localic factorisation of $\p' \colon \Sh(\X_\taulogo^\delta) \to \topos_\theory$.
\end{coro}
\begin{proof}
	There is a commutative triangle
	\[\begin{tikzcd}
		\Sh(\X_\taulogo^\delta) \ar{r}{w} \ar{rd}[']{\p'} & \Sh(\Xlog) \ar{d}{\p^\log} \\
		& \topos_\theory,
	\end{tikzcd}\]
	for which $w$ is hyperconnected by \cref{prop:coarsehyp} and $\p^\log$ is localic by \cref{prop:pislocalic}.
\end{proof}

%%%%%%%%%%%%%%%%%%%%%%%%%%%%%%%%

\section{The proof of the classification theorem}\label{sec:mainproof}

We are now in a position to combine the results of \cref{sec:facttopobjs} and \cref{sec:facttopars} to obtain the classification theorem stated in \cref{maintheorem}.  We separate the different steps of the proof to show clearly the interaction between the two conditions: conservativity and elimination of parameters.  Conservativity, unsurprisingly, is equivalent to the induced geometric morphism being a geometric surjection.  Conversely, elimination of parameters is equivalent the induced geometric morphism being a geometric embedding.  Penultimately, we also give a sense in which the logical topologies are the only topologies that need be considered, and finally demonstrate how \cref{coro:when-topologies-are-T0} can be deduced from \cref{maintheorem}.

\begin{lem}\label{lem:surj_iff_conservative}
	Let $\theory$ be a geometric theory and let $\X = (X_1 \rightrightarrows X_0)$ be a groupoid with a choice of functor $p \colon \X \to \Tmodels{\sets}$.  Given a pair of factoring topologies $\tau_0$ on $X_0$ and $\tau_1$ on $X_1$, the factoring geometric morphism $\p'' \colon \Sh\left(\Xtt\right) \to \topos_\theory$ is a geometric surjection if and only if $(\X,p)$ is conservative, i.e.\ the set of models in the image of $p$ is conservative.
\end{lem}
\begin{proof}
	Recall that there is commutative diagram of geometric morphisms
	\[
	\begin{tikzcd}
		\sets^{X_0} \simeq \Sh(X_0^\delta) \ar{rrd}[']{\p_0} \ar{r}{j} & \Sh(X_0^{\tau_0}) \ar{r}{u} & \Sh\left(\Xtt\right) \ar{d}{\p''} \\
		&& \topos_\theory
	\end{tikzcd}
	\]
	where the top horizontal composite $u \circ j$ is a geometric surjection.  Recall also that $\{ \, M \in X_0\,\}$ is a conservative set of models for $\theory$ if and only if the geometric morphism $\p_0 \colon \sets^{X_0} \to \topos_\theory$ is a surjection.
	
	Thus, using that geometric surjections are the left class in an orthogonal factorisation system for geometric morphisms -- the (surjection,inclusion)-factorisation (see \cite[\S A4.2]{elephant}) -- and are therefore closed under composites and have the right cancellation property (if $f \circ g$ and $g$ are surjections, then so is $f$), we conclude that $\p'' \colon \Sh\left(\Xtt\right) \to \topos_\theory$ is a geometric surjection if and only if $\{ \, M \in X_0\,\}$ is a conservative set of models for $\theory$.
\end{proof}

\paragraph{From a representing groupoid to elimination of parameters.}

We now continue with the proof for one implication of \cref{maintheorem}.  We first show that an open representing groupoid $\X$ yields a functor $p \colon \X \to \Tmodels{\sets}$ which can be given an indexing by parameters for which $(\X,p)$ is conservative and eliminates parameters.

\begin{prop}\label{thm:reprimplieselimpara}
	Let $\theory$ be a geometric theory and let $\X = (X_1 \rightrightarrows X_0)$ be a groupoid.  If there exist topologies $\tau_0$ and $\tau_1$ on $X_0$ and $X_1$ making $\X$ an open topological groupoid such that there is an equivalence $\Sh(\Xtt)\simeq \topos_\theory$, then there is a functor $p \colon \X \to \Tmodels{\sets}$ such that each model in the image of $p$ admits an indexing by a set of parameters $\Index$ such that $(\X,p)$ is conservative and eliminates parameters.
\end{prop}
\begin{proof}
	Let $\p$ denote the composite geometric morphism $w \circ j' \colon \Sh(\Xdd) \to \Sh(\Xtt) \simeq \topos_\theory$.  Thus, as in \cref{subsec:method}, the geometric morphism $\p$ yields a functor $p \colon \X \to \Tmodels{\sets}$.  We apply \cref{thm:logictopo} and \cref{rem:alwaysdefwithparam} to deduce that there exists an indexing of each $M \in X_0$ by parameters $\Index$ for which the space $\classv{\top}{x}_\X$, i.e.\ the sheaf corresponding to $\form{\top}{x} \in \topos_\theory$ under the equivalence $\topos_\theory \simeq \Sh(\Xtt)$, contains as open subsets the definables with parameters.

	Since $\Xtt$ is an open groupoid, the orbit of each open subset of the form $\class{\vec{x},\vec{m}:\psi}_{\X}$ is still an open subset of $\classv{\top}{x}_\X$.  Therefore, being a stable open, $\overline{\class{\vec{x},\vec{m}:\psi}}_{\X}$ defines a subobject of $\classv{\top}{x}_\X$.  Under the equivalence $\topos_\theory \simeq \Sh(\Xtt)$, the subobjects of $\form{\top}{x}$ and $\classv{\top}{x}_\X$ must also be identified.  Hence, recalling that the subobjects of $\form{\top}{x}$ are formulae in the context $\vec{x}$, we deduce that there exists a formula $\varphi$ such that
	\[
	\overline{\class{\vec{x},\vec{m}:\psi}}_{\X} = \classv{\varphi}{x}_\X.
	\]
	Therefore the pair $(\X,p)$, as indexed by $\Index$, eliminates parameters.  Finally, since an equivalence of toposes is, in particular, a surjection, we can apply \cref{lem:surj_iff_conservative} to deduce that $(\X,p)$ is also conservative.
\end{proof}

\paragraph{From elimination of parameters to a representing groupoid.}

We now prove the converse statement of \cref{maintheorem} -- that for a groupoid $\X$ with a functor $p \colon \X \to \Tmodels{\sets}$ where the models $M \in X_0$ admit an indexing such that $(\X,p)$ conservative and eliminates parameters, there exist topologies on $\X$ making it an open representing groupoid for the theory $\theory$.  Unsurprisingly, the topologies we consider are the logical topologies studied in \crefrange{sec:facttopobjs}{sec:facttopars}.  We first demonstrate that the condition that the pair $(\X,p)$ eliminates parameters is equivalent to the induced geometric morphism being an inclusion of a subtopos.

\begin{lem}[cf.\ Lemma 3.2.1.2 \cite{forssellphd}]\label{lem:logicaltopsgiveopengrpd}
	Let $\theory$ be a geometric theory, $\X = (X_1 \rightrightarrows X_0)$ a groupoid, and $p \colon \X \to \Tmodels{\sets}$ a functor where each model in the image of $p$ is indexed by parameters $\Index$.  If $(\X,p)$ eliminates parameters, then, when both $X_1$ and $X_0$ are endowed with the logical topologies, $\Xlog$ becomes an \emph{open} topological groupoid.
\end{lem}
\begin{proof}
	We have already seen that $\Xlog$ is a topological groupoid in \cref{lem:topologicalgrpdwhenlogtopa}, so it remains to show that either of the continuous maps $s, t \colon X_1^\tauloga \rightrightarrows X_0^\taulogo$ are open too.  We will show that $t$ is open.
	
	It suffices to show that the image of each basic open of $X_1^\tauloga$ is open in $X_0^\taulogo$.  Suppose that
	\[
	N \in t\left(\lrclass{
		\begin{matrix}
			\vec{a}:\varphi \\
			\vec{b} \mapsto \vec{c} \\
			\vec{d} : \psi
		\end{matrix}
	}_\X\right).
	\]
	Then there is some isomorphism $M \xrightarrow{\alpha} N$ of $X_1$ such that $M \vDash \varphi(\vec{a})$ and $\alpha(\vec{b}) = \vec{c}$, in addition to $N \vDash \psi(\vec{d})$.  Therefore,
	\[
	\langle\vec{c},N\rangle \in \overline{\lrclass{\vec{x},\vec{a} : \vec{b} = \vec{x} \land \varphi }}_\X.
	\]
	Since $\X$ eliminates parameters, there is some formula $\chi$ without parameters such that
	\[
	\overline{\lrclass{\vec{x},\vec{a}: \vec{b} = \vec{x} \land \varphi }}_\X = \classv{\chi}{x}_\X .
	\]
	We thus conclude that $N$ is contained in the open subset $ \class{\vec{c},\vec{d}:\chi \land \psi}_\X$ of $X_0^\taulogo$.
	
	Given any other $N' \in \class{\vec{c},\vec{d}:\chi \land \psi}_\X$, we have that
	\[ \langle \vec{c}, N' \rangle \in \classv{\chi}{x}_\X = 	\overline{\lrclass{\vec{x},\vec{a}: \vec{b} = \vec{x} \land \varphi }}_\X. \]
	Thus, there exists another isomorphism $M' \xrightarrow{\gamma} N'$ of $X_1$ such that $M' \vDash \varphi(\vec{a})$ and $\gamma(\vec{b}) = \vec{c}$.  Hence,
	\[
	N, N' \in \overline{\lrclass{\vec{c},\vec{d}:\chi \land \psi}}_\X \subseteq 
	t\left(\lrclass{
		\begin{matrix}
			\vec{a}:\varphi \\
			\vec{b} \mapsto \vec{c} \\
			\vec{d} : \psi
		\end{matrix}
	}_\X\right).
	\]
	
\end{proof}

\begin{prop}\label{thm:elimparaimpliesrepr}
	Let $\theory$ be a geometric theory, $\X = (X_1 \rightrightarrows X_0)$ a groupoid, and $p \colon \X \to \Tmodels{\sets}$ a functor where each model in the image of $p$ is indexed by parameters $\Index$.  The factoring geometric morphism $\p^{\log}$ is an inclusion of a subtopos
	\[\begin{tikzcd}
		\p^\log \colon 	\Sh(\Xlog) \ar[tail]{r} & \topos_\theory
	\end{tikzcd}\]
	if and only if $(\X,p)$ eliminates parameters.
\end{prop}
\begin{proof}
	The first step is to deduce that, under either hypothesis, the geometric morphism $\p^\log$ is a localic geometric morphism.  This is clear if $\p^\log$ is an inclusion of a subtopos since every inclusion is localic (see \cite[Examples A4.6.2(a)]{elephant}).  Conversly, if $\Xlog$ eliminates parameters then, by \cref{lem:logicaltopsgiveopengrpd}, $\Xlog$ is an open topological groupoid.  Thus, by applying \cref{prop:pislocalic}, the factoring geometric morphism $
	\p^\log \colon 	\Sh(\Xlog) \to \topos_\theory$ is a localic geometric morphism.

	Thus, by \cite{myselfintloc}, the geometric morphism $\p^\log$ is a geometric embedding if and only if, for each object $A$ in a generating set of $\topos_\theory$, the induced map on subobjects 
	\[{\p^\log_A}^\ast \colon \Sub_{\topos_\theory}(A) \to \Sub_{\Sh(\Xlog)}({\p^\log}^\ast A)\]
	is surjective.  Therefore, applying \cref{lem:subobjofform}, $\p^\log$ is an embedding if and only if, for each geometric formula $\phi$, the map
	\begin{equation}\label{proof-of-classification:the-map-class}
	\class{-}_\X \colon \lrset{\psi }{\theory \text{ proves } \psi \vdash_{\vec{x}} \phi} \cong \Sub_{\topos_\theory}(\form{\phi}{x}) \to \Sub_{\Sh(\Xlog)}(\classv{\phi}{x}_\X) \
\end{equation}
	which sends $\form{\psi}{x} $ to $ \classv{\psi}{x}_\X$	is surjective.

	Recall from \cref{lem:subobjofsheaves} that a subobject of $\classv{\phi}{x}_\X$ is an open subset $U$ that is stable under the action $\theta_{\classv{\top}{x}}$.  As the opens of the form $\class{\vec{x},\vec{m}:\psi}_\X$, where $\theory$ proves $\psi \vdash_{\vec{x}} \phi$, form a basis for the topology on $\classv{\phi}{x}_\X$, every stable open subset $U$ is of the form
	\[
	U = \bigcup_{i \in I} \overline{\class{\vec{x},\vec{m}_i: \psi_i}}_\X .
	\]
	Therefore, the stable opens of the form $\overline{\class{\vec{x},\vec{m}:\psi}}_\X$ form a basis for the frame of subobjects of $\classv{\phi}{x}$, and so the map $\class{-}_\X$ in \cref{proof-of-classification:the-map-class} is surjective if and only if every basic subobject $\overline{\class{\vec{x},\vec{m}:\psi}}_\X$ is in the image of ${\class{-}_\X}$. This is precisely the condition that $\X$ eliminates parameters, from which we deduce the result.
\end{proof}

\begin{rem}\label{rem:subtoposiselimpara}
	Let $\theory$ be a geometric theory over a signature $\Sigma$, and let $p \colon \X \to \Tmodels{\sets}$ be a functor, where $\X$ is a groupoid and each model $M \in X_0$ is indexed by parameters $\Index$.  As remarked in \cref{rem:df:conservandelimpara}\cref{enum:rem:df:elimpara:elimpara_and_sigma}, the condition that $(\X,p)$ eliminates parameters depends only on the signature of the theory $\theory$.  This can be retroactively justified topos-theoretically in light of \cref{thm:elimparaimpliesrepr} by equating those pairs $(\X,p)$ that eliminate parameters with those pairs for which $\p^\log$ is an inclusion.  
	
	As $\theory$ is a \emph{quotient theory} of $\mathbb{E}_\Sigma$, the empty theory over the signature $\Sigma$, the theory $\theory$ is classified by a subtopos $\topos_\theory \rightarrowtail \topos_{\mathbb{E}_\Sigma}$ (see \cite[Examples B4.2.8(i)]{elephant} and \cite[Theorem 3.2.5]{TST}), and so, by \cref{thm:elimparaimpliesrepr}, $(\X,p)$ eliminates parameters if and only if there are inclusions of subtoposes
	\[
	\begin{tikzcd}
		\Sh(\Xlog) \ar[tail]{r} & \topos_\theory \ar[tail]{r} & \topos_{\mathbb{E}_\Sigma},
	\end{tikzcd}
	\]
	whence we deduce that $(\X,p)$ eliminating parameters depended only on the signature $\Sigma$.
\end{rem}

Since a geometric morphism is an equivalence of toposes if and only if it is both a surjection and an inclusion (see \cite[Corollary A4.2.11]{elephant}), we deduce the following.

\begin{coro}\label{coro:conservative_and_elimpara_implies_repr}
	Let $\theory$ be a geometric theory, $\X = (X_1 \rightrightarrows X_0)$ a groupoid, and $p \colon \X \to \Tmodels{\sets}$ a functor where each model in the image of $p$ is indexed by a set of parameters $\Index$.  The geometric morphism $\p^\log$ is an equivalence of toposes
	\[
	\Sh(\Xlog) \simeq \topos_\theory
	\]
	if and only if $(\X,p)$ is conservative and eliminates parameters.
\end{coro}

Combining \cref{thm:reprimplieselimpara} and \cref{coro:conservative_and_elimpara_implies_repr} completes the proof of \cref{maintheorem}.  Moreover, we can also use these results to deduce the sense in which the logical topologies are, essentially, the only topologies that need be considered.

\begin{coro}
	Let $\theory$ be a geometric theory and let $\X = (X_1 \rightrightarrows X_0)$ be a groupoid.  If there exist topologies $\tau_0$ on $X_0$ and $\tau_1$ on $X_1$ making $\Xtt$ an open topological groupoid for which $ \Sh(\Xtt)\simeq \topos_\theory$, then there is a functor $p \colon \X \to \Tmodels{\sets}$ and an indexing of each model in the image of $p$ by a set of parameters $\Index$ such that
	\[\Sh({\X}_\taulogo^\tauloga) \simeq \topos_\theory \simeq \Sh(\Xtt).\]
\end{coro}

Finally, we demonstrate how \cref{coro:when-topologies-are-T0} can be deduced from \cref{maintheorem}.

\begin{proof}[Proof of \cref{coro:when-topologies-are-T0}]
	The implication \cref{enum:coro:topologies-T0:elimpara} $\implies$ \cref{enum:coro:topologies-T0:eqv} is an immediate consequence of \cref{maintheorem}.  For the converse, suppose that there are $T_0$ topologies $\tau_0$ on $\tau_1$ on $X_0$ and $X_1$ for which $\topos_\theory\simeq \Sh(\Xtt)$.  We observe that the functor $p \colon \X \to \Tmodels{\sets}$ induced by the composite geometric morphism $w \circ j'\colon \Sh(\Xdd) \to \Sh(\Xtt) \simeq \topos_\theory$ is the inclusion of a subgroupoid, from which the result follows.  
	
	To see why this is the case, recall that, for an object $M \in X_0$, its image under $p$ is the model corresponding to the composite
	\[
	\begin{tikzcd}
		\sets \ar{r}{\Tilde{M}} & \Sh(X_0^{\tau_0}) \ar{r}{u_\X} & \Sh(\Xtt) \simeq \topos_\theory,
	\end{tikzcd}
	\]
	where $\Tilde{M} \colon \sets \to \Sh(X_0^{\tau_0})$ is the geometric morphism corresponding to $M$ (i.e.\ the inverse image part $\Tilde{M}^\ast$ sends a sheaf to its fibre over $M$).  As $\tau_0$ is a $T_0$ topology, then for distinct points $M, N \in X_0$, the geometric morphisms $\Tilde{M}, \Tilde{N}$ are also distinct, as are their composites $u_\X \circ \Tilde{M}, u_\X \circ \Tilde{N}$.  Thus, the action of $p$ on objects is injective.  Similarly, for an arrow $M \xrightarrow{\alpha} N \in X_1$, its image under $p$ is the isomorphism of models corresponding to the composite 2-cell $\Tilde{M} \Rightarrow \Tilde{N}$ in the diagram
	\[\begin{tikzcd}
		\sets \\
		& {\Sh(X_1^{\tau_1})} & {\Sh(X_0^{\tau_0})} \\
		& {\Sh(X_0^{\tau_0})} & \Sh(\Xtt) &[-30pt] \simeq \topos_\theory,
		\arrow["{\Tilde{\alpha}}", from=1-1, to=2-2]
		\arrow["\Tilde{M}", curve={height=-20pt}, from=1-1, to=2-3]
		\arrow["\Tilde{N}"', curve={height=20pt}, from=1-1, to=3-2]
		\arrow["\Sh(s)", from=2-2, to=2-3]
		\arrow["\Sh(t)"', from=2-2, to=3-2]
		\arrow[Rightarrow, from=2-3, to=3-2]
		\arrow["\sim"{marking, allow upside down}, shift left=2, draw=none, from=2-3, to=3-2]
		\arrow["u_\X", from=2-3, to=3-3]
		\arrow["u_\X"', from=3-2, to=3-3]
	\end{tikzcd}\]
	where the isomorphism $u_\X \circ \Sh(s) \cong u_\X \circ \Sh(t)$ is the canonical isomorphism deriving from the fact that $\X$ is a groupoid.  Once again, since $\tau_1$ is $T_0$, $\Tilde{\alpha}, \Tilde{\alpha}'$ are distinct for distinct arrows $\alpha, \alpha' \in X_1$, and thus $p$ is also injective on arrows. 
\end{proof}

\begin{ex}
	Any example of an open representing groupoid $\X$ for a geometric theory $\theory$ for which the induced functor $p \colon \X \to \Tmodels{\sets}$ is not an inclusion of a subgroupoid is, to some extent, pathological.  This is because it follows from \cref{maintheorem} that the image $p(\X) \subseteq \Tmodels{\theory}$ can be endowed with topologies making it an open representing groupoid for $\theory$ as well.
	
	We give a simple example of an example where $p$ is not the inclusion of a subgroupoid.  Let $\mathbb{I}$ be the theory of the terminal object, i.e.\ the propositional geometric theory with no propositional symbols and no axioms.  Then the category of models of $\mathbb{I}\text{-}{\bf Mod}(\ftopos)$, in any topos $\ftopos$, is equivalent to the trivial category $\1$, and so $\mathbb{I}$ is classified by $\sets$.  For any (non-empty) groupoid $\X$, there is a unique functor $p \colon \X \to \1$, and $(\X,p)$ is vacuously conservative and eliminates parameters.  Indeed, it is easily observed that if $X_0$ and $X_1$ are endowed with the indiscrete topologies, then $\Sh(\X) \simeq \sets$.
\end{ex}

%%%%%%%%%%%%%%%%%%%%%%%%%%%%%%%%%%%%%%%%%

%%%%%%%%%%%%%%%%%%%%%%%%%%%%%%%%%%%%%%%%%%

\section{Applications}\label{sec:applications}

In this section we present some applications of \cref{maintheorem}.  From now on, we consider the special case of \cref{maintheorem} where $\X$ is a groupoid of set-based models of $\theory$, i.e.\ a subgroupoid of $\Tmodels{\sets}$, which we will refer to as a \emph{model groupoid}, and $p$ is this inclusion.  Our applications are divided as follows.
\begin{enumerate}[label = (\arabic*)]
	
	\item The first two sections justify our use of indexings of models, i.e.\ presenting models as a subquotient of a set of parameters.  \cref{sec:atomic} is devoted to the study of atomic theories.  These are the only theories whose models of a representing groupoid may be indexed by disjoint sets of parameters.  We will recover the logical `topological Galois theory' result of \cite{caramellogalois} that the automorphism group of a single model represents an atomic theory if and only if the model is conservative and ultrahomogeneous.  We also give a characterisation of Boolean toposes with enough points reminiscent of the characterisation established in \cite{blassscedrov}.
	
	\item In \cref{subsec:decidable}, we demonstrate that, instead of subquotients, we can present the models in a representing groupoid as subsets of a set of parameters only in the case when the theory has decidable equality.  We also generate examples of representing groupoids for decidable theories.
	
	\item We demonstrate in \cref{sec:etalecomplete} that every representing model groupoid is Morita-equivalent to its \emph{étale completion}, that is the model groupoid with the same objects and all possible isomorphisms between these constituent models.  We also show that the étale completion of a representing model groupoid can be calculated as a topological closure in the fashion of \cite[\S 4]{hodges}.
	
	\item We recover in \cref{sec:forssell} the representation theorems given by Butz and Moerdijk in \cite{BM} and by Awodey and Forssell in \cite{awodeyforssell,forssell} by demonstrating that the considered groupoids fall within a general framework of `maximal groupoids'.
	
	\item Finally, having studied how to generate representing groupoids of indexed structures for a given theory, we answer the converse direction and describe a theory which is represented by a given groupoid of indexed structures.  This extends the techniques developed in \cite[Theorem 4.14]{hodges} for subgroups of the topological permutation group on a set.
\end{enumerate}

\subsection{Atomic theories}\label{sec:atomic}

In this section we will study those model groupoids of a theory that eliminate parameters when their constituent models are indexed by disjoint sets of parameters.  We will observe in \cref{prop:disjoint_implies_atomic} that this requires the theory to be atomic.  

We will revisit Caramello's `topological Galois theory' and demonstrate that the results of \cite{caramellogalois} concerning atomic theories can be recovered via the classification theorem.  We will also give a characterisation of Boolean toposes with enough points in a manner reminiscent to \cite{blassscedrov}.

\begin{df}[Proposition D3.4.13 \cite{elephant}]\label{df:atomictheories}
	A geometric theory $\theory$ is \emph{atomic} if one of the following equivalent conditions is satisfied:
	\begin{enumerate}
		\item for each context $\vec{x}$, the frame of geometric formulae in context $\vec{x}$, ordered by entailment modulo the theory $\theory$, is generated by its atoms,
		\item the classifying topos $\topos_\theory$ of $\theory$ is \emph{atomic} (see \cite[\S C3.5]{elephant}).
	\end{enumerate}
	If $\theory$ is known to possess enough points, by Theorem 3.16 \cite{caramelloatomic} we can add a further equivalent condition to the list:
	\begin{enumerate}
		\setcounter{enumi}{2}
		\item every (model-theoretic) type of $\theory$ is \emph{isolated}, also called \emph{principal} and \emph{complete}, i.e.\ for each model $M$ of $\theory$ and tuple $\vec{n} \in M$, there is a formula $\chi_{\vec{n}}$, the \emph{minimal formula} of $\vec{n}$, such that:
		\begin{enumerate}
			\item for any other tuple $\vec{n}'$ of the same sort as $\vec{n}$ in another model $N$, then $N \vDash \chi_{\vec{n}}(\vec{n}')$ if and only if $\vec{n}$ and $\vec{n}'$ satisfy the same formulae 
			%, in the notation of \cref{sec:types} 
			-- i.e.\ $\tp_M(\vec{n}) = \lrangle{\chi_{\vec{n}}}$, read as $\chi_{\vec{n}}$ \emph{isolates} the type of $\vec{n}$ (see \cite[\S 4.1]{marker});
			\item for all formulae $\varphi$ in context $\vec{x}$, either  $\theory$ proves $\chi_{\vec{m}} \vdash_{\vec{x}} \varphi$ or $\theory$ proves $\chi_{\vec{m}} \land \varphi \vdash_{\vec{x}} \bot$ -- equivalently, given a pair of tuples in two models, if $\tp_M(\vec{n}) \subseteq \tp_N(\vec{n}')$ then $\tp_M(\vec{n}) = \tp_N(\vec{n}')$.			
		\end{enumerate}
	\end{enumerate}
\end{df}

Recall also from \cite[Corollary C3.5.2]{elephant} that, under the assumption that $\theory$ has enough points, the properties that $\theory$ is an atomic geometric theory and that $\theory$ is a Boolean geometric theory (i.e.\ the classifying topos $\topos_\theory$ is Boolean) coincide.

\begin{exs}\label{exs:atomictheories}
	\begin{enumerate}
		\item\label{exs:enum:acf} The terminology \emph{minimal formula} is derived from the analogy with minimal polynomials.  We define the theory of algebraically closed fields of finite characteristic which are algebraic over their prime subfield as the theory for which, in addition to the usual axioms of an algebraically closed field, we also include as axioms the sequents
		\[
		\top \vdash_{\emptyset} \bigvee_{p \text{ prime}} \underbrace{1 + 1 \dots + 1}_{p \text{ times}} = 0, \ \ 
		\top \vdash_x \bigvee_{q \in \Z[x]} q(x), 
		\]
		in which the former expresses that the characteristic is finite while the latter expresses that the field is algebraic over its prime subfield.  This is an atomic theory with enough points.  For each single element $a$ in an algebraically closed field $F$ algebraic over its prime subfield, the minimal formula of $a$ is precisely the conjunction of the minimal polynomial of $a$ with the formula $\top \vdash_{\emptyset} {1 + 1 \dots + 1} = 0$ expressing the characteristic of the field.  For a tuple $\vec{a} = (a_1, \, \dots \, , a_n)$ of $F$, the minimal formula of $\vec{a}$ is the formula
		\[
		\bigwedge_{i = 1}^n q_i(x_1, \, \dots \, , x_i) \land \underbrace{1 + 1 \dots + 1}_{p \text{ times}} = 0,
		\]
		where $q_i(a_1, \, \dots \, a_{i - 1}, x_i)$ is the minimal polynomial of $a_i$ over $F(a_1, \, \dots \, a_{i - 1} )$, and where the field $F$ has characteristic $p$ (cf.\ \cref{prop:AIform_is_repr}).

		\item\label{exs:enum:decinf} The theory $\Dtheory_\infty$ of \emph{infinite decidable objects} is also an atomic theory.  It is the single-sorted theory with one binary relation $\neq$ and the axioms $x = y \land x \neq y \vdash_{x,y} \bot$, $\top \vdash x = y \lor x \neq y$, and, for each $n \in \N$,
		\[\top \vdash_x \exists y_1, \, \dots \, , y_n \, \bigwedge_{i \leqslant n} x \neq y_i \, \land \! \bigwedge_{i < j \leqslant n} y_i \neq y_j.\]
		The minimal formula of a tuple $\vec{n}$ in a model is the finite conjunction of the atomic formulae $x_i = x_j$, when the elements $n_i, n_j \in \vec{n}$ are equal, and $x_i \neq x_j$ when they are not.

	\end{enumerate}

	In a similar fashion to \cref{exs:enum:decinf}, we can deduce that the theory of dense linear orders without endpoints, denoted by $\DLO$, the theory of atomless Boolean algebras, and the theory of the Rado graph are all also atomic theories.  A formal proof that these theories are atomic can be found in \cite{caramellofraisse,caramellogalois}.
	
\end{exs}

\begin{prop}\label{prop:disjoint_implies_atomic}
	Let $\X$ be model groupoid for $\theory$ that is conservative and eliminates parameters for an indexing such that the set of parameters used to index each model $M \in X_0$ are mutually disjoint.  Then $\theory$ is an atomic theory.
\end{prop}
\begin{proof}
	As each $M \in X_0$ is disjointly indexed from every other $N \in X_0$, the space $X_0^\taulogo$ is discrete.  Hence, by \cite[Lemma C3.5.3]{elephant}, $\Sh(X_0^\taulogo)$ is an atomic topos.  Recall from \cref{lem:forget_action_surj} that there is an open surjective geometric morphism $\Sh(X_0^\taulogo) \to \Sh(\Xlog) \simeq \topos_\theory$, and thus by applying \cite[Lemma C3.5.1]{elephant} we obtain the desired result.
\end{proof}

We now turn to the theory of \cite{caramellogalois} and consider the groupoid consisting of the automorphism group of a single model.  We note that for this groupoid there is essentially only one indexing in that parameters can be conflated with the element of the model they index.  Thus, we will assume that the automorphism group is trivially indexed.  We will show that, if $\theory$ is an atomic theory with enough points, then the automorphism group of a single model eliminates parameters if and only if that model is ultrahomogeneous.  Thus, we deduce the principal result of \cite{caramellogalois}.

\paragraph{Elimination of parameters implies ultrahomogeneity.} 

We first observe that elimination of parameters by the automorphism group of a single model implies ultrahomogeneity.  Recall that the model $M$ is \emph{ultrahomogeneous} if each finite partial isomorphism
\[
\begin{tikzcd}
	\vec{m} \ar{r}{\sim} \ar[hook]{d} & \vec{n} \ar[hook]{d} \\
	M & M
\end{tikzcd}
\]
can be extended to a total isomorphism $M \xrightarrow{\alpha} M$.

\begin{lem}\label{lem:elimparaimplieshomo}
	If $M$ is a model of an arbitrary geometric theory $\theory$ such that the group $\Aut(M)$ eliminates parameters, then $M$ is ultrahomogeneous.
\end{lem}
\begin{proof}
	For a fixed tuple $\vec{m} \in M$, by hypothesis there is a formula without parameters such that
	\[
	\overline{\class{\vec{x} = \vec{m}}}_{\Aut(M)} = \classv{\varphi}{x}_{\Aut(M)}.
	\]
	If there is a partial isomorphism $\vec{m} \xrightarrow{\sim} \vec{n}$, then $\vec{m}, \vec{n} \in \classv{\varphi}{x}_{\Aut(M)}$.  Therefore, $\vec{n} \in \overline{\class{\vec{x} = \vec{m}}}_{\Aut(M)}$, and so there exists an automorphism $M \xrightarrow{\alpha} M$ such that $\alpha(\vec{m}) = \vec{n}$.
\end{proof}

\paragraph{Ultrahomogeneity and atomicity imply elimination of parameters.} 

We now give the opposite implication under the assumption that $\theory$ is atomic.  Combining \cref{lem:elimparaimplieshomo} above and \cref{lem:atomicandhomoimplieselimpara} below, we recover the principal logical result of \cite{caramellogalois}.

\begin{lem}\label{lem:atomicandhomoimplieselimpara}
	Let $\theory$ be an atomic theory with enough points.  If $M$ is a ultrahomogeneous model, then $\Aut(M)$ eliminates parameters.
\end{lem}
\begin{proof}
	We claim that, for each tuple $\vec{m} \in M$, $\overline{\class{\vec{x} = \vec{m}}}_{\Aut(M)} = \classv{\chi_{\vec{m}}}{x}_{\Aut(M)}$, where $\chi_{\vec{m}}$ is the minimal formula of $\vec{m}$.  Since $M \vDash \chi_{\vec{m}}(\vec{m})$ by definition, one inclusion $\overline{\class{\vec{x} = \vec{m}}}_{\Aut(M)} \subseteq \classv{\chi_{\vec{m}}}{x}_{\Aut(M)}$ is immediate (see \cref{rem:df:conservandelimpara}\cref{enum:rem:df:elimpara:maxformula}).
	
	For the converse, if $\vec{m}' \in \classv{\chi_{\vec{m}}}{x}_{\Aut(M)}$, then $\vec{m},\vec{m}'$ have the same type and so there is a partial isomorphism $\vec{m} \xrightarrow{\sim} \vec{m}'$.  As $M$ is ultrahomogeneous, this extends to a total automorphism $\alpha \colon M \to M$ for which $\alpha(\vec{m}) = \vec{m}'$.  Hence, we obtain the converse inclusion
	\[
	\classv{\chi_{\vec{m}}}{x}_{\Aut(M)} \subseteq \overline{\class{\vec{x} = \vec{m}}}_{\Aut(M)}.
	\]		
	By \cref{rem:df:conservandelimpara}\cref{enum:rem:df:conservandelimpara:sufficetochecktuple}, this suffices to demonstrate that $\Aut(M)$ eliminates parameters.
\end{proof}

\begin{coro}[Theorem 3.1 \cite{caramellogalois}]\label{coro:topologicalgalois}
	Let $\theory$ be an atomic theory and let $G$ be a $T_0$ topological group.  There is an equivalence
	\[\topos_\theory \simeq \B G\]
	if and only if $G$ is the automorphism group of a conservative and ultrahomogeneous model $M$ of $\theory$.  Here, $G$ has been topologised with the \emph{Krull topology} (also called the pointwise convergence topology).  It is the coarsest topology making $G$ a topological group for which the subsets
	\[
	\lrset{M \xrightarrow{\alpha} M}{\alpha(\vec{n}) = \vec{n}},
	\]
	for every finite tuple $\vec{n} \in M$, form a basis of open neighbourhoods of the identity.
\end{coro}

\begin{exs}[\S 5 \cite{caramellogalois}]\label{exs:atomictheoriesrevisited}
	We revisit some of the examples of atomic theories from \cref{exs:atomictheories} and describe conservative ultrahomogeneous models for them.
	\begin{enumerate}[label = (\alph*)]
		\item\label{exs:atomictheoriesrevisited:enum:decidable_inf_objs} Any infinite set, in particular $\N$, is a conservative ultrahomogeneous model for the theory $\Dtheory_\infty$ of infinite decidable objects.  Therefore, $\Dtheory_\infty$ is classified by $\B \Aut(\N)$, the \emph{Schanuel topos}.
		
		\item The rationals $\Q$, with their usual ordering, is a conservative and ultrahomogeneous model for the theory $\DLO$.  Its classifying topos is thus $\B \Aut(\Q)$.
		
		\item Let $\mathcal{R}$ denote the Rado graph.  It is a conservative and ultrahomogeneous model for its namesake theory, which is therefore classified by $\B \Aut(\mathcal{R})$.
	\end{enumerate}
\end{exs}

Recall from \cite[Theorem 3.16]{caramelloatomic} that if an atomic theory $\theory$ is also \emph{complete}, by which we mean that any model is conservative (or equivalently, by \cite[Proposition 3.9]{caramelloatomic}, for every sentence $\varphi$ either $\theory$ proves $\top \vdash \varphi$ or $\theory$ proves $\varphi \vdash \bot$), then the theory $\theory$ is also \emph{countably categorical}, i.e.\ any two countable models of $\theory$ are isomorphic.  Therefore, it suffices in \cref{coro:topologicalgalois} to take $M$ as the unique countable model of the theory -- if this exists -- since this will automatically also be an ultrahomogeneous model (see \cite[\S 10.1]{hodges}).

\begin{ex}\label{ex:atomic_but_not_ultra}
	Let $\theory$ be an atomic theory and let $M$ be a conservative model of $\theory$.  In order to assure the reader that ultrahomogeneity is a non-trivial requirement on $M$, despite the hefty conditions placed on $\theory$ by being an atomic theory, we describe a conservative model for the theory $\DLO$ which is not ultrahomogeneous.  Let $\R$ denote the real numbers with the usual ordering, and let $\R + \R$ denote the model whose underlying set is 
	\[\{\,1\,\} \times \R \cup \{\,2\,\} \times \R\]
	given the lexicographic ordering.  This is a conservative model of the theory of $\DLO$ (since this is a complete theory).  However, it is not ultrahomogeneous.  
	
	We note that, for any $r \in \R$, the partial isomorphism $(1,r) \mapsto (2,r)$ cannot be extended to a total automorphism of $\R + \R$.  If there did exist such a total automorphism of $\R + \R$, then the subset $\{\,1\,\} \times (-\infty,r) \subseteq \R + \R$,	being the down-segment of $(1,r)$, would necessarily be mapped isomorphically to the subset $\{\,1\,\} \times \R \cup \{\,2\,\} \times (-\infty,r) \subseteq \R + \R$, the down-segment of $(2,r)$.  However, the interval $(-\infty,r) \cong \{\,1\,\} \times (-\infty,r)$ is \emph{Dedekind-complete}, meaning that every subset of $(-\infty,r)$ with an upper bound has a least upper bound, while $\{\,1\,\} \times \R \cup \{\,2\,\} \times (-\infty,r)$ is not Dedekind-complete -- the subset $\{\,1\,\} \times \R \subseteq \{\,1\,\} \times \R \cup \{\,2\,\} \times (-\infty,r)$ does not have a least upper bound.
	
	This one model therefore serves as a counterexample to several natural questions arising from the study of representing groupoids for toposes.
	\begin{enumerate}
		\item Recall from \cite{JT} and \cite{dubuc} that a connected atomic topos is represented by the \emph{localic} automorphism group of any of its points.  This is because any point of a connected atomic topos is an open surjection (see \cite[Proposition VII.4.1]{JT}).  Being a complete atomic theory, the classifying topos $\DLO$ is a connected atomic topos.  Thus, the theory is represented by the {localic} group of automorphisms $\Aut(\R + \R)^{\rm loc}$ but not by the topological group of automorphisms of $\R + \R$ (see \cite{caramellogalois} or \cref{coro:topologicalgalois} above).
		
		The discrepancy occurs because the localic automorphism group $\Aut(\R + \R)^{\rm loc}$ constructed in \cite[Proposition 4.7]{dubuc} is not spatial.  The underlying locale $\Aut(\R + \R)^{\rm loc}$ of the localic automorphism group can be described as the classifying locale for the following geometric propositional theory.
		\begin{enumerate}
			\item For each pair of elements $x,y \in \R + \R$, we add a pair of basic propositions $[\alpha(x) = y]_{\R + \R}$ and $[x < y]_{\R + \R}$.

			\item For every pair $x,y \in \R + \R$ with $x < y$, we add the sequent $\top \vdash [x < y]_{\R + \R} $ as an axiom to the propositional theory, and for every quadruple $x,x',y,y' \in \R + \R$ where $x \neq x'$ and $y \neq y'$, we also add to our axioms the sequents
			\[
				[\alpha(x) = y]_{\R + \R} \, \land \, [\alpha(x') = y]_{\R + \R}  \vdash \bot, \ \ 
				[\alpha(x) = y]_{\R + \R} \, \land \, [\alpha(x) = y']_{\R + \R} \vdash \bot, \ \ 
				\top  \vdash \!\!\!\! \bigvee_{x \in \R + \R} \!\!\!\! [\alpha(x) = y]_{\R + \R},
			\]
			expressing that the symbol $\alpha$ encodes a bijection from $\R + \R$ to itself.  Additionally, we include the bidirectional sequent
			\[
			[\alpha(x) = y]_{\R + \R} \,\land\, [ \alpha(x') = y']_{\R + \R} \,\land\, [x < x']_{\R + \R} \dashv  \vdash [\alpha(x) = y]_{\R + \R} \,\land\, [\alpha(x') = y']_{\R + \R}\, \land \,[y < y']_{\R + \R},
			\]
			as an axiom, expressing that $\alpha$ encodes an automorphism of the linear order $\R + \R$.				
		\end{enumerate}

		A point $\alpha \colon \2 \to \Aut(\R + \R)^{\rm loc}$ of this locale evidently corresponds to an automorphism
		of the model $\R + \R$.  We therefore deduce by the above analysis that the non-trivial open of $\Aut(\R + \R)^{\rm loc}$ corresponding to the basic proposition $[(1,r) \mapsto (2,r)]_{\R + \R}$ is evaluated by $\alpha^{-1}$ as $
		\alpha^{-1}\left([(1,r) \mapsto (2,r)]_{\R + \R}\right) = \bot$.  Hence, $\Aut(\R + \R)^{\rm loc}$ is not a spatial locale.
		
		In summary, we have that $\topos_\DLO \simeq \B\Aut(\R + \R)^{\rm loc} \not \simeq \B\Aut(\R + \R)$.  It is then natural to wonder: if not $\DLO$, what theory is classified by $\B \Aut(\R + \R)$?  Such a theory is described in \cref{ex:theory_R+R} as an application of the techniques exposited in \cref{sec:theory_of_groupoid}.

		\item Let $\theory'$ be a theory of \emph{presheaf type}, i.e.\ the classifying topos of $\theory'$ is a presheaf topos.  By \cite[\S 6.1.1]{TST}, this presheaf topos can be chosen to be the topos $\sets^{{\rm f.p.}\Tmodels{\sets}\op}$, where ${\rm f.p.}\Tmodels{\sets}$ denotes the category of \emph{finitely presented models} (see \cite[Definition 6.1.11]{TST}).  We will say, in accordance with \cite[Definition 2.3(a)]{caramellofraisse}, that a model $M'$ of $\theory'$ is \emph{homogeneous with respect to the finitely presented models}, or \emph{f.p.-homogeneous} for short, if for every pair of finitely presented models $N,N'$ and homomorphisms $f \colon N \to M'$ and $g \colon N \to N'$ of $\theory'$-models, there exists a homomorphism of $\theory'$-models $h$ such that the triangle
		\[
		\begin{tikzcd}
			N \ar{r}{f} \ar{d}[']{g} & M' \\
			N' \ar[dashed]{ru}[']{h}
		\end{tikzcd}
		\]
		commutes.  Every ultrahomogeneous model is f.p.-homogeneous (see \cite[Remark 2.4(a)]{caramellofraisse}).
		
		The theory of linear orders is a theory of presheaf type by \cite[\S VIII.8]{SGL}, and the finitely presented linear orders are simply the finite linear orders.  Moreover, the dense linear orders without endpoints are precisely those linear orders that are f.p.-homogeneous (see \cite[Remark 3.8(a)]{caramellogalois}).  Thus, $\R + \R$ is an example of a f.p.-homogeneous model that is not ultrahomogeneous.
	\end{enumerate}
\end{ex}

\paragraph{Boolean toposes with enough points.}

Recall that a topos with enough points is Boolean if and only if it is atomic.  Extending \cref{coro:topologicalgalois}, we can use the classification theorem to characterise Boolean toposes with enough points.  In \cite{blassscedrov}, Blass and \v{S}\v{c}edrov gave the special case for Boolean coherent toposes of our characterisation, stated in \cref{coro:booltopos} below.  Our construction is reminiscent of their argument.

We first require one lemma on the \emph{quotient theories} of a theory classified by a Boolean topos.  Recall that a quotient theory $\theory'$ of $\theory$ is a theory over the same signature whose axioms include the axioms of $\theory$.

\begin{lem}\label{lem:quotientofbool}
	If $\theory$ is a geometric theory whose classifying topos is Boolean, then every \emph{quotient theory} of $\theory$
	is determined by the addition of a single extra sentence $\top \vdash_{\emptyset} \varphi$ as an axiom.
\end{lem}
\begin{proof}
	There are two ways to see this: one topos-theoretic, and one syntactic.  We include both, though of course they are merely translations of one another.
	
	A quotient theory $\theory'$ of $\theory$ is classified by a subtopos of $\topos_\theory$.  Every subtopos of a Boolean topos is \emph{open} by \cite[Corollary 3.5]{open}.  Therefore, being an open subtopos of $\topos_\theory$, the topos $\topos_{\theory'}$ corresponds to a subterminal in $\topos_\theory$, i.e.\ a sentence $\{\,\emptyset:\varphi\,\}$.
	
	Alternatively, if $\topos_\theory$ is a Boolean topos, then we recall from \cite[\S D3.4]{elephant} that every formula of \emph{infinitary first order logic} (i.e.\ including negation $\neg$, implication $\to$, universal quantification $\forall$, and infinitary conjunction $\bigwedge$) is $\theory$-provably equivalent to a geometric formula.  We are therefore free to manipulate geometric sequents as though they existed in infinitary first order logic.  Hence, we easily recognise that any quotient theory $\theory' = \theory \cup \lrset{\varphi_i \vdash_{\vec{x}_i} \psi_i}{i \in I}$ of $\theory$ is Morita equivalent to the theory
	\[\theory' \equiv  \theory \cup \lrset{\top \vdash_{\emptyset} \forall \vec{x}_i \, \varphi_i \to \psi_i }{i \in I} \equiv \theory \cup \left\{ \,  \top \vdash_{\emptyset} \bigwedge_{i \in I} \forall \vec{x}_i \, \varphi_i \to \psi_i\,\right\}.\]
\end{proof}

\begin{coro}[cf.\ Theorem 2 \cite{blassscedrov}]\label{coro:booltopos}
	A topos $\topos$ with enough points is Boolean if and only if there is a set of topological groups $\{\,G_x \mid x \in X_0 \,\}$ such that
	\[ \topos \simeq \coprod_{x \in X_0} \B G_x.\]
\end{coro}
\begin{proof}
	Firstly, we recognise that a topos of the form $\coprod_{x \in X_0} \B G_x$ is really just the topos of sheaves for the topological groupoid
	\[
	\G = \left(
	\coprod_{x \in X_0}  G_x \rightrightarrows \coprod_{x \in X_0} 1 \cong X_0^\delta
	\right).
	\]
	This topological groupoid is, of course, automatically open.  We could call groupoids of this form the \emph{bouquet} groupoids since, when written out diagrammatically, they appear as a collection of `flowers' -- the group elements $g \in G_x$ being the `petals', e.g.\
	\[\begin{tikzcd}
		G_x \arrow[loop, in=120, out=60, distance = 30pt] 
		\arrow[loop, in=235, out=175, distance = 30pt]
		\arrow[loop, in=5, out=305, distance = 30pt]
		\arrow["\dots", phantom, loop, distance= 15pt, in=305, out=235]
		& \dots & G_y.
		\arrow[loop, in=120, out=60, distance = 30pt] 
		\arrow[loop, in=235, out=175, distance = 30pt]
		\arrow[loop, in=5, out=305, distance = 30pt]
		\arrow["\dots", phantom, loop, distance= 15pt, in=305, out=235]
	\end{tikzcd}\]
	By \cref{prop:disjoint_implies_atomic}, if $\topos \simeq \coprod_{x \in X_0} \B G_x$ for some set of groups $\{\,G_x \mid x \in X_0\,\}$, then $\topos$ is Boolean.
	
	For the converse direction, let $\theory$ be a geometric theory, over a signature $\Sigma$, classified by the topos $\topos$.  By the hypotheses, $\theory$ is an atomic theory with enough points.  We can therefore find a conservative set of models $X_0$ for $\theory$.  In fact we can choose each model $M \in X_0$ to be ultrahomogeneous since, via a standard result in model theory, every model is an elementary substructure of a ultrahomogeneous model (see \cite[\S 10.2]{hodges}).  Moreover, we can evidently choose the models $M \in X_0$ to be pairwise elementarily inequivalent (i.e.\ no two models satisfy all the same sentences) and also pairwise disjoint.  Therefore, by \cref{lem:atomicandhomoimplieselimpara} we deduce that each automorphism group $\Aut(M)$ eliminates parameters.
	
	Thus, when given the trivial indexing the model groupoid $\coprod_{M \in X_0} \Aut(M)$, obtained by taking as objects the models $M \in X_0$ and as arrows all automorphisms, also eliminates parameters.  To see this, we first note that for each definable with parameters $\class{\vec{x},\vec{m}:\psi}_{\coprod_{M \in X_0} \Aut(M)}$, the parameters $\vec{m}$ are only instantiated in one model $M \in X_0$ (since the models of $X_0$ were chosen to be pairwise disjoint).  Therefore, since there are no arrows between distinct models of our groupoid, (modulo a transparent abuse of notation) we have that
	\[
	\overline{\class{\vec{x},\vec{m}:\psi}}_{\coprod_{M \in X_0} \Aut(M)} = \overline{\class{\vec{x},\vec{m}:\psi}}_{\Aut(M)}.
	\]
	Let $\varphi$ be a formula without parameters such that $\overline{\class{\vec{x},\vec{m}:\psi}}_{\Aut(M) } = \classv{\varphi}{x}_{\Aut(M)}$.
	
	We are not quite done since $ \classv{\varphi}{x}_{\Aut(M)} \neq \classv{\varphi}{x}_{\coprod_{M \in X_0}\Aut(M)} $.  Instead, we must find a formula that isolates those realisations of $\varphi$ in $M$ from those in other models $M' \in X_0$.  This is achieved by \cref{lem:quotientofbool}.  Let $\theory_M$ denote the \emph{theory of the model} $M$, i.e the set of geometric sequents
	\[
	\theory_M = \lrset{\chi \vdash_{\vec{x}} \xi}{\classv{\chi}{x}_M \subseteq \classv{\xi}{x}_M}.
	\]
	This is evidently a quotient theory of $\theory$.  Thus, by \cref{lem:quotientofbool}, there exists a sentence $\xi_M$ such that $M$ is the only model in $X_0$ which satisfies $\xi_M$ -- since the models of $X_0$ were chosen to be pairwise elementarily inequivalent.  Therefore, we have that
	\[\overline{\class{\vec{x},\vec{m}:\psi}}_{\coprod_{M \in X_0} \Aut(M)} = \classv{\varphi \land \xi_M}{x}_{\coprod_{M \in X_0}\Aut(M)}\]
	as required.  Thus, $\coprod_{M \in X_0} \Aut(M)$ is a conservative model groupoid for $\theory$ that eliminates parameters and so, by \cref{thm:elimparaimpliesrepr}, we conclude that $\topos \simeq \coprod_{M \in X_0} \B \Aut(M)$, once each automorphism group has been suitably topologised.
\end{proof}

\begin{ex}[Proposition 2.4 \cite{caramellojohnstone}]
	We return to the theory of algebraically closed fields of finite characteristic that are algebraic over their prime subfield from \cref{exs:atomictheories}\cref{exs:enum:acf}.  Being an atomic theory, by \cref{coro:booltopos} we know it can be presented as a coproduct of toposes of actions by topological groups.  Indeed, the theory is classified by the topos
	\[
	\coprod_{p \text{ prime}} \B \Aut\left(\overline{\Z/\langle p \rangle}\right),
	\]
	where $\overline{\Z/\langle p \rangle}$ is the algebraic closure of $\Z/\langle p \rangle$, and $\Aut\left(\overline{\Z/\langle p \rangle}\right)$ has been topologised with the usual Krull topology.  This is precisely \cite[Proposition 2.4]{caramellojohnstone}.
\end{ex}

The principal result of \cite{blassscedrov} classifies Boolean coherent toposes.  As coherent toposes have enough points, the result can be obtained from \cref{coro:booltopos} by discerning when the topos $\coprod_{x \in X_0} \B G_x$ is coherent.  This occurs when $X_0$ is a finite set and each $G_x$ is a \emph{coherent} topological group.  We refer to \cite{blassscedrov} for the details.

%%%%%%%%%%%%%%%%%%%%%%%%%%%%%%%%%%%%%%%%%%%%
%%%%%%%%%%%%%%%%%%%%%%%%%%%%%%%%%%%%%%%%%%%%%

\subsection{Decidable theories}\label{subsec:decidable}

We saw in \cref{prop:disjoint_implies_atomic} that we cannot, in general, require that the models in a representing model groupoid are disjointly indexed.  In this subsection, we demonstrate further that nor can we remove the requirement that multiple parameters may index the same element of a model, i.e.\ that each model is presented as a subquotient of its set of parameters rather than a subset.  We will observe that this is possible only if the theory is decidable.

\begin{df}
	A geometric theory $\theory$ over a signature $\Sigma$ is \emph{decidable} if, for each pair of free variables $x,x'$ of the same sort of $\Sigma$, there is a formula in context $x,x'$, which we suggestively denote as $x \neq x'$, such that $\theory$ proves the sequents
	\[
	x = x' \land x \neq x' \vdash_{x,x'} \bot , \ \ \ \top \vdash_{x,x'} x= x' \lor x \neq x'.
	\]
\end{df}

Let $\X = (X_1 \rightrightarrows X_0)$ be a model groupoid for a geometric theory $\theory$.  Consider the trivial indexing of $\X$, described in \cref{ex:indexedstructures}\cref{ex:indexedstructures:enum:trivial}, with parameters as the elements of the constituent models $\bigcup_{M \in X_0} M$.  For this indexing, each element $n \in M \in X_0$ is indexed by precisely one parameter.

It is not hard to see that, up to isomorphism, this is the unique indexing of $\X$ with this property.  Suppose that $\X$ is indexed by a set of parameters $\Index$ in such a way that, for each $n \in M \in X_0$, there is a unique parameter $m \in \Index$ that indexes $n$.  This is equivalent to presenting each model $M \in X_0$ as a subset of the set of parameters, rather than a subquotient.  By replacing the underlying set of $M$ with the corresponding subset of parameters, we have now trivially indexed $\X$.

\begin{prop}\label{prop:subsetsimpliesdecidable}
	Let $\X = (X_1 \rightrightarrows X_0)$ be a conservative model groupoid for a theory $\theory$.  If, when $\X$ is given the trivial indexing, $\X$ eliminates parameters, then $\theory$ is a decidable theory.
\end{prop}
\begin{proof}
	We must show that, for each pair of free variables $x,x'$ of the same sort, the formula $x=x'$ has a complement modulo the theory $\theory$.  As $\X$ is conservative and eliminates parameters it is a representing model groupoid by \cref{maintheorem}, and so by the isomorphism 
	\[\Sub_{\topos_\theory}(\{\,x,x':\top\,\}) \cong \Sub_{\Sh\left(\Xlog\right)}(\class{x,x':\top}_\X),\]
	finding a complement for the formula $x = x'$ is equivalent to showing that $\class{x=x'}_\X$ has a complement.  We claim that this complement, which we denote by $\class{x\neq x'}_\X$, is given by
	\[
	\bigcup_{\substack{m,m' \in \bigcup_{M \in X_0} M \\ m \neq m'}} \class{x = m \land x' = m'}_\X .
	\]
	We must first show that $\class{x\neq x'}_\X$ does indeed define a subobject, i.e.\ a stable open subset, of $ \class{x,x':\top}_\X $.

	The subset $\class{x\neq x'}_\X$ is the union of opens, and therefore open itself.  Now suppose we are given an element $\lrangle{n,n',M} \in \class{x \neq x'}_\X$.  By the definition of $\class{x \neq x'}_\X$, $n \neq n'$ and so $\alpha(n) \neq \alpha(n')$ for any isomorphism $M \xrightarrow{\alpha} N$.  Therefore, $\lrangle{\alpha(n,n'),N} $ is also an element of $\class{x\neq x'}_\X$.  Thus, $\class{x\neq x'}_\X$ is a stable open as desired.
	
	It is now not hard to conclude that $\class{x \neq x'}_\X$ is a complement to $\class{x=x'}_\X$.  Given $\lrangle{n,n',M} \in \class{x,x':\top}_\X$, either $n=n'$ or $n \neq n'$, yielding 
	\[\class{{x}={x}'}_\X \cup \class{{x} \neq {x}'}_\X = \class{{x},{x}':\top}_\X.\]
	It is similarly easy to conclude that $\class{{x}={x}'}_\X \cap \class{{x} \neq {x}'}_\X = \emptyset$.
\end{proof}

%Since not every geometric theory is decidable, we cannot in general require that each element of a model groupoid is indexed by a unique parameter.  Even when our theory is decidable, we must allow for the models of our groupoid to share elements, i.e.\ parameters.

\paragraph{Examples of representing groupoids for decidable theories.}

Finally, we present a useful result from which it is possible to easily generate representing groupoids for many theories.  Note that the only theories that can be represented by the groupoids generated using the below method must also be decidable.  This can be seen by an application of \cref{prop:subsetsimpliesdecidable}.

\begin{prop}\label{prop:generate_decidable_grpds}
	Let $\theory$ be geometric theory over a signature $\Sigma$.  Suppose that a conservative set of models for $\theory$ can be found as substructures of an ultrahomogeneous $\Sigma$-structure $U$ whose theory $\theory_U$ (i.e.\ the theory over $\Sigma$ whose axioms are all sequents satisfied by $U$) is atomic, and moreover the minimal formula $\chi_{\vec{n}}$ of any tuple of elements $\vec{n} \in U$ is quantifier free.  Then the model groupoid $\Sub_\theory(U)$ of $\theory$,
	\begin{enumerate}
		\item whose objects are the substructures of $U$ that are models of $\theory$,
		\item and whose arrows are all isomorphisms between these,
	\end{enumerate}
	is a representing model groupoid for $\theory$, i.e.\ there is an indexing of $\Sub_\theory(U)$ such that
	\[\topos_\theory \simeq \Sh\left(\Sub_\theory(U)^\tauloga_\taulogo\right).\]
\end{prop}
\begin{proof}
	By hypothesis, the model groupoid $\Sub_\theory(U)$ is conservative.  It remains to show that the the groupoid $\Sub_\theory(U)$ has an indexing by a set of parameters for which the groupoid eliminates parameters.  The indexing set we use are the elements of the $\Sigma$-structure $U$.  The indexing of a model $M \in \Sub_\theory(U)$ is determined by the inclusion $M \subseteq U$ of $M$ as a substructure of $U$.
	
	Since $U$ is an ultrahomogeneous model of the atomic theory $\theory_U$, by \cref{lem:atomicandhomoimplieselimpara} we have that, for each tuple $\vec{m} \in U$,
	\[
	\overline{\class{\vec{x} = \vec{m}}}_{\Aut(U)} = \classv{\chi_{\vec{m}}}{x}_{\Aut(U)},
	\]
	where $\chi_{\vec{m}}$ is the minimal formula of $\vec{m} \in U$.  
	
	We claim that $\overline{\class{\vec{x} = \vec{m}}}_{\Sub_\theory(U)} = \classv{\chi_{\vec{m}}}{x}_{\Sub_\theory(U)}$.  Since $\chi_{\vec{m}}$ is a quantifier free formula, $M \vDash \chi_{\vec{m}}(\vec{m})$ for any $\theory$-model $M \subseteq U$ that contains $\vec{m}$.  Thus, the first inclusion $\overline{\class{\vec{x} = \vec{m}}}_{\Sub_\theory(U)} \subseteq \classv{\chi_{\vec{m}}}{x}_{\Sub_\theory(U)}$ follows from \cref{rem:df:conservandelimpara}\cref{enum:rem:df:elimpara:maxformula}.  Conversely, given another witness $\vec{n}$ of $\chi_{\vec{m}}$ in a $\theory$-model $N \subseteq U$, there exists, by the ultrahomogeneity of $U$, an automorphism $\alpha$ of $U$ that sends $\vec{m} $ to $\vec{n}$.  We have that $\alpha^{-1}(N),N \subseteq U$ constitute a pair $\theory$-models, and $\alpha|_{\alpha^{-1}(N)} \colon \alpha^{-1}(N) \to N
	$ is an isomorphism of $\Sigma$-structures that sends $\vec{m} \in \alpha^{-1}(N)$ to $\vec{n} \in N$.  Therefore, we have that $\lrangle{\vec{n},N} \in \overline{ \class{\vec{x} = \vec{m}}}_{\Sub_\theory(U)}$, completing the reverse inclusion.
\end{proof}

\begin{exs}\label{ex:groupoids_for_decidables}
	We apply \cref{prop:generate_decidable_grpds} to give a pair of simple examples of model groupoids for decidable theories.
	\begin{enumerate}
		\item(\S 2.4.1 \cite{awodeyforssell}) The theory of \emph{decidable objects} is the single-sorted theory with one binary predicate $\neq$ and the axioms
		\[
		x = x' \land x \neq x' \vdash_{x,x'} \bot , \ \ \ \top \vdash_{x,x'} x= x' \lor x \neq x'.
		\]
		As observed in \cref{exs:atomictheoriesrevisited}\cref{exs:atomictheoriesrevisited:enum:decidable_inf_objs}, the natural numbers is an ultrahomogeneous model of the theory of decidable objects.  The theory of $\N$ is the theory of infinite decidable objects $\Dtheory_\infty$ from \cref{exs:atomictheories}\cref{exs:enum:decinf}, an atomic theory whose minimal formulae are quantifier free.
		
		Moreover, the subsets of $\N$ are a conservative set of models for the theory of decidable objects.  Hence, by an application of \cref{prop:generate_decidable_grpds}, the theory of decidable objects is classified by
		\[
		\Sh\left(\Sub(\N)^{\tauloga}_\taulogo\right).
		\]

		\item\label{ex:enum:alg_ext} Let $K$ be a field.  We denote by $\theory_{(-/K)}$ the theory of \emph{algebraic extensions} of $K$.  This is the single-sorted theory over the signature consisting of the standard signature of a ring with an additional constant symbol for each element of $K$.  The axioms of $\theory_{(-/K)}$ consist of the following:
		\begin{enumerate}
			\item the standard axioms of a field and an axiom $\top \vdash \varphi(\vec{k})$
			for each sentence $\varphi$ with constants $\vec{k} \in K$ satisfied by $K$, ensuring that each model of $\theory_{(-/K)}$ is a field extension of $K$,
			\item and the sequent
			\[
			\top \vdash_x \bigvee_{q \in K[x]} q(x) = 0,
			\]
			expressing that any model is an algebraic extension of $K$.
		\end{enumerate}
		The algebraic closure $\overline{K}$ of $K$ is an ultrahomogeneous structure whose theory $\Th\left(\overline{K}\right)$ is atomic and whose minimal formulae are quantifier free (cf.\ \cref{exs:atomictheories}\cref{exs:enum:acf}).  Hence, by an application of \cref{prop:generate_decidable_grpds}, the theory $\theory_{(-/K)}$ is classified by the topos
		\[
		\Sh\left(\Sub\left(\overline{K}\right)^\tauloga_\taulogo\right),
		\]
		where $\Sub\left(\overline{K}\right)$ is the groupoid of intermediate extensions of $K$, and all isomorphisms between these.
		
		By \cref{prop:subsetsimpliesdecidable}, the theory $\theory_{(-/K)}$ is decidable.  Indeed, we can identify the complement of the equality predicate as the formula $\exists y \ y \cdot (x-x') = 1$.
	\end{enumerate}
\end{exs}

%%%%%%%%%%%%%%%%%%%%%%%%%%%%%%%%%%%%%%%%%%%%%%

%%%%%%%%%%%%%%%%%%%%%%%%%%%%%%%%%%%%%%%%%%%%%%%%%%
%%%%%%%%%%%%%%%%%%%%%%%%%%%%%%%%%%%%%%%%%%%%%%%%%%%
%%%%%%%%%%%%%%%%%%%%%%%%%%%%%%%%%%%%%%%%%%%%%%%%%%%%%%

\subsection{\'{E}tale complete groupoids}\label{sec:etalecomplete}

We now study the behaviour of a representing model groupoid when we expand its set of arrows, inspired by the consequence of the descent theory of Joyal and Tierney \cite[Theorem VIII.3.2]{JT} that every open localic groupoid is \emph{Morita-equivalent} to its \emph{étale completion} (see \cite[Definition 7.2]{cont1}).  Loosely speaking, a localic groupoid is étale complete if every isomorphism between points in its space of objects, viewed as points of the topos of equivariant sheaves, is instantiated in its space of arrows (this can be made precise by speaking of \emph{generalised} points).  We will study the analogous topological definition.

\begin{dfs}[cf.\ Definition 7.2 \cite{cont1}]\label{df:etale-complete}
	Let $\theory$ be a geometric theory and let $\X = (X_1 \rightrightarrows X_0)$ be a model groupoid for $\theory$.
	\begin{enumerate}
		\item\label{enum:df:etale-complete:ec} The groupoid $\X$ is said to be \emph{étale complete} if every $\theory$-model isomorphism between models $M, N \in X_0$ is instantiated in $X_1$.  
		\item We denote by $\hat{\X}$ the \emph{étale completion} of $\X$.  This is the model groupoid $\hat{\X} = (\hat{X}_1 \rightrightarrows X_0)$ whose set of objects is the same set of models $X_0$ as for $\X$, but whose arrows are all $\theory$-model isomorphisms between models $M, N \in X_0$.
	\end{enumerate} 
\end{dfs}

\begin{rem}
	Let $\X$ be a model groupoid of a theory $\theory$.  Conflating functions $f \colon Y \to Z$ with their induced geometric morphisms $\sets/Y \to \sets/X$, \cref{df:etale-complete}\cref{enum:df:etale-complete:ec} expresses that for every natural isomorphism
	\[\begin{tikzcd}
		\sets & {\sets/X_0} \\
		{\sets/X_0} & {\topos_\theory},
		\arrow["M", from=1-1, to=1-2]
		\arrow["N"', from=1-1, to=2-1]
		\arrow["\alpha"', Rightarrow, from=1-2, to=2-1]
		\arrow["\sim"{marking, allow upside down}, shift left=2, draw=none, from=1-2, to=2-1]
		\arrow["\p", from=1-2, to=2-2]
		\arrow["\p"', from=2-1, to=2-2]
	\end{tikzcd}\]
	there is a point $\breve{\alpha} \in X_1$ such that $\alpha$ is the composite 2-cell $\p \circ M \Rightarrow \p \circ N$ in the diagram
	\[\begin{tikzcd}
		\sets \\
		& {\sets/X_1} & {\sets/X_0} \\
		& {\sets/X_0} & {\topos_\theory},
		\arrow["{\breve{\alpha}}", from=1-1, to=2-2]
		\arrow["M", curve={height=-20pt}, from=1-1, to=2-3]
		\arrow["N"', curve={height=20pt}, from=1-1, to=3-2]
		\arrow["s", from=2-2, to=2-3]
		\arrow["t"', from=2-2, to=3-2]
		\arrow[Rightarrow, from=2-3, to=3-2]
		\arrow["\sim"{marking, allow upside down}, shift left=2, draw=none, from=2-3, to=3-2]
		\arrow["\p", from=2-3, to=3-3]
		\arrow["\p"', from=3-2, to=3-3]
	\end{tikzcd}\]
	where the isomorphism $\p \circ s \cong \p \circ t$ is the canonical isomorphism deriving from the fact that $\X$ is a groupoid.  (The analogy with the `bi-pullback' definition of \'etale completeness exposited in \cite[\S 7]{cont1} is now limpid.)
\end{rem}

So far in \cref{sec:atomic} and \cref{subsec:decidable}, the only specific examples we have considered are all étale complete model groupoids.  Our classification given in \cref{maintheorem} is powerful enough to also recognise a representing model groupoid even when it is not étale complete.  We give an example of a representing model groupoid that is étale incomplete in \cref{ex:notallisos}.  However, we will observe that, just as is the case for localic groupoids, every open topological groupoid is \emph{Morita-equivalent} to its étale completion, by which we mean that the toposes of equivariant sheaves are equivalent.

\begin{ex}\label{ex:notallisos}
	Let $\theory$ be an atomic theory and $M$ a ultrahomogeneous and conservative model of $\theory$.  We note that we do not require all the automorphisms of $M$ in order to extend every possible partial isomorphism of finite substructures.  Therefore, if we were to take a non-trivial subgroup of $\Aut(M)$ that contains all those automorphisms induced by such partial isomorphisms, we would still be able to use the ultrahomogeneity property that was so crucial in proving that $\Aut(M)$ eliminates imaginaries in \cref{lem:atomicandhomoimplieselimpara}.  We manufacture such an example below.

	We once again consider the theory of dense linear orders without endpoints $\DLO$.  Recall from \cref{exs:atomictheoriesrevisited} that this theory is represented by the automorphism group $\Aut(\Q)$.  We show that we can take a (topologically dense) subgroup $\X$ of $\Aut(\Q)$ which does not contain all automorphisms, and yet $\X$ eliminates parameters and hence is a representing group.
	
	We note that, for any rational number $r \in \Q$, the map $p \mapsto p + r$ is an automorphism of $\Q$.  We will say that an automorphism $\alpha \colon \Q \to \Q$ is \emph{boundedly additive} if, apart from a bounded interval, $\alpha$ is given by addition.  Explicitly, $\alpha$ is boundedly additive if there are bounded (closed) intervals $[q_1,r_1] \subseteq \Q$ and $[q_2,r_2] \subseteq \Q$ such that:
	\begin{enumerate}
		\item firstly, $\alpha$ maps $[q_1,r_1]$ to $[q_2,r_2]$,
		\item on the interval $(-\infty,q_1)$, $\alpha$ acts by $p \mapsto p + q_2 - q_1$,
		\item and on the interval $(r_1, \infty)$, $\alpha$ acts by $p \mapsto p + r_2 - r_1$.
	\end{enumerate} 
	The identity is clearly boundedly additive, and if $\alpha$ and $\gamma$ are boundedly additive, then by choosing a sufficiently large interval we can ensure that their composite $\alpha \circ \gamma$ is boundedly additive too.  Let $\X$ denote the subgroup of $\Aut(\Q)$ of boundedly additive automorphisms.
	
	We claim that for any tuple $\vec{q}_1 \in \Q$, we have that $\overline{\class{\vec{y} = \vec{q}_1}}_\X = \classv{\chi_{\vec{q}_1}}{y}_\X$, where $\chi_{\vec{q}_1}$ is the minimal formula of $\vec{q}_1$, and thus that $\X$ eliminates parameters.  We automatically have one inclusion $\overline{\class{\vec{y} = {\vec{q}}_1}}_\X \subseteq \classv{\chi_{{\vec{q}}_1}}{y}_\X$.  For the converse, we must show that for any other tuple $\vec{q}_2 \in \Q$ with the same order type as $\vec{q}_1$, there is a boundedly additive automorphism $\alpha \colon \Q \to \Q$ that maps $\vec{q}_1$ onto $\vec{q}_2$.  This is straightforward.  Let $q_1$ and $r_1$ denote, respectively, the least and greatest elements of $\vec{q}_1$, and similarly define $q_2, r_2$ for $\vec{q}_2$.  Using the standard back-and-forth methods one uses to show that $\Q$ is ultrahomogeneous (see \cite[\S 3.2]{hodges}), we can construct an order isomorphism $[q_1,r_1] \cong [q_2,r_2]$ that maps $\vec{q}_1$ onto $\vec{q}_2$.  It is now clear that this can be extended to a total and boundedly additive automorphism of $\Q$.  Thus, $\X$ is an automorphism subgroup on a conservative model that eliminates parameters, and hence a representing group of the theory.
	
	The subgroup $\X$ is not the whole group $\Aut(\Q)$.  An example of an automorphism of $\Q$ that is not boundedly additive can be constructed out of one which is.  Firstly we note that $\Q$ is order isomorphic to countably many copies of itself given the lexicographic ordering since, given some irrational $a$, 
	\[\Q \cong \bigcup_{n \in \Z} (a+n,a+n+1) \cong \coprod_{\omega_0} \Q.\]  
	Let $\alpha$ be a boundedly additive automorphism whose non-additive part $[q_1,r_1] \cong [q_2,r_2]$ is non-additive (the automorphism $\alpha$ could be, for example, the total automorphism induced by the partial isomorphism $1<2<4 \mapsto 1<3<4$).  An automorphism of $\Q$ which is not boundedly additive is now obtained via the composite
	\[
	\begin{tikzcd}
		{\displaystyle \Q \cong \coprod_{\omega_0} \Q} \ar{rr}{\coprod_{\omega_0} \alpha} && {\displaystyle \coprod_{\omega_0} \Q \cong \Q.}
	\end{tikzcd}
	\]
\end{ex}

\begin{prop}\label{prop:canaddisos}
	Let $\X = (X_1 \rightrightarrows X_0)$ be a model groupoid for a geometric theory $\theory$, indexed by a set of parameters $\Index$, that is conservative and eliminates parameters.  For any other model groupoid $\X' = (X'_1 \rightrightarrows X'_0)$ such that $X_0 = X'_0$ and $X_1 \subseteq X'_1$, i.e.\ $\X$ is a surjective on objects subgroupoid of $\X'$, then $\X'$ is also a conservative model groupoid that eliminates parameters when the models $M \in X_0 = X'_0$ are given the same indexing by $\Index$.
\end{prop}
\begin{proof}
	We first note that, since $\X$ and $\X'$ contain as objects the same models given the same indexing by the parameters $\Index$, for a formula $\psi$ and a tuple of parameters $\vec{m}$, we have that $\class{\vec{x},\vec{m}:\psi}_\X = \class{\vec{x},\vec{m}:\psi}_{\X'}$.  We also conclude that $\X'$ is conservative since $\X$ is.
	
	Let $\class{\vec{x},\vec{m}:\psi}_\X$ be a definable with parameters.  Since $\X$ eliminates parameters, there is a formula $\varphi$ such that
	\[
	\overline{\class{\vec{x},\vec{m}:\psi}}_\X = \classv{\varphi}{x}_\X .
	\]
	We claim that $\overline{\class{\vec{x},\vec{m}:\psi}}_{\X'} = \classv{\varphi}{x}_{\X'}$ too.  One inclusion is immediate since
	\[
	\classv{\varphi}{x}_{\X'} = \classv{\varphi}{x}_{\X} = \overline{\class{\vec{x},\vec{m}:\psi}}_\X \subseteq \overline{\class{\vec{x},\vec{m}:\psi}}_{\X'}.
	\]
	For the converse inclusion, for each element $\lrangle{\vec{n},N} \in \overline{\class{\vec{x},\vec{m}:\psi}}_{\X'}$, there exists some model $M$ and $\vec{n}'$ such that $M \vDash \psi(\vec{n}',\vec{m})$ and a $\theory$-model isomorphism $M \xrightarrow{\alpha} N \in X'_1$ such that $\alpha(\vec{n}') = \vec{n}$.  Hence, $M \vDash \varphi(\vec{n}')$ since
	\[
	\lrangle{\vec{n}',M} \in \class{\vec{x},\vec{m}:\psi}_\X \subseteq \overline{\class{\vec{x},\vec{m}:\psi}}_\X = \classv{\varphi}{x}_\X,
	\]
	and so $N \vDash \varphi(\vec{n})$ too.
\end{proof}

Hence, we are immediately able to deduce the following.

\begin{coro}\label{coro:eqtoetalecompletion}
	If $\X = (X_1 \rightrightarrows X_0)$ is a representing model groupoid for $\theory$, then $\X$ is {Morita-equivalent} to its étale completion, i.e.\ there exists an indexing of the models $M \in X_0$ such that
	\[\Sh\left(\Xlog\right) \simeq \Sh\left(\hat{\X}_\taulogo^\tauloga\right).\]
\end{coro}

\paragraph{The étale completion as the topological closure.}

The étale completion of a model groupoid can be calculated entirely topologically via an adaptation of \cite[Theorem 4.14]{hodges}.  Therein, it is demonstrated that, for a subgroup $G \subseteq {\bf Sym}(A)$ of the permutation group on a set $A$, the following are equivalent:
\begin{enumerate}
	\item the subgroup $G \subseteq {\bf Sym}(A)$ is a closed set, when $ {\bf Sym}(A)$ is endowed with the Krull topology (also called the pointwise convergence topology),
	
	\item the group $G$ is the automorphism group of the set $A$ when equipped with a $\Sigma$-structure, for some single-sorted signature $\Sigma$.
\end{enumerate}
We present how this result can be adapted to calculate the étale completion.

Let $\X = (X_1 \rightrightarrows X_0)$ be a model groupoid for a geometric theory $\theory$ over a signature $\Sigma$ with an indexing $\Index \paronto \X$ by a set of parameters $\Index$.  For each pair $M, N \in X_0$, we define the \emph{hom-space} $\Hom_{\X}(M,N)$ as the subspace
\[
\Hom_{\X}(M,N) = s^{-1}(M) \cap t^{-1}(N) \subseteq X_1^\tauloga.
\]
Equivalently, $\Hom_{\X}(M,N)$ is the set of isomorphisms $M \xrightarrow{\alpha} N \in X_1$ endowed with the topology generated by the basis
\[
\class{\vec{b} \mapsto \vec{c}}_{\X(M,N)} = \lrset{M \xrightarrow{\alpha} N }{\alpha(\vec{b}) = \vec{c}},
\]
for each pair of tuples of parameters $\vec{b},\vec{c} \in \Index$.

If we were to forget that the models $M$ and $N$ had $\Sigma$-structure, we could still construct a hom-space $\Iso[M,N]$ of \emph{all} isomorphisms between the underlying sets interpreting the sorts of $M, N$.  The space $\Iso[M,N]$ is endowed with the analogous topology generated by the basis
\[
\class{\vec{b} \mapsto \vec{c}}_{\Iso[M,N]} = \lrset{M \xrightarrow{\alpha} N }{\alpha(\vec{b}) = \vec{c}},
\]
for each pair of tuples of parameters $\vec{b},\vec{c} \in \Index$.  Evidently, $\Hom_\X(M,N)$ can be embedded as a subspace into $\Iso[M,N]$.

\begin{prop}[Theorem 4.14 \cite{hodges}]
	Suppose that $\X$ eliminates parameters.  For each pair $M,N \in X_0$, the hom-space $\Hom_{\hat{\X}}(M,N)$ in the étale completion $\hat{\X}$ is the topological closure of the subspace $\Hom_\X(M,N) \subseteq \Iso[M,N]$.
\end{prop}
\begin{proof}
	We must show that a point $M \xrightarrow{\alpha} N \in \Iso[M,N]$ is an accumulation point of $\Hom_\X(M,N)$ if and only if it is an isomorphism of $M$ and $N$ as $\Sigma$-structures.
	
	Fist, suppose that $\alpha$ preserves $\Sigma$-structure, and let $\class{\vec{b} \mapsto \vec{c}}_{\Iso[M,N]}$ be any basic open neighbourhood of $\alpha$.  Since $\X$ eliminates parameters, there exists a formula $\chi$ over $\Sigma$ such that
	\[
	\overline{\class{\vec{x} = \vec{b}}}_\X  = \classv{\chi}{x}_\X = \classv{\chi}{x}_{\hat{\X}}.
	\]
	The isomorphism $\alpha$ preserves the interpretation of $\chi$, and so $\alpha(\vec{b}) = \vec{c} \in \classv{\chi}{x}_{\hat{\X}} = \overline{\class{\vec{x} = \vec{b}}}_\X$.  Therefore, there exists an isomorphism $M \xrightarrow{\gamma} N \in \Hom_\X(M,N) \subseteq \Iso[M,N]$ such that $\gamma \in \class{\vec{b} \mapsto \vec{c}}_{\Iso[M,N]}$.  Hence, $\alpha$ is an accumulation point of $\Hom_{\X}(M,N)$.
	
	Conversely, if $\alpha$ is an accumulation point of $\Hom_\X(M,N)$, then for each tuple of parameters $\vec{m} \in \Index$, there is an isomorphism of $\Sigma$-structures $M \xrightarrow{\gamma} N$ such that $\alpha(\vec{m}) = \gamma(\vec{m})$.  Thus, since every tuple of elements of $M$ is the interpretation of some tuple of parameters, $\alpha$ preserves the $\Sigma$-structure.
\end{proof}

%%%%%%%%%%%%%%%%%%%%%%%%%%%%%%%%%%

%%%%%%%%%%%%%%%%%%%%%%%%%%%%%%%%%%

\subsection{Maximal groupoids}\label{sec:forssell}

We now turn to the topological groupoids considered in the works of Awodey, Butz, Forssell and Moerdijk \cite{BM,awodeyforssell,forssell} and demonstrate that these too fall within our general framework.  Let $\X$ be a model groupoid for a geometric theory $\theory$ indexed by a set of parameters $\Index$.  Recall from \cref{rem:df:conservandelimpara}\cref{enum:rem:df:elimpara:maxformula} that, for every tuple of parameters $\vec{m} \in \Index$, the formula in context $\left\{\vec{x}:\bigwedge_{ m_i = m_j} x_i = x_j\right\}$ can be thought of as a `universal upper bound' for elimination of parameters in that we always have an inclusion
\[
\overline{\class{\vec{x} = \vec{m}}}_\X \subseteq \lrclass{\vec{x}: \bigwedge_{ m_i = m_j} x_i = x_j}_\X.
\]
The particular model groupoids $\X$ considered in \cite{BM,awodeyforssell,forssell} can be considered to be \emph{maximal} in the sense that this inclusion is an equality
\[
\overline{\class{\vec{x} = \vec{m}}}_\X = \lrclass{\vec{x}: \bigwedge_{ m_i = m_j} x_i = x_j}_\X
\]
(this is observed, for instance, in \cite[Lemma 3.1.2.1]{forssellphd}).  The topological groupoids considered in \cite{BM,awodeyforssell,forssell} are also closely related to the original Joyal-Tierney representation result \cite[Theorem VIII.3.2]{JT} that every topos is the topos of sheaves on an open localic groupoid.  This is elaborated further in \cite[\S 4.4]{myselfandgraham} where it is shown that if the theory is countable, the constructions coincide.

We briefly motivate the use of what we call \emph{Forssell groupoids}.  If we were to equate open representing groupoids of a theory $\theory$ with those model groupoids from which all other models can be reconstructed, then intuitively the groupoid of all models contains sufficient information.  However, this is not a small groupoid.  We might imagine that is suffices to restrict to the groupoid of all models of some sufficiently large cardinality.  That this is the case is demonstrated in \cite{awodeyforssell,forssellphd,forssell}.

\begin{df}[\S 1.2 \cite{awodeyforssell}, \S 3.1 \cite{forssell}]
	Let $\theory$ be a geometric theory and let $\Indexset$ be an infinite set.  The \emph{Forssell groupoid} $\FG(\Indexset)$ is the étale complete groupoid of all models whose underlying sets are subquotients of $\Indexset$, i.e.\ the groupoid of all $\Indexset$-indexed models.
\end{df}

If $\theory$ is a geometric theory whose $\Indexset$-indexed models are conservative, then the groupoid $\FG(\Indexset)$ is an open representing groupoid.  Thus, by the classification in \cref{maintheorem}, we know that there exists an indexing of $\FG(\Indexset)$ for which the groupoid eliminates parameters.  Given the construction of $\FG(\Indexset)$, we would expect this to be the already present indexing by $\Indexset$.  Indeed, that $\FG(\Indexset)$ eliminates parameters for this indexing was shown in Lemma 3.4 \cite{forssell}.

A similar idea is pursued in the work of Butz and Moerdijk \cite{BM} through the use of \emph{enumerated models}.

\begin{dfs}[\S 2 \cite{BM}]
	Let $\theory$ be a geometric theory and let $\Indexset$ be an infinite set.
	\begin{enumerate}
		\item A $\theory$-model $M$ is said to be $\Indexset$-\emph{enumerated} if it is $\Indexset$-indexed and each element $n \in M$ is indexed by infinitely many parameters.
		\item The \emph{Butz-Moerdijk groupoid} $\BM(\Indexset)$ is the étale complete groupoid of all $\Indexset$-enumerated models.
	\end{enumerate}
\end{dfs}

We will show that both the Forssell groupoids of all $\Indexset$-indexed models studied in \cite{awodeyforssell,forssell,forssellphd} and the Butz-Moerdijk groupoids of all $\Indexset$-enumerated models of \cite{BM} fall within our framework via the following consequence of \cref{maintheorem}.

\begin{prop}\label{prop:maximalgrpds}
	Let $\X$ be an étale complete model groupoid for $\theory$ with an indexing by a set of parameters $\Index$ satisfying the following properties.
	\begin{enumerate}
		\item\label{enum:infinite_index} The indexing set $\Index$ is infinite.
		\item\label{enum:finite_reindexing} The set of models $X_0$ is \emph{closed under finite reindexing} -- by which we mean that for each $M \in X_0$ with an indexing $ \Index \paronto M$, then for any partial surjection $\Index \paronto \Index$ whose domain is a cofinite subset of $\Index$ and whose fibre at each $m \in \Index$ is finite, $X_0$ also contains an isomorphic model $M' \cong M$ whose indexing is given by the composite $\Index \paronto \Index \paronto M \cong M'$, i.e.\ we can add or remove finitely many parameters to the indexing of a model $M \in X_0$.

		%injective endomorphism $\Index\rightarrowtail \Index$ whose image is cofinite, $X_0$ also contains an isomorphic model $M' \cong M$ whose indexing is given by the composite $\Index \rightarrowtail \Index \paronto M \cong M'$, i.e.\ we can change finitely many of the parameters for any model $M \in X_0$.
		%\item\label{enum:further_indexing} The set of models $X_0$ is \emph{closed under further indexing} -- by which we mean that for each model $M \in X_0$ with an indexing $\Index \paronto M$, and each finite tuple of parameters $\vec{m}$ not in the domain of $\Index \paronto M$, $X_0$ also contains the isomorphic model $M'\cong M$ whose indexing is given by any extension of $\Index \paronto M \cong M'$ to include $\vec{m}$ in the domain, i.e.\ we can add any unused parameters to the indexing of a model $M \in X_0$.
	\end{enumerate}
	Then the model groupoid $\X$ eliminates parameters.
\end{prop}
\begin{proof}
	By \cref{rem:df:conservandelimpara}\cref{enum:rem:df:conservandelimpara:sufficetochecktuple}, it suffices to show that, for each tuple of parameters $\vec{m} \in \Index$, $\overline{\class{\vec{x}=\vec{m}}}_\X$ is definable without parameters.  We claim that, for each tuple $\vec{m} \in \Index$,
	\[
	\overline{\class{\vec{x}=\vec{m}}}_\X = \lrclass{\vec{x}:\bigwedge_{ m_i = m_j} x_i = x_j }_\X ,
	\]
	where the (finite) conjunction $\bigwedge_{ m_i = m_j} x_i = x_j$ ranges over the elements $m_i,m_j \in \vec{m}$ that are equal.  As observed in \cref{rem:df:conservandelimpara}\cref{enum:rem:df:elimpara:maxformula}, there is an evident inclusion 
	\[\overline{\class{\vec{x}=\vec{m}}}_\X \subseteq \lrclass{\vec{x}:\bigwedge_{ m_i = m_j} x_i = x_j }_\X,\]
	so it remains to demonstrate the reverse inclusion.  
	
	Suppose we are given an element
	\[\lrangle{\vec{n},M} \in \lrclass{\vec{x}:\bigwedge_{ m_i = m_j} x_i = x_j }_\X,\]
	and let $\vec{m}' \in \Index$ be a tuple of parameters that indexes $\vec{n}$ under the indexing $\Index \paronto M$. Since our indexing set is infinite by \cref{enum:infinite_index}, there exists some partial surjection $\Index \paronto \Index$ whose domain is cofinite, whose fibres are finite, and moreover sends the tuple $\vec{m}$ to $\vec{m}'$.  Thus, by hypothesis \cref{enum:finite_reindexing}, $X_0$ contains an isomorphic model $M' \cong M$ whose indexing is given by the composite $\Index \paronto \Index \paronto M \cong M'$, and so the tuple $\vec{n}' \in M'$ corresponding to $\vec{n} \in M$ under the isomorphism $M \cong M'$ is indexed by $\vec{m}$.  Thus, as $\X$ is étale complete, we obtain that $\lrangle{\vec{n},M} \in \overline{ \class{\vec{x} = \vec{m}}}_\X$ as required.
\end{proof}

\begin{coro}[Theorem 1.4.8 \cite{awodeyforssell}, Theorem 5.1 \cite{forssell}, \cite{BM}]
	Let $\theory$ be a geometric theory and let $\Indexset$ be an infinite set.
	\begin{enumerate}
		\item The $\Indexset$-indexed models of $\theory$ are conservative if and only if, by endowing $\FG(\Indexset)$ with the logical topologies we obtain a representing groupoid
		\[
		\Sh\left(\FG(\Indexset)^\tauloga_\taulogo\right) \simeq \topos_\theory.
		\]
		\item The $\Indexset$-enumerated models of $\theory$ are conservative if and only if, by endowing $\BM(\Indexset)$ with the logical topologies we obtain a representing groupoid
		\[
		\Sh\left(\BM(\Indexset)^\tauloga_\taulogo\right) \simeq \topos_\theory.
		\]
	\end{enumerate}
\end{coro}
\begin{proof}
	The proof is simply to recognise that Forssell groupoids and Butz-Moerdijk groupoids satisfy the conditions of \cref{prop:maximalgrpds}.  We expand the details for Butz-Moerdijk groupoids.  By hypothesis, $\Indexset$ is infinite.  Since the set of parameters that index each $n \in M\in \BM(\Indexset)$ is infinite, we can change finitely many of the parameters and still end up with a $\Indexset$-enumerated model, and so $\BM(\Indexset)$ is also closed under finite reindexing.
\end{proof}

Using \cref{prop:maximalgrpds}, we can easily deduce that other similar indexed model groupoids are representing, e.g.\ the groupoid of all $\Indexset$-\emph{finitely indexed} models of a theory $\theory$ -- i.e.\ those models that are $\Indexset$-indexed and whose equivalence class of each $n \in M$ is finite.  Also using maximal groupoids, we are able to deduce a useful construction for positing the existence of representing model groupoids with certain structures present in the objects.

\begin{coro}\label{coro:fgallowsbiggerandbigger}
	Let $\theory$ be a geometric theory.  If $W$ is a conservative set of $\theory$-models, then there exists a groupoid $\X = (X_1 \rightrightarrows X_0)$ of $\theory$-models which represents $\theory$ for which every model $M \in X_0$ is isomorphic to some $M' \in W$.
\end{coro}
\begin{proof}
	Let $\Index$ be an infinite indexing set for $W$, and let $\X = (X_1 \rightrightarrows X_0)$ be the étale complete model groupoid of all $\Index$-indexed models of $\theory$ that are isomorphic to some model contained in $W$.  By construction, $\X$ is a conservative groupoid, and $\X$ also eliminates parameters since it satisfies the hypotheses of \cref{prop:maximalgrpds}.
\end{proof}

%%%%%%%%%%%%%%%%%%%%%%%%%%%%%%

%%%%%%%%%%%%%%%%%%%%%%%%%%%%%%%%%%%%%%%%%%%%%%%
%%%%%%%%%%%%%%%%%%%%%%%%%%%%%%%%%%%%%%%%%%%

\subsection{A theory classified by an indexed groupoid}\label{sec:theory_of_groupoid}

Let $\theory$ be a geometric theory over a signature $\Sigma$ with enough set-based models.  The methods of \cref{sec:forssell} ensure that we can always find a groupoid of $\Sigma$-structures, with an indexing by parameters $\Index$, for which the resultant open topological groupoid is a representing groupoid for the theory $\theory$.

In this section, we consider the converse problem: given a groupoid $\X$ of $\Sigma$-structures with an indexing $\Index \paronto \X$, what is a theory classified by the topos $\Sh(\Xlog)$ of sheaves on the resulting topological groupoid?  It arises that, in general, we cannot choose a theory over the same signature $\Sigma$.  Instead, we must choose a relational/localic extension.  This extends the correspondence between relational extensions and closed subgroups of the permutation group found in \cite[Theorem 4.14]{hodges} and discussed in \cref{sec:etalecomplete}.

\begin{df}\label{df:Sigma_of_index}
	Let $\Sigma$ be a signature, and let $\X$ be a groupoid of $\Sigma$-structures with an indexing $\Index \paronto \X$ (i.e.\ $\X$ is an indexed model groupoid for $\Etheory_\Sigma$, the empty theory over the signature $\Sigma$).  We denote by $\Sigma_{\Index \to \X}$ the relational extension of the signature $\Sigma$ which adds, for each tuple of parameters $\vec{m} \in \Index$, a relation symbol $R_{{\vec{m}}}$ of the same sort as $\vec{m}$.
\end{df}

The groupoid of $\Sigma$-structures $\X = (X_1 \rightrightarrows X_0)$ is automatically a groupoid of $\Sigma_{\Index \to \X}$-structures.  For each $\Sigma$-structure $M \in X_0$, we interpret $R_{{\vec{m}}}$ as the subset
\[
\lrset{\vec{n} \in M}{\lrangle{\vec{n},M} \in \overline{\class{\vec{x}=\vec{m}}}_\X}.
\]
The subset $\overline{\class{\vec{x}=\vec{m}}}_\X \subseteq \coprod_{M \in X_0} M^{\vec{m}}$ is, by definition, stable, and thus every isomorphism $M \xrightarrow{\alpha} N \in X_1$ preserves the interpretation of the relation $R_{{\vec{m}}}$.  Hence, $\alpha$ is also an isomorphism of $\Sigma_{\Index \to \X}$-structures.

\begin{df}\label{df:theory-of-indexed-groupoid}
	Let $\X$ be a groupoid of $\Sigma$-structures with an indexing $\Index \paronto \X$.  We denote by $\theory_{\Index \paronto \X}$ the \emph{theory of the indexed groupoid}.  It is the theory over the signature $\Sigma_{\Index \to \X}$ whose axioms are precisely those sequents $\varphi \vdash_{\vec{x}} \psi$ over $\Sigma_{\Index \to \X}$ which are satisfied in all structures $M \in X_0$ (once each $M \in X_0$ is interpreted as a $\Sigma_{\Index \paronto \X}$-structure).
\end{df}

\begin{coro}\label{coro:theory_for_grpd}
	Let $\Index \paronto \X$ be an indexed groupoid of $\Sigma$-structures.  There is an equivalence of toposes
	\[
	\Sh\left(\Xlog\right) \simeq \topos_{\theory_{\Index \paronto \X}}.
	\]
\end{coro}
\begin{proof}
	By the definition of $\theory_{\Index \paronto \X}$, the groupoid $\X$ is a conservative groupoid for $\theory_{\Index \paronto \X}$.  Next, by the construction of the signature $\Sigma_{\Index \paronto \X}$, the groupoid $\X$ eliminates parameters as a groupoid of $\Sigma_{\Index \paronto \X}$-structures.  Explicitly, for each tuple of parameters, we have that
	\[
	\overline{\class{\vec{x} = \vec{m}}}_\X = \classv{R_{\vec{m}}}{x}_\X.
	\]
	Thus, when endowed with the logical topologies, $\Xlog$ is an open topological groupoid (by \cref{lem:logicaltopsgiveopengrpd}) and $\Sh(\Xlog)$ classifies $\theory_{\Index \paronto \X}$ by \cref{maintheorem}.
\end{proof}

%%%%%%%%%%%%%%%%%%%%%%%%%%%%%%%%%
\begin{ex}[The theory of a generic Dedekind section]\label{ex:theory_R+R}
	Let $\X$ be a groupoid of $\Sigma$-structures with an indexing $\Index \paronto \X$.  While the signature $\Sigma_{\Index \paronto \X}$ constructed in \cref{df:Sigma_of_index} ensures that $\X$ eliminates parameters over the signature $\Sigma_{\Index \paronto \X}$, we may however wish to refrain from adding too many symbols to our signature.
	
	Evidently, we do not need to add a new relation symbol $R_{\vec{m}}$ for every tuple of parameters $\vec{m} \in \Index$, but only those for which the orbit $\overline{\class{\vec{x}= \vec{m}}}_\X$ is not definable without parameters.  By making astute choices about how to expand the signature, we can minimise the number of new symbols we must add.

	Recall from \cref{ex:atomic_but_not_ultra} that $\R + \R$ is a model for the theory of dense linear orders without endpoints that is not ultrahomogeneous, and consequently the automorphism group $\Aut(\R+\R)$ does not eliminate parameters (where we have assumed that $\R + \R$ is trivially indexed).  We describe a theory classified by the topos $\B \Aut(\R + \R)$ using the above techniques.

	For $i = 1,2$, the automorphism group $\Aut(\R + \R)$ acts transitively on the subset $\{\,i\,\} \times \R \subseteq \R + \R$, i.e.\ for any $r \in \R$, $\overline{\class{x = (i,r)}}_{\Aut(\R + \R)} = \{\,i\,\} \times \R \subseteq \R + \R$, and so we are motivated to consider the localic extension of $\DLO$ by the addition of a pair of unary relation symbols $U_1$ and $U_2$, where these symbols are interpreted in the model $\R + \R$ as the subsets $\class{U_1(x)}_{\R + \R} = \{\,1\,\} \times \R, $ and $\class{U_2(x)}_{\R + \R} = \{\,2\,\} \times \R$.  The automorphism group $\Aut(\R + \R)$ eliminates parameters over this expanded signature.  Namely, we have that $\overline{\class{\vec{x} = \vec{m}}}_{\Aut(\R + \R)}$ is given by
	\[
	\lrclass{\vec{x}: 
		\bigwedge_{\substack{m_i , m_j \in \vec{m}, \\ m_i = m_j}}\!\!\! x_i = x_j \land 
		\!\!\!\bigwedge_{\substack{m_i , m_j \in \vec{m}, \\ m_i < m_j}}\!\!\! x_i < x_j \land 
		\!\!\!\bigwedge_{\substack{m_i\in \vec{m}, \\ m_i \in \{\,1\,\}\times \R}} \!\!\!U_1(x_i) \land 
		\!\!\!\bigwedge_{\substack{m_j\in \vec{m}, \\ m_j \in \{\,2\,\}\times \R}}\!\!\! U_2(x_j) 
	}_{\Aut(\R + \R)}.
	\]
	
	Thus, in a manner similar to \cref{coro:theory_for_grpd}, we deduce that the topos $\B \Aut(\R + \R)$ classifies the localic expansion of the theory $\DLO$ by two unary predicates whose axioms are those sequents satisfied in the model $\R + \R$, which we denote by $\theory_{\R + \R}$.  The sequents
	\begin{align*}
		U_1(x) \land U_2(x) & \vdash_x \bot, & x < y & \vdash_{x,y} U_1(x) \lor U_2(y), \\
		\top & \vdash_\emptyset \exists x \ U_1(x), & \top & \vdash_\emptyset \exists y \ U_2(y), \\
		U_1(x) \land y < x & \vdash_{x,y} U_1(y), & U_2(y) \land y < x & \vdash_{x,y} U_2(y), \\
		U_1(x) &\vdash_x \exists y \ U_1(y) \land x < y, & U_2(y) & \vdash_y \exists x \ U_2(x) \land x < y
	\end{align*}
	(in addition to those from $\DLO$) suffice to generate this theory.

	The theory $\theory_{\R + \R}$ can be likened to a \emph{theory of Dedekind sections} for an arbitrary dense linear order without endpoints\footnote{In \cite[\S 3.5]{vickerslocales}, Vickers describes, as a localic expansion of the theory of dense linear orders without endpoints $\DLO$, a theory of Dedekind sections on the rationals.  Tacitly, an interpretation of the rationals is fixed by introducing a constant symbol $c_q$ for each rational $q \in \Q$ and axioms
		\begin{enumerate}
			\item $\top \vdash c_p < c_q$, for each pair of rationals $p,q$ with $p < q$,
			\item and $\top \vdash_x \bigvee_{q \in \Q} x = c_q$.
		\end{enumerate}
		Consequently, there are no non-trivial isomorphisms of models.}.  Indeed, the rationals $\Q$ can be made into a model of $\theory_{\R + \R}$ with the interpretations
	\[
	\class{U_1(x)}_\Q = (-\infty,a) \text{ and } \class{U_2(x)}_\Q = (a,\infty)
	\]
	for any irrational $a$.  Of course, not every automorphism of $\Q$ as a linear order will preserve the further $U_1$ and $U_2$ structure.  In contrast, $\R$ does not admit an interpretation as a $\theory_{\R + \R}$-model.
\end{ex}

%%%%%%%%%%%%%%%%%%%%%%%%%%%%%%%%

%%%%%%%%%%%%%%%%%%%%%%%%%%%%%%%%%%%%%%%%%%%%%%%%
%%%%%%%%%%%%%%%%%%%%%%%%%%%%%%%%%%%%%%%%%%%%%%%%%

\section{A worked example: the theory of algebraic integers}\label{sec:algint}

We have seen in \cref{sec:atomic} that the examples of representing groups/groupoids considered in \cite{blassscedrov} and \cite{caramellogalois} can be subsumed by the classification result \cref{maintheorem}.  Similarly, in \cref{sec:forssell} we demonstrated that the `maximal' representing groupoids constructed in \cite{forssell,forssellphd,awodeyforssell,BM} also fall within the scope of \cref{maintheorem}.  Examples of representing groupoid that do not directly originate via the methods exposited therein have been given in \cref{ex:groupoids_for_decidables}\cref{ex:enum:alg_ext} and \cref{ex:notallisos}.

In this section, we study in further detail another representing groupoid that does not arise from the previous approaches found in the literature.  The theory we consider is the theory of \emph{algebraic integers}.

\paragraph{The theory of algebraic integers.}

For each prime $p$, the monic minimal polynomial of each element of the field $\overline{\Z / \langle{p}\rangle}$ has integer coefficients.  In this sense, the algebraic numbers and algebraic integers modulo $p$ coincide, and the field $\overline{\Z / \langle{p}\rangle}$ is a model for a theory of algebraic integers.  Along with the standard ring of algebraic integers $\overline{\Z}$, these rings can be axiomatised as follows.

\begin{df}
	We denote by $\AItheory$ the (geometric) theory of \emph{algebraic integers}.  It is the single-sorted theory over the signature of rings whose axioms are the following:
	\begin{enumerate}
		\item the standard axioms of a commutative ring,
		\item the sequent
		\[
		x \cdot y = 0 \vdash_{x,y} x= 0 \lor y = 0,
		\]
		expressing that any model is an integral domain,
		\item the sequent 
		\[
		\top \vdash_x \bigvee_{q \in \Z[x]_{\rm monic}} q(x),
		\]
		where $\Z[x]_{\rm monic}$ is the set of the \emph{monic} polynomials with integer coefficients, expressing that every element is in the \emph{integral closure} of the prime subring,
		\item and for each $n \in \N$, the sequent
		\[
		\top \vdash_{x_{n-1}, \dots , x_0} \exists y \ y^n + x_{n-1} y^{n-1} + \dots + x_0 ,
		\]
		expressing that the model is algebraically closed with respect to monic polynomials.
	\end{enumerate}
\end{df}

\begin{lem}\label{lem:models_of_ai}
	Every model of $\AItheory$ is isomorphic to $\overline{\Z}$ or $\overline{\Z / \langle p \rangle}$ for some prime $p$.  
\end{lem}
\begin{proof}
	A model $R$ of $\AItheory$ is an integral domain, and it is isomorphic to $\overline{\Z}$ or $\overline{\Z / \langle p \rangle}$ for some prime $p$ depending on the characteristic of $R$.  We will show that if $R$ has characteristic $0$, then $R \cong \overline{\Z}$.  The proof in the case where $R$ has finite characteristic is almost identical.
	
	Since $R$ is an integral domain, we can form its field of fractions $\Frac(R)$ as well as the algebraic closure $\overline{\Frac(R)}$ of this field.  Subsequently, there exist isomorphic copies of $\Z, \overline{\Z}, \Q$ and $\overline{\Q}$ inside $\overline{\Frac(R)}$ along with the following inclusions of rings:
	\[
	\begin{tikzcd}[row sep=tiny]
		& R \ar[hook]{rr} && \Frac(R) \ar[hook]{rd} & \\
		\Z \ar[hook]{ru} \ar[hook]{rd} \ar[hook]{rr} && \Q \ar[hook]{ru} \ar[hook]{rd} && \overline{\Frac(R)}. \\
		& \overline{\Z} \ar[hook]{rr} && \overline{\Q} \ar[hook]{ru} &
	\end{tikzcd}
	\]
	Viewed as subrings of $\overline{\Frac(R)}$, the condition that $R$ is algebraic over its prime subring ensures that $R \subseteq \overline{\Z}$, while the condition that $R$ is algebraically closed with respect to monic polynomials ensures the converse inclusion.
\end{proof}

Since $\AItheory$ is not an atomic theory (for example, there is no single geometric formula that is provably equivalent to the infinite conjunction $\bigwedge_{p \text{ prime}} p \neq 0$), the classifying topos $\topos_\AItheory$ cannot be equivalent to a topos of sheaves on a simple disjoint coproduct of automorphism groups as in \cref{coro:booltopos}.  Instead, we must search further afield for a representing groupoid.

We note as well that the standard ring of algebraic integers $\overline{\Z}$ plays a special role amongst all models of the theory $\AItheory$.  It is not a conservative model -- this is clear since $\overline{\Z}$ satisfies the sequent
\[
\underbrace{1 + 1 + \dots + 1}_{p \text{ times}} = 0 \vdash \bot ,
\]
but $\AItheory$ has a model of cardinality $p$.  However, $\overline{\Z}$ does have the property that $\overline{\Z}$ satisfies a (geometric) sentence $\varphi$ if and only if $\theory$ proves the sequent $\top \vdash_\emptyset \varphi$.  For this reason, we will say that $\overline{\Z}$ is \emph{sentence-complete}.  Below, we construct a representing groupoid for $\AItheory$ where the fact that $\overline{\Z}$ is sentence-complete can be captured topologically, as seen in \cref{coro:algint_is_sent_complete}.

\paragraph{A representing groupoid for the theory of algebraic integers.}

Let $a \in \overline{\Z}$ be an algebraic integer, and let $q_a$ be its minimal polynomial.  For any ring homomorphism  $f \colon \overline{\Z} \to \overline{\Z / \langle{p}\rangle}$ for $p$ a prime, $f(a)$ is also a root of the polynomial $q_a$.  However, over $\Z / \langle p \rangle$, the polynomial $q_a$ may no longer be irreducible.  Suppose that $q_a$ factors into irreducible polynomials as $q_a^1 q_a^2 \dots q_a^k$ over $\Z / \langle p \rangle$.  Then $f(a) \in \overline{\Z / \langle{p}\rangle}$ has as its minimal polynomial $q^i_a$ for some $i$.

Assuming the axiom of choice, there exists a maximal ideal $M$ of $\overline{\Z}$ which contains both the prime $p$ and $q_a^i(a)$.  Taking the quotient ring yields a field $\overline{\Z}/M$ which is moreover an algebraic closure for its prime subfield $\Z / \langle p \rangle$.  Hence, by uniqueness of algebraic closures, we deduce the existence of a surjective ring homomorphism
\[
\begin{tikzcd}
	\pi_M \colon \overline{\Z} \ar[two heads]{r} & \overline{\Z}/M \cong \overline{\Z / \langle p \rangle}
\end{tikzcd}
\]
with the property that $\pi_M(a)$ has minimal polynomial $q_a^i$.

\begin{df}\label{df:repr_grpd_for_AI}
	Let $\AIform = (\AIform_1 \rightrightarrows \AIform_0)$ denote the étale complete model groupoid for $\AItheory$ whose underlying set of objects $\AIform_0$ is constructed as follows:
	\begin{enumerate}
		\item $\AIform_0$ contains one copy of the model $\overline{\Z}$;
		\item we add, for each prime $p$ and each maximal ideal $M \subseteq \overline{\Z}$ containing $p$, a copy of the model $\overline{\Z / \langle{p}\rangle}$.
	\end{enumerate}
	That is, $\AIform$ is the groupoid
	\[
	\Aut\left(\overline{\Z}\right) + \coprod_{p \text{ prime}} {\bf ConGrpd}\left(\left\{\, M \subseteq \overline{\Z} \text{ a maximal ideal containing } p \,\right\}, \Aut\left(\overline{\Z / \langle p \rangle}\right)\right),
	\]  
	where we use the notation ${\bf ConGrpd}(Y,G)$, for a set $Y$ and a group $G$, to denote the (unique) connected groupoid whose objects are $Y$ and, for each pair $y , y' \in Y$, $\Hom_{{\bf ConGrpd}(Y,G)}(y,y') =G$.
	
	We endow $\AIform$ with the indexing whose set of parameters are the algebraic integers $\overline{\Z}$.  The model $\overline{\Z}$ is given the trivial indexing of itself by its own elements.  Meanwhile, the indexing of the model $\overline{\Z / \langle p \rangle}$ corresponding to the maximal ideal $M \subseteq \overline{\Z}$ is determined by the surjective ring homomorphism $\pi_M \colon \overline{\Z} \twoheadrightarrow \overline{\Z}/M \cong \overline{\Z / \langle{p}\rangle}$.  To make this indexing explicit, we will not abuse notation (as we have done in the rest of the paper) and instead denote the interpretation of the parameter $a \in \overline{\Z}$ in a copy of $\overline{\Z / \langle p \rangle}$ by $\pi_M(a)$.
\end{df}

\begin{prop}\label{prop:AIform_is_repr}
	The groupoid $\AIform$, with the indicated indexing $\overline{\Z} \paronto {\AIform}$, is conservative and eliminates parameters.  Therefore, there is an equivalence of toposes
	\[
	\topos_\AItheory \simeq \Sh\left(\AIform^\tauloga_\taulogo\right).
	\]
\end{prop}
\begin{proof}
	By \cref{lem:models_of_ai}, the set of objects $\AIform_0$ is a conservative set of models.  By \cref{rem:df:conservandelimpara}\cref{enum:rem:df:conservandelimpara:sufficetochecktuple}, to show that $\AIform$ eliminates parameters it suffices to demonstrate that, for each tuple of parameters $\vec{a} \in \overline{\Z}$, the orbit $\overline{\class{\vec{x} = \vec{a}}}_\AIform$ is definable without parameters.  We claim that 
	\[
	\overline{\class{\vec{x} = \vec{a}}}_\AIform = \classv{ q_{a_1}(x_1) = 0 \land q_{a_2}(x_1,x_2) = 0 \land \dots \land q_{a_n}(x_1, \dots , x_n) = 0}{x}_\AIform,
	\]
	where $q_{a_i}(a_1, \dots , a_{i-1}, x_i)$ denotes a minimal polynomial of $a_i \in \vec{a}$ over $\Z[a_1, \dots , a_{i-1}]$.
	
	By definition, the tuple $\vec{a} \in \overline{\Z}$ satisfies the formula
	\[
	\overline{\Z} \vDash q_{a_1}(a_1) = 0 \land q_{a_2}(a_1,a_2) = 0 \land \dots \land q_{a_n}(a_1, \dots , a_n) = 0.
	\]
	Moreover, since the interpretation of the parameters $\vec{a}$ in any other model $\overline{\Z / \langle p \rangle} \in \AIform_0$ is determined by a ring homomorphism $\pi_M \colon \overline{\Z} \twoheadrightarrow \overline{\Z / \langle{p}\rangle}$, we conclude that
	\begin{align*}
		\overline{\Z/ \langle p \rangle} & \vDash \pi_M(q_{a_1}(a_1)) = 0 \land  \dots \land \pi_M(q_{a_n}(a_1, \dots , a_n)) = 0, \\
		& \vDash q_{a_1}(\pi_M(a_1)) = 0 \land  \dots \land q_{a_n}(\pi_M(a_1), \dots , \pi_M(a_n)) = 0.
	\end{align*}
	Consequently, there is an inclusion
	\[
	\class{\vec{x} = \vec{a}}_\AIform \subseteq \overline{\class{\vec{x} = \vec{a}}}_\AIform \subseteq \classv{ q_{a_1}(x_1) = 0 \land \dots \land q_{a_n}(x_1, \dots , x_n) = 0}{x}_\AIform .
	\]

	We turn to the converse inclusion.  Since the automorphism group $\Aut\left(\overline{\Z}\right)$ acts transitively on the set of solutions to an irreducible polynomial, we have that
	\[
	\classv{ q_{a_1}(x_1) = 0 \land  \dots \land q_{a_n}(x_1, \dots , x_n) = 0}{x}_\AIform \cap \overline{\Z} \subseteq \overline{\class{\vec{x} = \vec{a}}}_{\Aut\left(\overline{\Z}\right)} = \overline{\class{\vec{x} = \vec{a}}}_\AIform \cap \overline{\Z} .
	\]
	Now let $\vec{w} \in \overline{\Z / \langle p \rangle}$ be a tuple for which
	\[
	\overline{\Z/ \langle p \rangle} \vDash q_{a_1}(w_1) = 0 \land q_{a_2}(w_1,w_2) = 0 \land \dots \land q_{a_n}(w_1, \dots , w_n) = 0,
	\]
	and let $q_{w_i}(x_1,\dots, x_i)$ denote a minimal polynomial of $w_i$ over $\Z / \langle p \rangle [w_1, \dots , w_{i-1}]$.  By Zorn's lemma, we can extend the non-trivial ideal on $\overline{\Z}$ generated by the set $\left\{\,p , q_{w_1}(a_1), \dots , q_{w_n}(a_1, \dots , a_n)\,\right\}$ to a maximal ideal $M \subseteq \overline{\Z}$.  Hence, there is an isomorphic copy of $\overline{\Z / \langle{p}\rangle}$ in $\AIform_0$ in which a minimal polynomial of $\pi_M(a_i)$ over $\Z / \langle{p}\rangle[\pi_M(a_1), \dots , \pi_M(a_{i-1})]$ is given by $q_{w_i}$.  Thus, there exists an automorphism of the field extension $\overline{\Z / \langle p \rangle}$ sending $\pi_M(\vec{a})$ to $\vec{w}$, completing the proof of the converse inclusion
	\[
	\classv{ q_{a_1}(x_1) = 0 \land  \dots \land q_{a_n}(x_1, \dots , x_n) = 0}{x}_\AIform \subseteq \overline{\class{\vec{x} = \vec{a}}}_\AIform .
	\]
\end{proof}

\paragraph{Properties of the space of objects.}

We now describe the space of objects $\AIform_0^\taulogo$ in more detail.  Eventually, we will observe that $\AIform_0^\taulogo$ can be described, up to homeomorphism, in entirely topological terms.  Recall that a basic open of $\AIform_0^\taulogo$ is given by the interpretation of a sentence with parameters $\classv{\varphi}{m}_\AIform \subseteq \AIform_0$.  Recall also from \cref{rem:taulogo:atomic_sents_only} that it suffices to consider only the atomic sentences with parameters, which in the case of rings amounts to formulae of the form $q(\vec{a}) = 0$, for some tuple of parameters $\vec{a}$ and a polynomial $q$.

First, we note that the subset of $\AIform_0$ consisting only of models of the form $\overline{\Z / \langle p \rangle}$ is an open subset in $\taulogo$.  Namely, it is the subset
\[
\class{p=0}_\AIform \cong \lrset{M \subseteq \overline{\Z}}{p \in M, \,M \text{ a maximal ideal}} \subseteq \AIform_0.
\]

\begin{lem}
	For each prime $p$, the subspace $\lrclass{p = 0}_\AIform \cong \lrset{M \subseteq \overline{\Z}}{p \in M, \,M \text{ a maximal ideal}} $ of $\AIform_0^\taulogo$ is homeomorphic to the Cantor space $2^\N$.
\end{lem}
\begin{proof}
	We first remark that the Cantor space $2^\N$ is homeomorphic to any uncountable closed subspace of itself (this is a consequence of Brouwer's characterisation of the Cantor space).  Thus, it suffices to show that $\class{p = 0}_\AIform$, an uncountable space, is homeomorphic to a closed subspace of $2^\N$.
	
	There is an evident inclusion
	\[
	\lrclass{p = 0}_\AIform \cong \lrset{M \subseteq \overline{\Z}}{p \in M, \,M \text{ a maximal ideal}} \subseteq 2^{\overline{\Z}} \cong 2^\N.
	\]
	We must first show that the induced topology on $\lrclass{p = 0}_\AIform$ as a subspace of $\AIform_0^\taulogo$ is the same as that induced as a subspace of $2^{\overline{\Z}}$.  The topology on $2^{\overline{\Z}}$ is generated by the basic opens $\lrset{M}{a \in M}$ and $\lrset{M}{a \not \in M}$, for each algebraic integer $a \in \overline{\Z}$.
	
	Under the bijection $\class{p=0}_\AIform \cong \lrset{M}{p \in M}$, the former corresponds to the open $\class{a = 0 }_\AIform \cap \class{ p = 0}_\AIform$.  To show that the latter $\lrset{M}{a \not \in M}$ is also open, we note that if $a \not \in M$, then $\pi_M(a) $ is a non-zero element of the field $\overline{\Z} / M \cong \overline{Z / \langle p \rangle}$ and hence invertible.  Thus, under the bijection $\class{p=0}_\AIform \cong \lrset{M}{p \in M}$, we have that $\class{\exists y \ a \cdot y = 1}_\AIform \cap \class{p = 0}_\AIform \cong \lrset{M}{a \not \in a}$.  Hence, the topology on $\class{p = 0}_\AIform$ induced as a subspace of $\AIform_0^\taulogo$ contains that as induced as a subspace of $2^{\overline{\Z}}$.  For the reverse inclusion of topologies, by \cref{rem:taulogo:atomic_sents_only} it suffices to note that
	\[
	\class{q(\vec{a}) = 0}_\AIform \cap \class{p = 0}_\AIform \cong \lrset{M}{q(\vec{a}) \in M}.
	\]
	Thus, $\class{p = 0}_\AIform$ is a subspace of $2^{\overline{\Z}}$.
	
	It remains to show that $\class{p=0}_\AIform$ is a closed subset of $2^{\overline{\Z}}$.  It is straightforward to demonstrate that the complement $2^{\overline{\Z}} \setminus \class{p=0}_\AIform$ is open once we recall that, in $\overline{\Z}$, the maximal ideals are precisely the non-zero prime ideals.  If $P$ fails any of the conditions to be a prime ideal containing $p$, it easy to find an open neighbourhood of $P$ that is contained entirely in the complement $2^{\overline{\Z}} \setminus \class{p = 0}_\AIform$.  As an example, if $P \subseteq \overline{\Z}$ contains the product $a \cdot b$ but neither $a$ nor $b$, i.e.\ $P$ is not \emph{prime}, then 
	\[
	\lrset{P \subseteq \overline{\Z}}{a \cdot b \in P} \cap \lrset{P \subseteq \overline{\Z}}{a \not\in P} \cap \lrset{P \subseteq \overline{\Z}}{ b \not  \in P} 
	\]
	is such an open neighbourhood of $P$.  Thus, $\class{p=0}_\AIform \subseteq 2^{\overline{\Z}}$ is closed, from which the result follows.
\end{proof}

Next, we deduce that the point $\overline{\Z} \in \AIform_0$ is a \emph{universal accumulation point} of $\AIform_0$, by which we mean that $\overline{\Z}$ is an accumulation point of every subset $S \subseteq \AIform_0 \setminus \{\,\overline{\Z}\,\}$, or rather -- the only open containing $\overline{\Z}$ is the whole space.  This is because if $\overline{\Z}$ is contained in a basic open $\class{q(\vec{a}) = 0}_\AIform$, i.e.\ $\Z \vDash q(\vec{a}) = 0$, then $\overline{\Z/ \langle p \rangle } \vDash q(\pi_M(\vec{a}))$ for each maximal ideal $M$ of $\overline{\Z}$ containing $p$.  Hence, each copy of $\overline{\Z / \langle p \rangle} \in \AIform_0$ is also contained in the open $\class{q(\vec{a}) = 0}_\AIform$, and thus $\class{q(\vec{a}) = 0}_\AIform = \AIform_0$.  As a consequence, we deduce the following.

\begin{coro}\label{coro:algint_is_sent_complete}
	The ring of algebraic integers $\overline{\Z}$ is a `sentence-complete' model of $\AItheory$.
\end{coro}

Thus, as expressed below, we are able to give a description of the space $\AIform_0$ devoid of any mention of algebraic structures from which it derives.

\begin{coro}
	The space of objects $\AIform_0^\taulogo$ is obtained by the addition of a universal accumulation point to the coproduct of countably many copies of the Cantor space.
\end{coro}

%%%%%%%%%%%%%%%%%%%%%%%%%%%%%%%%

%%%%%%%%%%%%%%%%%%%%%%%%%%%%%%%%%%%%%%%%%%%%%%%%%%%%%%%%%%%%

	%%%%%%%%%%%%%%%%%%%%%%%%%%%%%%%%%%%%%%%%%%%%%%%%%%%%%%%%%%%%%%%%%%%%%%%%%%%%%%%%%%%%%%%%%%
%%%%%%%%%%%%%%%%%%%%%%%%%%%%%%%%%%%%%%%%%%%%%%%%%%%%%%%%%%%%%%%%%%%%%%%%%%%%%%%%%%%%%%%%%%

\section*{Acknowledgements}

This work was produced while a PhD student at the University of Insubria.  I am grateful to my PhD supervisor, Olivia Caramello, for her suggestion of the topic and her valuable support along the way.  I also acknowledge the financial support I received from the Insubria-Huawei studentship for the project ``Grothendieck toposes for information and computation''.

%%%%%%%%%%%%%%%%%%%%%%%%%%%%%%%%%%%%%%%%%%

%\printbibliography

\end{document}